\pdfoutput=1
\documentclass[11pt,a4paper]{article}

\usepackage[T1]{fontenc}
\usepackage[utf8]{inputenc}
\usepackage{amsmath,amssymb,mathtools}
\usepackage{amsthm}
\usepackage{txfonts}
\usepackage{array}
\usepackage[margin=2.4cm]{geometry}
\usepackage{tikz}
\usetikzlibrary{cd}
\usetikzlibrary{decorations.pathmorphing}
\usepackage{xcolor}
\usepackage[colorlinks=true,linkcolor=blue!50!black,citecolor=blue!50!black,urlcolor=blue!50!black]{hyperref}
\hypersetup{pdftitle={Statistical models as natural transformations: meaningfulness, coherence and priors as states in Markov categories},
  pdfauthor={Francesco Vaccarino}}

\theoremstyle{plain}
\newtheorem{theorem}{Theorem}[section]
\newtheorem{proposition}[theorem]{Proposition}
\newtheorem{lemma}[theorem]{Lemma}
\newtheorem{corollary}[theorem]{Corollary}

\theoremstyle{definition}
\newtheorem{definition}[theorem]{Definition}
\newtheorem{example}[theorem]{Example}

\theoremstyle{remark}
\newtheorem{remark}[theorem]{Remark}
\newtheorem{caveat}[theorem]{Remark}

\newcommand{\cat}[1]{\mathsf{#1}}
\newcommand{\catU}{\cat{U}}
\newcommand{\catV}{\cat{V}}
\newcommand{\catO}{\cat{\Omega}}
\newcommand{\catS}{\cat{S}}
\newcommand{\catD}{\cat{D}}
\newcommand{\catK}{\cat{K}}
\newcommand{\Meas}{\cat{Meas}}
\newcommand{\Set}{\cat{Set}}
\newcommand{\Stoch}{\cat{Stoch}}
\newcommand{\FinStoch}{\cat{FinStoch}}
\newcommand{\FinSet}{\cat{FinSet}}
\newcommand{\Alg}{\cat{Alg}}
\newcommand{\G}{\mathcal{G}}
\newcommand{\Prb}{\mathcal{P}}
\newcommand{\op}{\mathrm{op}}
\newcommand{\id}{\mathrm{id}}
\newcommand{\pr}{\mathrm{pr}}
\newcommand{\del}{\delta}
\newcommand{\R}{\mathbb{R}}

\newcommand{\Vect}{\cat{Vect}}
\newcommand{\Norm}{\mathcal{N}}
\newcommand{\Bern}{\mathrm{Bern}}
\newcommand{\expit}{\mathrm{expit}}
\newcommand{\res}{\mathrm{res}}
\DeclareMathOperator{\ob}{ob}
\DeclareMathOperator{\Nat}{Nat}
\DeclareMathOperator{\Hom}{Hom}
\newcommand{\StatMod}{\cat{StatMod}}
\newcommand{\Statk}{\StatMod_{\Meas}}
\newcommand{\Jc}{\Delta}
\newcommand{\Jcm}{\Delta_{\Meas}}
\newcommand{\Arrdet}[1]{#1^{\rightarrow}_{\det}}
\newcommand{\Arrcop}[1]{#1^{\rightarrow}_{\det,\,\mathrm{cop}}}
\newcommand{\BorelStoch}{\cat{BorelStoch}}

\title{Statistical models as natural transformations:\\
meaningfulness, coherence and priors as states in Markov categories}
\author{Francesco Vaccarino\footnote{francesco.vaccarino@polito.it; ORCID 0000-0002-0610-9168.}\\
\small Dipartimento di Scienze Matematiche ``G.\ L.\ Lagrange'', Politecnico di Torino}
\date{December 2026}

\begin{document}
\maketitle

\begin{abstract}
\noindent
A statistical model in the sense of McCullagh and Br\o ns is a natural
transformation between two functors from the category
of designs to the Kleisli category $\Stoch$ of the Giry monad, provided
its components are measurable in the parameter. A design-indexed quantity is
a family of morphisms of $\Stoch$ on the parameter objects,
meaningful if natural. Tjur's criterion, imposed on
parameter functions indexed by finite samples with multiplicities, forces the
indexing by the support and then coincides with naturality over the
insertions. For finite designs, a quantity
can be corrected to a natural one within a class of corrections exactly when
a class vanishes in the first cohomology group of a Baues--Wirsching
complex relative to that class, while its image in the absolute group
always vanishes. In the one-way layout the marginal dispersion is not
meaningful, and the within-group dispersion is its unique correction leaving
the merged design unchanged. A prior is a family of states on the parameter
objects, coherent over a class of design morphisms if natural
over it. Coherence at a merge confines the prior
to the image of the parameter map, coherence over the
insertions is Kolmogorov consistency, and coherence over the injections
adds the exchangeability of the categorical de~Finetti theorem. The
analogue of Jeffreys' general rule is not coherent, while that of his rule
for location--scale families is. Ridge regression is the Bayesian inversion
of the Gaussian linear model with
respect to a Gaussian prior, coherent over the insertions and never
over the injections.
\end{abstract}


\setcounter{tocdepth}{2}
\tableofcontents

\section{Introduction}\label{sec:intro}

Let $S$ be a measurable space and let $\Prb(S)$ be the set of probability
measures on $S$. A statistical model is usually defined as a map
$\Theta_0 \to \Prb(S)$, where $\Theta_0$ is a set, i.e.\ as a parametrised
family of probability measures on a fixed sample space. McCullagh observed
in~\cite{mccullagh_what_2002} that this definition does not say how the family
must transform when the experimental design changes. Therefore it cannot state
the requirement that a parameter retain its meaning across designs. His remedy
was categorical. Designs form a category, sample spaces and parameter spaces
are functors on it, and a model is a natural transformation.

In his discussion of~\cite{mccullagh_what_2002},
Br\o ns~\cite[pp.~1280--1282]{brons_discussion_2002} organised these data into
a hexagonal diagram of categories and functors. He observed that the resulting
notion of model is equivalent, by a general argument on comma categories, to a
diagram in a fixed ``category of statistical models'' (see
Proposition~\ref{prop:brons}). He also noted that this category is too small,
since to include all statistical transformations ``Markov kernels must be
allowed''~\cite[p.~1282]{brons_discussion_2002}. This remark was the starting
point of the present work.

In the same volume, Tjur~\cite{tjur_discussion_2002} considered the rules which
assign a parameter, i.e.\ a function of the unknown distribution, to each
finite sample of covariate values, repetitions allowed. He called such a rule
a \emph{parameter function}, and he called it \emph{meaningful} when it is
invariant under the formation of marginal models, i.e.\ when reducing the
design by removing some of the sampled $x$'s leaves the associated parameter
unchanged. His non-example is the design-averaged expectation
$\alpha + \beta\bar x$. The criterion of
McCullagh~\cite[\S4.5]{mccullagh_what_2002} is that a subparameter be a natural
transformation of functors on the design category, and in his rejoinder he
reads Tjur's remarks as explaining what is meant by a natural parameter in
regression~\cite[p.~1304]{mccullagh_what_2002}. The two criteria are stated
over different classes of morphisms, namely design reduction alone for Tjur
and the full design category for McCullagh, and neither source asserts that
they coincide.

In~\cite{mccullagh_what_2002} the codomain of the construction is the category
$\Set$ of sets, where conditioning, sufficiency and almost sure equality cannot
be expressed. 
Markov
categories~\cite{fritz_synthetic_2020,cho_disintegration_2019,perrone_markov_2024,fritz_definetti_2021}
provide a categorical probability theory in which they can. The aim of this
paper is to replace $\Set$ with the Markov category $\Stoch$, i.e.\ the Kleisli
category of the Giry monad, whose objects are the measurable spaces and whose
morphisms are the Markov kernels, and to study the resulting notion of model.

Let $\catD$ be the category of designs and let $\catV$ be the category of
response scales (see Definition~\ref{def:designdata}). A statistical model in
the sense of McCullagh and Br\o ns is a natural transformation
$P \colon L \Rightarrow R$ between two functors
$L, R \colon \catV \times \catD^{\op} \to \Set$, where $L$ gives the parameter
sets and $R$ gives the sets of the probability measures on the sample spaces
(see Definition~\ref{def:mcbmodel}). We denote by
$\del \colon \Meas \to \Stoch$ the functor which is the identity on the objects
and maps a measurable map $f$ to the kernel $x \mapsto \delta_{f(x)}$. Suppose
that the parameter sets are endowed with measurable structures such that the
reparametrisation maps are measurable. Composing with $\del$ we obtain two
functors $\mathbf{L}, \mathbf{R} \colon \catV \times \catD^{\op} \to \Stoch$.
We shall prove the following.

\begin{theorem}[Bridge]\label{thm:main}
There is a canonical bijection between the statistical models
$P \colon L \Rightarrow R$ whose components are measurable in the parameter
and the natural transformations
$\mathbf{P} \colon \mathbf{L} \Rightarrow \mathbf{R}$. The morphisms
$\mathbf{L}(m)$ and $\mathbf{R}(m)$ are deterministic for every morphism $m$ of
$\catV \times \catD^{\op}$. If the design data are finite, then every
statistical model has measurable components, and $\Stoch$ can be replaced by
its full subcategory $\FinStoch$ of the finite discrete spaces, whose
morphisms are the stochastic matrices.
\end{theorem}

\begin{remark}\label{rem:main}
Theorem~\ref{thm:main} is proved in Section~\ref{sec:bridge}, see
Theorem~\ref{thm:bridge-finite} and Theorem~\ref{thm:bridge-meas}. The proof
is a transposition under the bijection
$\Stoch(X,Y) \cong \Meas(X,\G Y)$. The measurability of the components is a
restriction with respect to~\cite{mccullagh_what_2002}, where no measurability
in the parameter is imposed. The role of Theorem~\ref{thm:main} is to place
models, quantities and priors in the same category, where the results below
are stated. For the statistician this means that marginalisation,
conditioning, sufficiency and Bayesian inversion become operations of the same
category in which the model is a morphism.
\end{remark}

We call $\mathbf{P}$ the \emph{bridged model} associated to $P$. Our main
results concern two kinds of families of morphisms of $\Stoch$ indexed by the
covariate spaces, which are the objects of a category $\catO$.

Let $\Theta_\Meas(\omega)$ be the parameter object at the covariate space
$\omega$ and let $\Lambda \colon \catO^{\op} \to \Meas$ be a functor. A
\emph{design-indexed quantity} is a family of morphisms
$g_{\omega} \colon \Theta_\Meas(\omega) \rightsquigarrow \Lambda(\omega)$ of
$\Stoch$. It is called \emph{coherent} over a subcategory $\cat{Q}$ of $\catO$
if it is natural over $\cat{Q}$ (Definition~\ref{def:coherence}),
\emph{meaningful} if it is coherent over $\catO$, and \emph{Tjur-natural} if it
is coherent over the subcategory $\catO_{\mathrm{red}}$ generated by the proper
inclusions of covariate spaces, which we call \emph{insertions} after Tjur.
The parameter functions of Tjur are indexed by the finite samples of covariate
values, with multiplicities, and not by the covariate spaces. Accordingly, a
\emph{sample-indexed quantity} is a family of morphisms
$G_{d} \colon \Theta_\Meas(\omega) \rightsquigarrow \Lambda(\omega)$ indexed
by the designs $d = (u,\omega,x)$ with $x$ surjective. A \emph{Tjur reduction}
of $d$ is the morphism of designs given by the removal of some units together
with the covariate values which are no longer observed, and $G$ is called
\emph{Tjur-meaningful} if~\eqref{eq:sampletjur} holds at every Tjur reduction
(Definition~\ref{def:sampleindexed}).

\begin{theorem}[Tjur's criterion]\label{thm:main-tjur}
Let $\catO = \catO_{\mathrm{lab}}$ be the category of the finite non-empty
sets and let $\catU = \FinSet_{\mathrm{inj}}$ be the category of the finite
sets and injective maps, as in the one-way layout, with an arbitrary parameter
functor. A sample-indexed quantity $G$ is
Tjur-meaningful if and only if
$G_{(u,\omega,x)} = g_{\omega}$ depends only on the support $\omega$ of the
sample and the design-indexed quantity $g$ is Tjur-natural. In this case
$g_{\omega} = G_{(\omega,\omega,\id_{\omega})}$. Every meaningful quantity is
Tjur-natural, and the converse does not hold in general.
\end{theorem}

\begin{remark}\label{rem:main-tjur}
Theorem~\ref{thm:main-tjur} is proved in Section~\ref{sec:meaningful}, see
Proposition~\ref{prop:descent} for the general hypotheses. Thus Tjur's
criterion first forces the indexing by the support and then coincides with
naturality over the insertions, which is weaker than the criterion of
McCullagh. The non-examples of Tjur are sample-indexed and not
support-indexed. The quantity of Example~\ref{ex:tjurgap} is natural at a
morphism $\psi$ if and only if $\psi$ is injective. The regression examples,
whose covariate spaces are vector spaces, are not covered by
Theorem~\ref{thm:main-tjur}.
\end{remark}

Suppose now that $\catO$ is a finite category and that the sets
$\Theta(\omega)$ and $\Lambda(\omega)$ are finite. Let $k \supseteq \mathbb{Q}$
be a field and let $A(\omega) = k^{\Theta(\omega)}$ and
$B(\omega) = k^{\Lambda(\omega)}$ be the algebras of functions. Let
$C^{\bullet}$ be the complex of Baues and
Wirsching~\cite{baues_cohomology_1985} of $\catO$ with coefficients in
$\Hom_k(B,A)$, so that $H^{0} = \Nat(B,A)$. A deterministic quantity $g$ gives
the $0$-cochain $g^{*} = (g_{\omega}^{*} \colon B(\omega) \to A(\omega))$, and
$g$ is meaningful if and only if its \emph{defect} $c_g = d^{0}g^{*}$ is zero.
For a linear subspace $\mathcal{C} \subseteq C^{0}$ we denote by
$C^{\bullet}_{\mathcal{C}}$ the subcomplex of $C^{\bullet}$ with
$C^{0}_{\mathcal{C}} = \mathcal{C}$ and $C^{n}_{\mathcal{C}} = C^{n}$ for
$n \ge 1$.

\begin{theorem}[Obstruction]\label{thm:main-obstruction}
There is $h \in \mathcal{C}$ such that $g^{*} + h$ is natural if and only if
the class of $c_g$ in $H^{1}(C^{\bullet}_{\mathcal{C}})$ is zero, and $h$ is
unique modulo $\mathcal{C} \cap \Nat(B,A)$. The image of this class in
$H^{1}(\catO;\Hom_k(B,A))$ is zero for every $g$. In the one-way layout with
two groups and finitely many values of the mean and of the dispersion, let
$\mathcal{C}_{1}$ be the subspace of the $0$-cochains which vanish at the
merged design. Then $\mathcal{C}_{1} \cap \Nat(B,A) = 0$, and the marginal
dispersion has a unique correction in $\mathcal{C}_{1}$, which gives the
within-group dispersion.
\end{theorem}

\begin{remark}\label{rem:main-obstruction}
Theorem~\ref{thm:main-obstruction} is proved in Section~\ref{sec:hochschild},
see Propositions~\ref{wprop:obstruction}, \ref{prop:relative}
and~\ref{prop:canonical}. Thus the absolute group $H^{1}$ is an invariant of
the design scheme and not an obstruction group for quantities. We compute it
in Example~\ref{ex:oneway-h1}, where for a constant target it is given by
copies of the first derived limit of the parameter diagram. When the mean takes
at least two values, the marginal dispersion is not meaningful, and it has no
injective natural recalibration (Propositions~\ref{prop:recalib}
and~\ref{prop:oneway-cont-recalib}).
\end{remark}

A \emph{prior} for a bridged model is a family of states
$\pi_{\omega} \colon I \rightsquigarrow \Theta_\Meas(\omega)$, i.e.\ of
probability measures on the parameter objects, and it is coherent over
$\cat{Q}$ when $\pi_{\omega} = \Theta_\Meas(\psi)_{*}\,\pi_{\omega'}$ for all
$\psi \colon \omega \to \omega'$ in $\cat{Q}$ (Definition~\ref{def:prior}).
Thus quantities transform contravariantly with the parameter, and priors
covariantly. Let $\catO_{\mathrm{inj}}$ be the subcategory of $\catO$
generated by the insertions and the isomorphisms.

\begin{theorem}[Priors]\label{thm:main-priors}
Let $\pi$ be a prior for a bridged model.
\begin{enumerate}
\item If $\pi$ is coherent at $\psi \colon \omega \to \omega'$, then
$\pi_{\omega}$ vanishes on the measurable sets disjoint from the image of
$\Theta_\Meas(\psi)$.
\item If $\pi$ is coherent over $\catO_{\mathrm{inj}}$, then $\pi_{\omega}$ is
invariant under $\Theta_\Meas(\mathrm{Aut}\,\omega)$ for every $\omega$.
Coherence over $\catO_{\mathrm{red}}$ imposes no such symmetry.
\item Let $\mathbf{P}$ be the Gaussian linear model with known variance
$\sigma_{0}^{2}$, and let
$\pi^{\tau}_{\omega} = \Norm(0,\tau^{2}I_{\omega}) \otimes
\delta_{\sigma_{0}^{2}}$ be the prior associated to a choice of inner products
on the covariate spaces which restrict along the insertions, where
$I_{\omega}$ is the induced form on $\omega^{*}$. Then the Bayesian
inversion of $\mathbf{P}_{d}$ with respect to $\pi^{\tau}_{\omega}$ exists in
$\BorelStoch$, and it is the Gaussian kernel whose mean is the ridge estimator
with penalty $\lambda = \sigma_{0}^{2}/\tau^{2}$. The prior $\pi^{\tau}$ is
coherent over $\catO_{\mathrm{red}}$, and it is coherent at a linear
automorphism $A$ if and only if $A$ is orthogonal. In particular, it is not
coherent over $\catO_{\mathrm{inj}}$.
\end{enumerate}
\end{theorem}

\begin{remark}\label{rem:main-priors}
Theorem~\ref{thm:main-priors} is proved in Sections~\ref{sec:bayes}
and~\ref{sec:ridge}, see Propositions~\ref{prop:coherence}, \ref{prop:ridge}
and~\ref{prop:ridgecoherence}. By part~(1), coherence at a merge of covariate
values confines the prior to the image of the corresponding parameter map,
which is a diagonal in the one-way layout. For this reason naturality of
priors is required over $\catO_{\mathrm{red}}$ or $\catO_{\mathrm{inj}}$ and
not over $\catO$. Coherence over $\catO_{\mathrm{red}}$ is Kolmogorov consistency, and
the invariance of part~(2) is the exchangeability assumed by the categorical
de~Finetti theorem of Fritz, Gonda and
Perrone~\cite{fritz_definetti_2021} (Remark~\ref{rem:definetti}). Part~(3)
corresponds to the usual instruction to standardise the covariates. In the
finite one-way scheme, the coherent priors form polytopes whose dimensions we
compute (Proposition~\ref{prop:priordim}), and the analogue of Jeffreys'
general rule is not coherent, while the analogue of his rule for
location--scale families is (Proposition~\ref{wprop:jeffreys}). Improper
priors and the marginalisation paradoxes of Dawid, Stone and
Zidek~\cite{dawid_stone_zidek_1973}, which cannot arise from proper priors,
require a codomain with unnormalised states, and they are left open
(Remark~\ref{rem:open}).
\end{remark}

The closest work is the thesis of Patterson~\cite{patterson_algebra_2020}, an
algebraic theory of statistical models with machine representation in view,
where the signature of a model is primitive, while here the structure indexed
by the designs is. Fritz~\cite[\S7]{fritz_synthetic_2020} shows that the
functors from a small category $\catD$ to the deterministic morphisms of a
Markov category $\cat{C}$, with all the natural transformations as morphisms
and the pointwise monoidal structure, form a Markov category
$\mathrm{Fun}(\catD,\cat{C})$~\cite[Prop.~7.1]{fritz_synthetic_2020}; he
applies it to stochastic processes, with $\catD = (\mathbb{Z},\ge)$, and asks
whether it has conditionals when $\cat{C}$
does~\cite[Problem~11.9]{fritz_synthetic_2020}. Theorem~\ref{thm:main} says
that a statistical model is a morphism of
$\mathrm{Fun}(\catV\times\catD^{\op},\Stoch)$ (Remark~\ref{rem:meas-hyp}), so
that the naturality of McCullagh and Br\o ns is the membership in a hom-set of
this category. That the models of McCullagh should fit into a category of
this kind is suggested by
Fritz~\cite[Def.~14.1, footnote~29]{fritz_synthetic_2020}, who leaves the
details open. What is specific to statistics is the indexing by the designs,
with the parameter functor as domain, and the questions it raises for the
quantities and the priors of Sections~\ref{sec:meaningful}
and~\ref{sec:bayes}: Fritz's synthetic theorems on
sufficiency~\cite[Def.~14.1]{fritz_synthetic_2020} concern one hom-set, and
his Remark~7.3 notes that a theory of processes requires structure on $\catD$
beyond the diagram category.
Conditionals, almost sure equality and Bayesian inversion are taken
from~\cite{fritz_synthetic_2020} and from Cho and
Jacobs~\cite{cho_disintegration_2019} (see Section~\ref{sec:ambient}). The
Kolmogorov products of Rischel and Fritz~\cite{fritz_rischel_infinite_2020}
and the categorical de~Finetti theorem are used only in
Remark~\ref{rem:definetti}, and we do not apply them to the designs of
Definition~\ref{def:designdata}. Information
cohomology~\cite{baudot_homological_2015,vigneaux_information_2021} has
coefficients and indexing categories different from those of
Section~\ref{sec:hochschild}, and we do not construct a comparison map.
Jeffreys' general rule~\cite{jeffreys_invariant_1946}, his recommendation for
location--scale families and his objection to the general rule in the problem
of several means~\cite{jeffreys_theory_1961} are used as they are presented by
Kass and Wasserman~\cite{kass_selection_1996} and by Consonni, Fouskakis, Liseo
and Ntzoufras~\cite{consonni_prior_2018}. Proposition~\ref{wprop:jeffreys}
agrees with the preference for reference
priors~\cite{bernardo_reference_1979,berger_formal_2009} over the
multidimensional Jeffreys rule, but we do not claim that coherence
characterises them. Ridge regression is due to Hoerl and
Kennard~\cite{hoerl_ridge_1970}, and its reading as a Gaussian posterior mean
is classical. What is new in Theorem~\ref{thm:main-priors}(3) is its
description as a Bayesian inversion indexed by the designs, together with the
coherence of its prior.

The paper goes as follows. In Section~\ref{sec:mcb} we recall the framework of
McCullagh and Br\o ns, with a complete proof of the equivalence stated by
Br\o ns (Proposition~\ref{prop:brons}), and in Section~\ref{sec:markov} the
notions on Markov categories that we need.
Sections~\ref{sec:bridge}, \ref{sec:meaningful} and~\ref{sec:hochschild}
concern models and quantities, and we refer to their setting as the
\emph{Fisherian shape}. Sections~\ref{sec:ambient}, \ref{sec:bayes}
and~\ref{sec:ridge} concern priors, and we refer to their setting as the
\emph{Bayesian shape}. In Appendix~\ref{sec:Bref} we identify the category of
statistical models of Br\o ns, on the objects which satisfy two regularity
hypotheses, with the arrows of $\Stoch$ whose domain is a copower of the
monoidal unit (Proposition~\ref{wprop:copower}), and we show that it is a
coreflective subcategory of its measurable analogue
(Proposition~\ref{prop:coreflection}). The rest of the paper does not depend
on the appendix. The companion paper~\cite{vaccarino_monads_2026} carries
these constructions from the Giry monad to an arbitrary monad, and in
particular to the monads of distributions with values in a semiring.
\section{The McCullagh--Br\o ns framework}\label{sec:mcb}

We recall the definition of a statistical model given by McCullagh
in~\cite{mccullagh_what_2002}, in the form given by Br\o ns
in~\cite{brons_discussion_2002}. All the categories considered in this section
are locally small.

\subsection{Statistical models}

\begin{definition}\label{def:designdata}
The \emph{design data} consist of the following.
\begin{enumerate}
\item A category $\catU$, whose objects are called \emph{statistical units}.
\item A category $\catO$, whose objects are called \emph{covariate spaces}.
\item A category $\catV$, whose objects are called \emph{response scales}.
\item The category $\catS \coloneqq \catV \times \catU^{\op}$ of \emph{sample
space indices} and a functor $\Gamma \colon \catS \to \Meas$. The measurable
space $\Gamma(v,u)$ is called the \emph{sample space} of the $v$-valued
responses on the units $u$.
\item Two functors $J_{\catU} \colon \catU \to \cat{C}$ and
$J_{\catO} \colon \catO \to \cat{C}$ into a category $\cat{C}$. The comma
category
\[
\catD \;\coloneqq\; (J_{\catU} \downarrow J_{\catO})
\]
is called the category of \emph{designs}. Its objects are the triples
$(u,\omega,x)$, where $x \colon J_\catU u \to J_\catO \omega$ is a morphism of
$\cat{C}$ called the \emph{design map}. A morphism
$(u,\omega,x) \to (u',\omega',x')$ is a pair $(\varphi,\psi)$, where
$\varphi \colon u \to u'$ is a morphism of $\catU$ and
$\psi \colon \omega \to \omega'$ is a morphism of $\catO$ such that
$x' \circ J_{\catU}\varphi = J_{\catO}\psi \circ x$. We denote by
$\pr_{\catU} \colon \catD \to \catU$ and $\pr_{\catO} \colon \catD \to \catO$
the projections.
\item A category $\catK$ of \emph{parameter objects} with a faithful functor
$U_{\catK} \colon \catK \to \Set$, and a functor
$\Theta \colon \catO^{\op} \to \catK$ called the \emph{parameter functor}.
\end{enumerate}
\end{definition}

In the examples, the objects of $\catU$ are finite sets of experimental units
and its morphisms are injective maps, i.e.\ inclusions of subsamples. The
morphisms of $\catV$ are measurable maps between response scales, such as
coarsenings of the response or affine rescalings (see
Remark~\ref{cav:scaleV} for the latter). Furthermore $\cat{C} = \Set$ and
$J_{\catU}$, $J_{\catO}$ are the underlying-set functors, so that a design map
assigns a covariate value to each unit, and $\Gamma(v,u) = v^{u}$ with the
product structure.

Let $\Prb \colon \Meas \to \Set$ be the functor which maps a measurable space
to the set of its probability measures and a measurable map to the
pushforward. It is the composition of the Giry monad
$\G \colon \Meas \to \Meas$ (see~\cite{giry_categorical_1982} and
Section~\ref{sec:markov}) with the forgetful functor. Note that a map
$\Theta_0 \to \Prb(S)$ is a family of probability measures on $S$ indexed by
the set $\Theta_0$, and no measurability in the index is required. This is the
only difference between a model and a family of morphisms of $\Stoch$. It
disappears when the parameter object is discrete, and it is an additional
hypothesis in general (see Theorem~\ref{thm:bridge-meas}).

\begin{definition}\label{def:mcbmodel}
Let the design data be as in Definition~\ref{def:designdata} and consider the
functors $L, R \colon \catV \times \catD^{\op} \to \Set$ given by
\[
L \;\coloneqq\; U_{\catK} \circ \Theta \circ \pr_{\catO}^{\op} \circ \pr_{2}
\qquad\text{and}\qquad
R \;\coloneqq\; \Prb \circ \Gamma \circ \bigl(\id_{\catV} \times \pr_{\catU}^{\op}\bigr),
\]
where $\pr_2 \colon \catV \times \catD^{\op} \to \catD^{\op}$ is the
projection. A \emph{statistical model} in the sense of McCullagh and Br\o ns,
or \emph{McCullagh--Br\o ns model}, is a natural transformation
\[
P \colon L \Longrightarrow R.
\]
\end{definition}

Thus a statistical model is given by a map
\[
P_{v,d} \colon \; U_\catK\Theta(\omega) \longrightarrow \Prb\bigl(\Gamma(v,u)\bigr),
\qquad \theta \longmapsto P_{v,d}(\,\cdot \mid \theta),
\]
for every response scale $v$ and every design $d = (u,\omega,x)$, such that the
square corresponding to $m$ commutes for every morphism $m$ of
$\catV \times \catD^{\op}$ (see~\eqref{eq:ND} and~\eqref{eq:NV} below). The
categories and the functors introduced so far fit into the following hexagonal
diagram, due to Br\o ns, and a statistical model is a natural transformation
between the two compositions from $\catV \times \catD^{\op}$ to $\Set$.
\[
\begin{tikzcd}[column sep=1.6em, row sep=1.9em]
& \catV \times \catD^{\op}
\arrow[dl, "\pr_2"'] \arrow[dr, "\id_\catV \times \pr_\catU^{\op}"] & \\
\catD^{\op} \arrow[d, "\pr_\catO^{\op}"'] & & \catS \arrow[d, "\Gamma"] \\
\catO^{\op} \arrow[d, "\Theta"'] \arrow[rr, phantom, "\overset{P}{\Longrightarrow}" description] & & \Meas \arrow[d, "\Prb"] \\
\catK \arrow[dr, "U_\catK"']
& & \Set \arrow[dl, "\id_{\Set}"]\\
&\Set
\end{tikzcd}
\]

\begin{remark}\label{rem:attribution}
The product categories $\catS = \catV \times \catU^{\op}$ and
$\catV \times \catD^{\op}$ are due to
McCullagh~\cite[\S4.7, p.~1240]{mccullagh_what_2002} and are restated by
Br\o ns in his own notation~\cite[p.~1280]{brons_discussion_2002}.
In~\cite[\S4.2]{mccullagh_what_2002} the response scale is fixed, and for fixed
$v$ the naturality squares of Definition~\ref{def:mcbmodel} are the squares
of~\cite[p.~1236]{mccullagh_what_2002}. McCullagh takes the morphisms of
$\catU$ and $\catO$ to be injective~\cite[\S4.1]{mccullagh_what_2002}. Br\o ns
considers this restriction dispensable and does not impose
it~\cite[p.~1281]{brons_discussion_2002}. In
Definition~\ref{def:designdata} we follow Br\o ns for $\catO$, and this choice
fixes the class of morphisms over which naturality is quantified in
Section~\ref{sec:meaningful}. In all our examples
$\catU = \FinSet_{\mathrm{inj}}$, and the injectivity of the maps between units
is used in the verifications of Section~\ref{sec:worked}.
\end{remark}

\subsection{Br\o ns' comma-category equivalence}

Let $F \colon \cat{B} \to \cat{E}$ be a functor. Recall that the objects of the
comma category $(\id_{\cat E} \downarrow F)$ are the triples $(e, b, p)$, where
$p \colon e \to F b$ is a morphism of $\cat{E}$.

\begin{definition}\label{def:statmod}
The \emph{category of statistical models} is the comma category
\[
\StatMod \;\coloneqq\; \bigl(\id_{\Set} \downarrow \Prb\,\bigr).
\]
Its objects are the triples $(\Theta_0,S,p)$, where $\Theta_0$ is a set, $S$ is
a measurable space and $p \colon \Theta_0 \to \Prb(S)$ is a map, i.e.\ a
parametrised family of probability measures on $S$. A morphism
$(\Theta_0,S,p) \to (\Theta_0',S',p')$ is a pair $(r,f)$, where
$r \colon \Theta_0 \to \Theta_0'$ is a map and $f \colon S \to S'$ is a
measurable map such that $\Prb(f) \circ p = p' \circ r$.
\end{definition}

We have the following result.

\begin{proposition}[Br\o ns]\label{prop:brons}
There is a bijection, natural in all the data, between
\begin{enumerate}
\item the statistical models $P \colon L \Rightarrow R$ of
Definition~\ref{def:mcbmodel};
\item the functors $\widehat{P} \colon \catV \times \catD^{\op} \to \StatMod$
whose compositions with the two projections of the comma category are
$L \colon \catV \times \catD^{\op} \to \Set$ and
$\Gamma \circ (\id_{\catV} \times \pr_\catU^{\op}) \colon
\catV \times \catD^{\op} \to \Meas$.
\end{enumerate}
\end{proposition}

\begin{proof}
Let $G \colon \cat{A} \to \cat{E}$, $H \colon \cat{A} \to \cat{B}$ and
$F \colon \cat{B} \to \cat{E}$ be functors. By the universal property of the
comma category, the functors $\cat{A} \to (\id_{\cat E} \downarrow F)$ whose
compositions with the two projections are $G$ and $H$ correspond bijectively
to the natural transformations $G \Rightarrow F H$. Indeed, an assignment
$a \mapsto (G a, H a, P_a)$ on the objects is functorial if and only if the
squares of $P$ commute, and the values on the morphisms are forced. The
statement follows by taking $\cat{A} = \catV \times \catD^{\op}$,
$\cat{E} = \Set$, $\cat{B} = \Meas$, $G = L$,
$H = \Gamma \circ (\id_{\catV} \times \pr_\catU^{\op})$ and $F = \Prb$. The
naturality in the data is a straightforward verification.
\end{proof}

\begin{remark}\label{rem:brons-sketch}
In~\cite[pp.~1280,~1282]{brons_discussion_2002} the result is stated with a
sketch of the proof.
\end{remark}

\subsection{Examples}\label{sec:worked}

We give three examples of the design data of Definition~\ref{def:designdata},
namely the one-way layout, multiple linear regression and binary
classification, and we verify that they give statistical models. We also give
two families which are not statistical models. We denote by
$\FinSet_{\mathrm{inj}}$ the category of finite sets and injective maps, and by
$\Vect_{\R}^{\mathrm{fd}}$ the category of finite dimensional real vector
spaces and linear maps.

In all the examples $\cat{C} = \Set$ and $J_{\catU}$, $J_{\catO}$ are the
underlying-set functors. Hence an object of $\catD$ is a triple $(u,\omega,x)$
with $x \colon u \to \omega$ a map, and a morphism
$(u,\omega,x)\to(u',\omega',x')$ is a pair $(\varphi,\psi)$ such that
\begin{equation}\label{eq:commacond}
x' \circ \varphi \;=\; \psi \circ x .
\end{equation}
In words, the covariate value of a unit of the smaller design is carried by
$\psi$ to the covariate value of the same unit in the larger design. This is
the compatibility needed to compare a model across two designs.
Let $v \in \ob\catV$ and let
$f = (\varphi,\psi) \colon d=(u,\omega,x) \to d'=(u',\omega',x')$ be a morphism
of $\catD$. The square of Definition~\ref{def:mcbmodel} corresponding to
$(\id_{v},f^{\op})$ is
\begin{equation}\label{eq:ND}
\begin{tikzcd}[column sep=5em, row sep=2.6em]
U_{\catK}\Theta(\omega') \arrow[r,"P_{v,d'}"] \arrow[d,"\Theta(\psi)"'] &
\Prb\bigl(\Gamma(v,u')\bigr) \arrow[d,"\Prb(\res_{\varphi})"] \\
U_{\catK}\Theta(\omega) \arrow[r,"P_{v,d}"'] &
\Prb\bigl(\Gamma(v,u)\bigr)
\end{tikzcd}
\end{equation}
where $\res_{\varphi} \colon v^{u'} \to v^{u}$, $y \mapsto y\circ\varphi$, is
the restriction along $\varphi$. It says that the marginal on the subsample of
the model at $d'$ is the model at $d$ evaluated at the pulled back parameter.
Let now $\alpha \colon v \to v'$ be a morphism of $\catV$ and let $d$ be fixed.
Since $L$ does not depend on $\catV$, the corresponding square reduces to
\begin{equation}\label{eq:NV}
\Prb(\alpha^{u}) \circ P_{v,d} \;=\; P_{v',d}
\qquad\text{on } U_{\catK}\Theta(\omega).
\end{equation}
We write (N-$\catU$), (N-$\catO$) and (N-$\catV$) for the conditions obtained
from~\eqref{eq:ND} and~\eqref{eq:NV} by letting only $\varphi$, only $\psi$ and
only $\alpha$ be different from the identity, respectively.

The verifications below use only~\eqref{eq:commacond}, the finiteness of $u$
and the injectivity of $\varphi$. The latter is necessary. Indeed, if $\varphi$
is not injective, then~\eqref{eq:ND} fails in the three models below, since two
units with the same image have equal responses and are not independent.

\subsubsection{The one-way layout}

This example will be used throughout the paper (see
Example~\ref{ex:oneway-cont}, Section~\ref{sec:meaningful} and
Example~\ref{ex:oneway}). A single factor with finitely many levels, called
\emph{treatments}, is applied to the experimental units, and each unit receives
exactly one treatment. There is no other explanatory variable. A
\emph{treatment label set} is a finite non-empty set $\omega$ whose elements
are the names of the treatments. It carries no further structure, and for this
reason the covariate spaces below are sets and not vector spaces as in
Example~\ref{ex:linreg}. The passage between the two descriptions is dummy
coding.

\begin{example}\label{ex:oneway-setup}
\leavevmode
\begin{enumerate}
\item \textbf{Units.} $\catU = \FinSet_{\mathrm{inj}}$. An object $u$ is the
set of the enrolled units, and a morphism
$\varphi \colon u \hookrightarrow u'$ is the inclusion of a subsample. We write
$n = |u|$.

\item \textbf{Covariate spaces.} $\catO = \catO_{\mathrm{lab}}$ is the full
subcategory of $\FinSet$ of the finite non-empty sets, i.e.\ of the treatment
label sets. Every morphism $\psi \colon \omega \to \omega'$ is a surjection
followed by an injection. We use the following terminology.
\begin{itemize}
\item If $\psi$ is surjective it is called a \emph{merging} of treatments. The
fibres $\psi^{-1}(b)$, with $b \in \omega'$, are the blocks of the treatments
which are identified.
\item If $\psi$ is injective and not bijective it is called a \emph{selection}
of a proper subfamily of treatments. If $\psi$ is a proper inclusion
$\omega' \subsetneq \omega$ we call it an \emph{insertion}, or design
reduction, after Tjur~\cite[p.~1299]{tjur_discussion_2002}. Every injection is
a bijection followed by an insertion. Tjur's criterion acts on the quantities
of Section~\ref{sec:meaningful} through the insertions only (see
Lemma~\ref{lem:tjurshadow}).
\item If $\psi$ is bijective it is called a \emph{relabelling} of the
treatments.
\item If $\psi$ is constant, it collapses all the treatments to a single one.
\end{itemize}
We shall also consider some subcategories of $\catO_{\mathrm{lab}}$, namely the
wide subcategory $\catO_{\mathrm{merge}}$ of the surjections and the
categories $\catO_{\mathrm{arr}} \subset \catO_{\mathrm{full}}$ with two
objects of Example~\ref{ex:oneway}. Naturality over a category implies
naturality over its subcategories, so that Lemma~\ref{lem:oneway-model} holds
for all of them.

\item \textbf{Response scales.} $\catV$ has the single object $\R$, with its
Borel $\sigma$-algebra, and only the identity morphism. Hence~\eqref{eq:NV} is
empty (see Remark~\ref{cav:scaleV} for the affine maps).

\item \textbf{Sample spaces.} $\catS = \catV \times \catU^{\op}$ and
$\Gamma(\R,u) = \R^{u}$ with the product $\sigma$-algebra, i.e.\ the space of
the response vectors $y = (y_i)_{i \in u}$. The morphism
$\varphi^{\op} \colon u' \to u$ of $\catU^{\op}$ is mapped by $\Gamma$ to the
restriction $\res_{\varphi} \colon \R^{u'} \to \R^{u}$,
$y \mapsto y \circ \varphi$.

\item \textbf{Designs.} An object of $\catD$ is a triple $(u,\omega,x)$, where
$x \colon u \to \omega$ is the \emph{treatment assignment}, i.e.\ $x(i)$ is the
treatment received by the unit $i$. Equivalently $x$ is the partition of $u$
into the \emph{groups}
\[
u_a \;\coloneqq\; x^{-1}(a) \subseteq u, \qquad a \in \omega,
\qquad n_a \coloneqq |u_a|, \qquad \textstyle\sum_{a \in \omega} n_a = n,
\]
some of which can be empty. The design is called \emph{balanced} if all the
$n_a$ are equal, and $x$ is surjective if and only if every treatment is
observed. For a morphism $(\varphi,\psi)$ condition~\eqref{eq:commacond} reads
\[
x'(\varphi(i)) \;=\; \psi(x(i)) \qquad (i \in u),
\]
i.e.\ each retained unit receives in the larger design the image under $\psi$
of the treatment it receives in the smaller one.

\item \textbf{Parameters.} The objects of $\catK$ are the products
$W \times \R_{>0}$ with $W \in \Vect_{\R}^{\mathrm{fd}}$, its morphisms are the
maps $A \times \id_{\R_{>0}}$ with $A$ linear, and $U_{\catK}$ is the
underlying-set functor. The parameter functor is
\[
\Theta(\omega) \;=\; \R^{\omega} \times \R_{>0},
\qquad
\Theta(\psi)(\mu',\sigma^{2}) \;=\; (\mu' \circ \psi,\ \sigma^{2}),
\]
i.e.\ $\bigl(\Theta(\psi)\mu'\bigr)_{a} = \mu'_{\psi(a)}$. The map
$\mu' \mapsto \mu' \circ \psi$ is linear, so that $\Theta(\psi)$ is a morphism of
$\catK$. A parameter value is
a pair $\theta = (\mu,\sigma^{2})$, where $\mu = (\mu_a)_{a \in \omega}$ is the
vector of the \emph{group means}, one for each treatment, and $\sigma^{2}$ is
the common error variance, which is not changed by $\Theta(\psi)$.
\end{enumerate}

\noindent For $d = (u,\omega,x)$ and
$\theta = (\mu,\sigma^{2}) \in \R^{\omega} \times \R_{>0}$ we set
\[
P_{\R,d}\bigl(\,\cdot \mid \mu,\sigma^{2}\bigr)
\;\coloneqq\; \bigotimes_{i \in u} \Norm(\mu_{x(i)},\,\sigma^{2})
\qquad \text{on } \R^{u},
\]
i.e.\ $Y_i = \mu_{x(i)} + \varepsilon_i$ with the $\varepsilon_i$ independent
and distributed as $\Norm(0,\sigma^{2})$. This is the Gaussian one-way layout,
where the group sizes $n_a$ are given by the design and no constraint such as
$\sum_a \mu_a = 0$ is imposed on $\mu$.
\end{example}

\begin{lemma}\label{lem:oneway-model}
The family $P$ of Example~\ref{ex:oneway-setup} is a statistical model.
\end{lemma}

\begin{proof}
Let $f = (\varphi,\psi) \colon d \to d'$ be a morphism of $\catD$ and let
$\theta' = (\mu',\sigma^{2}) \in \Theta(\omega')$. We have
$\Theta(\psi)\theta' = (\mu' \circ \psi,\sigma^{2})$, hence
\[
P_{\R,d}\bigl(\,\cdot \mid \Theta(\psi)\theta'\bigr)
\;=\; \bigotimes_{i \in u} \Norm(\mu'_{\psi(x(i))},\ \sigma^{2}).
\]
On the other hand
$P_{\R,d'}(\,\cdot \mid \theta') = \bigotimes_{j \in u'}
\Norm(\mu'_{x'(j)},\sigma^{2})$, and the pushforward along $\res_{\varphi}$ is
the marginalisation onto the coordinates in $\varphi(u) \subseteq u'$. Since
the factors are independent, it is equal to
$\bigotimes_{i \in u} \Norm(\mu'_{x'(\varphi(i))},\sigma^{2})$. The two
measures coincide by~\eqref{eq:commacond}, so that~\eqref{eq:ND} commutes.
Condition~\eqref{eq:NV} is empty because $\catV$ is discrete.
\end{proof}

The proof uses only the independence across the units and the fact that the
law of the unit $i$ depends on the design only through $x(i)$.

\begin{remark}\label{rem:oneway-slots}
For the one-way layout, the conditions (N-$\catU$), (N-$\catO$) and
(N-$\catV$) read as follows.
\begin{itemize}
\item[(N-$\catU$)] Let $\psi = \id_{\omega}$ and let $\varphi$ be a proper
inclusion. The marginal of the layout on the subsample $\varphi(u)$ is the
layout on that subsample, at the same $(\mu,\sigma^{2})$ and with the
assignment $x' \circ \varphi$. This is consistency under subsampling.
\item[(N-$\catO$)] Let $\varphi = \id_{u}$. If $\psi$ is a merging, then
$\mu' \circ \psi$ is constant on the blocks of $\psi$, and the square says
that the model at $\omega$ evaluated at such a mean vector is the merged model
at $\mu'$. Note that naturality does not say that the merged groups have equal
means. If $\psi$ is injective, the same square is read on a subfamily of
treatments. When $\psi$ is an inclusion, this is design reduction, and by
Lemma~\ref{lem:tjurshadow} it is the only part of (N-$\catO$) which is reached
by Tjur's criterion. A general injection adds a relabelling of the surviving
groups. In Example~\ref{ex:oneway-findings} the merging alone detects nothing,
see part~(i), and the selections detect the failure, see part~(ii).
\item[(N-$\catV$)] It is empty by
Example~\ref{ex:oneway-setup}(3).
\end{itemize}
\end{remark}

\begin{remark}\label{rem:oneway-aut}
By~\eqref{eq:commacond} an automorphism $(\varphi,\psi)$ of a design
$(u,\omega,x)$ is given by a relabelling $\psi$ which preserves the group sizes
and by a bijection $\varphi$ which carries $u_a$ onto $u_{\psi(a)}$ for all
$a$. Naturality at $(\varphi,\id)$ is the within-group exchangeability of
$P_{\R,d}(\,\cdot \mid \theta)$, and naturality at a general $(\varphi,\psi)$
is its equivariance under the relabelling of the treatments, which is
available for $\catO$ and not for $\catV$ because $\Theta$ is defined on
$\catO^{\op}$ (see Remark~\ref{cav:scaleV}).
\end{remark}

\begin{remark}\label{rem:oneway-ident}
The map $\Theta(\psi) = \psi^{*} \times \id$ is injective if and only if $\psi$
is surjective, and it is surjective if and only if $\psi$ is injective. Since
$\mu_a$ appears in $P_{\R,d}$ only for $a \in x(u)$, the mean $\mu$ is
identified only on the observed levels, which is the analogue of the condition
on the column rank under which $\beta$ is identified in
Example~\ref{ex:linreg}.
\end{remark}

\subsubsection{Multiple linear regression}

\begin{example}\label{ex:linreg}
\leavevmode
\begin{enumerate}
\item \textbf{Units.} $\catU = \FinSet_{\mathrm{inj}}$, as in
Example~\ref{ex:oneway-setup}.

\item \textbf{Covariate spaces.} $\catO = \Vect_{\R}^{\mathrm{fd}}$. An object
$\omega \cong \R^{p}$ is the space of the covariate vectors. A morphism
$\psi \colon \omega \to \omega'$ is the inclusion of the covariates of a
submodel into a larger set, or a linear recoding such as centring,
orthogonalisation or contrast coding.

\item \textbf{Response scales.} $\catV$ has the single object $\R$, with its
Borel $\sigma$-algebra, and only the identity morphism (see
Remark~\ref{cav:scaleV}).

\item \textbf{Sample spaces.} $\catS = \catV\times\catU^{\op}$ and
$\Gamma(\R,u) = \R^{u}$ with the product $\sigma$-algebra. The morphism
$\varphi^{\op}\colon u' \to u$ is mapped to the restriction
$\res_\varphi \colon \R^{u'} \to \R^{u}$.

\item \textbf{Designs.} An object $(u,\omega,x)$ of $\catD$ is a design
matrix, i.e.\ $x(i) \in \omega$ is the covariate vector of the unit $i$. Once a
basis of $\omega\cong\R^{p}$ and an enumeration of $u$ are chosen, the map $x$
is the $n\times p$ matrix $X$.

\item \textbf{Parameters.} $\catK$ and $U_{\catK}$ are as in
Example~\ref{ex:oneway-setup}. The functor $U_{\catK}$ is faithful because a
linear map is determined by its underlying function. The parameter functor is
the dual, i.e.\ the subrepresentation of the linear functionals inside the
standard representation of
McCullagh~\cite[\S4.2, p.~1235]{mccullagh_what_2002},
\[
\Theta(\omega) \;=\; \omega^{*}\times\R_{>0},
\qquad
\Theta(\psi)(\beta',\sigma^{2}) \;=\; (\beta'\circ\psi,\ \sigma^{2}) .
\]
\end{enumerate}

\noindent For $d=(u,\omega,x)$ and
$\theta = (\beta,\sigma^{2}) \in \omega^{*}\times\R_{>0}$ we set
\[
P_{\R,d}(\,\cdot \mid \beta,\sigma^{2})
\;\coloneqq\; \bigotimes_{i\in u} \Norm(\beta(x(i)),\,\sigma^{2})
\qquad\text{on } \R^{u},
\]
i.e.\ $Y \sim \Norm_{n}(X\beta,\sigma^{2}I_{n})$ in coordinates. This model has
no intercept. One is obtained by taking as $\catO$ the category of the finite
dimensional real affine spaces and affine maps, and as $\Theta(\omega)$ the
product of the space of the affine functionals on $\omega$ with $\R_{>0}$. This
is what the regression examples of~\cite[\S6]{mccullagh_what_2002} effectively
use.
\end{example}

\begin{lemma}\label{lem:linreg-model}
The family $P$ of Example~\ref{ex:linreg} is a statistical model.
\end{lemma}

\begin{proof}
The proof is the same as that of Lemma~\ref{lem:oneway-model}. Let
$f=(\varphi,\psi) \colon d \to d'$ and let $\theta'=(\beta',\sigma^{2}) \in
\Theta(\omega')$. Then $P_{\R,d}\bigl(\,\cdot \mid \Theta(\psi)\theta'\bigr)
= \bigotimes_{i\in u}\Norm(\beta'(\psi(x(i))),\sigma^{2})$, and the pushforward
of $P_{\R,d'}(\,\cdot\mid\theta')$ along $\res_{\varphi}$ is
$\bigotimes_{i\in u}\Norm(\beta'(x'(\varphi(i))),\sigma^{2})$. The two measures
coincide by~\eqref{eq:commacond}.
\end{proof}

\begin{caveat}\label{cav:scaleV}
Since $L$ does not depend on $\catV$, a morphism $\alpha \colon v \to v'$
different from the identity imposes~\eqref{eq:NV} at the same value of the
parameter. Let $\alpha(y)=ay+b$. Then the left hand side of~\eqref{eq:NV} is
$\bigotimes_{i}\Norm(a\beta(x(i))+b,\ a^{2}\sigma^{2})$, which is equal to
$P_{\R,d}(\,\cdot\mid\beta,\sigma^{2})$ only if $(a,b)=(1,0)$. Hence the
Gaussian linear model is not natural over $\catV = \mathrm{Aff}(\R)$, and for
this reason we take $\catV$ discrete here and in
Example~\ref{ex:oneway-setup}. The affine equivariance of the model is the
statement that the model is a fixed point of an action on the parameter. To
express it in the hexagon of Br\o ns one needs a parameter functor defined on
$\catV^{\op}\times\catO^{\op}$ and not on $\catO^{\op}$ alone. Indeed the
parameter of McCullagh is a subrepresentation of $v^{\omega}$ and it does
change with the response scale~\cite[p.~1235]{mccullagh_what_2002}, so that
the independence of $\Theta$ from $\catV$ comes from the formulation of
Br\o ns. Thus affine equivariance is a mild generalisation of
Definition~\ref{def:designdata} and not an instance of it. The equivariance
under the relabelling of the treatments is instead an instance of naturality,
because it acts through $\catO^{\op}$.
\end{caveat}

\subsubsection{Binary classification}

\begin{example}\label{ex:logreg}
\leavevmode
\begin{enumerate}
\item \textbf{Units.} $\catU = \FinSet_{\mathrm{inj}}$, as above.

\item \textbf{Covariate spaces.} $\catO = \Vect_{\R}^{\mathrm{fd}}$. The
objects are feature spaces, and the morphisms are linear feature maps.

\item \textbf{Response scales.} $\catV$ has the two objects
$\mathbf{2}=\{0,1\}$ and $\mathbf{1}=\{*\}$, and its morphisms are the
identities and the unique map $!\colon\mathbf{2}\to\mathbf{1}$. Thus the label
can be discarded and cannot be relabelled. Indeed the involution
$y \mapsto 1-y$ carries $P_{\mathbf{2},d}(\,\cdot\mid\beta)$ to
$P_{\mathbf{2},d}(\,\cdot\mid -\beta)$, hence it does not
satisfy~\eqref{eq:NV}. It is an equivariance in the sense of
Remark~\ref{cav:scaleV}.

\item \textbf{Sample spaces.} $\Gamma(v,u)=v^{u}$ with the discrete
$\sigma$-algebra, so that $\Gamma(\mathbf{2},u)=\{0,1\}^{u}$ is the space of
the labels and $\Gamma(\mathbf{1},u)$ is a point. The morphism $\varphi^{\op}$
is mapped to the restriction.

\item \textbf{Designs.} An object $(u,\omega,x)$ of $\catD$ is the feature
matrix. The labels are random and appear in $\Gamma$, not in $\catD$, i.e.\ the
model is conditional on the features.

\item \textbf{Parameters.} $\catK = \Vect_{\R}^{\mathrm{fd}}$, $U_{\catK}$ is
the underlying-set functor, $\Theta(\omega)=\omega^{*}$ and
$\Theta(\psi)(\beta') = \beta'\circ\psi$ for $\beta' \in \omega'^{*}$.
\end{enumerate}

\noindent Let $\expit(t)=(1+e^{-t})^{-1}$. For $d = (u,\omega,x)$ and
$\beta \in \Theta(\omega) = \omega^{*}$ we set
\[
P_{\mathbf{2},d}(\,\cdot\mid\beta)
\;\coloneqq\; \bigotimes_{i\in u}\Bern\!\bigl(\expit(\beta(x(i)))\bigr)
\quad\text{on }\{0,1\}^{u},
\qquad
P_{\mathbf{1},d}(\,\cdot\mid\beta) \;\coloneqq\; \delta \ \text{on } \{*\}^{u},
\]
where $\delta$ is the unique probability measure on the point $\{*\}^{u}$.
\end{example}

\begin{lemma}\label{lem:logreg-model}
The family $P$ of Example~\ref{ex:logreg} is a statistical model.
\end{lemma}

\begin{proof}
For~\eqref{eq:ND} at $v=\mathbf{2}$ the proof is the same as that of
Lemma~\ref{lem:oneway-model}. Indeed, the marginal onto $\varphi(u)$ and the
model at the pulled-back parameter are
\[
\bigotimes_{i\in u}\Bern\bigl(\expit(\beta'(x'(\varphi(i))))\bigr)
\qquad\text{and}\qquad
\bigotimes_{i\in u}\Bern\bigl(\expit((\beta'\circ\psi)(x(i)))\bigr),
\]
which coincide by~\eqref{eq:commacond}. At $v=\mathbf{1}$ both sides are the
unique probability measure on a point. For~\eqref{eq:NV} at
$!\colon\mathbf{2}\to\mathbf{1}$, the pushforward of
$P_{\mathbf{2},d}(\,\cdot\mid\beta)$ along
$!^{u}\colon\{0,1\}^{u}\to\{*\}^{u}$ is the Dirac measure, which is
$P_{\mathbf{1},d}(\,\cdot\mid\beta)$ for every $\beta$.
\end{proof}

\subsubsection{Non-examples}

The two families below are well-defined and indexed by the designs, but they
are not statistical models. In both cases condition (N-$\catU$) fails.

\begin{example}\label{ex:centring}
With the data of Example~\ref{ex:linreg} we set
\[
\widetilde{P}_{\R,d}(\,\cdot\mid\beta,\sigma^{2})
\coloneqq \bigotimes_{i\in u}\Norm(\beta(x(i)-\bar{x}_{d}),\ \sigma^{2}),
\qquad
\bar{x}_{d} = \frac{1}{|u|}\sum_{i\in u} x(i).
\]
Then~\eqref{eq:ND} fails already for $\psi=\id$ and $\varphi$ a proper
inclusion. Indeed $\bar{x}_{d} \neq \bar{x}_{d'}$ in general, so that the mean
of a given unit changes when other units are removed. The mean $\bar{x}_{d}$ is
a statistic of the design, and the failure comes from its use in the definition
of the family. The same mechanism underlies the non-example of Tjur (see
Example~\ref{ex:alphabetaxbar}).
\end{example}

\begin{example}\label{ex:softmax}
With the data of Example~\ref{ex:logreg} we set
\[
\widetilde{P}_{\mathbf{2},d}(\,\cdot\mid\beta)
\coloneqq \bigotimes_{i\in u}\Bern\bigl(s_{i}(\beta,d)\bigr),
\qquad
s_{i} \;=\; \frac{e^{\beta(x(i))}}{\sum_{j\in u}e^{\beta(x(j))}},
\]
i.e.\ a softmax taken across the sample and not across the label set.
Then~\eqref{eq:ND} fails as in Example~\ref{ex:centring}, since removing a unit
rescales all the remaining probabilities. Normalisation across the sample is a
common operation in learning pipelines.
\end{example}

\begin{remark}\label{rem:slots}
Summing up, condition (N-$\catU$) says that the marginal of the model on a
subsample is the model of the subsample, and it excludes the normalisations
which depend on the sample, as in Examples~\ref{ex:centring}
and~\ref{ex:softmax}. Condition (N-$\catO$) says that the parameter is pulled
back along the maps of covariate spaces, and it excludes the
reparametrisations which depend on the design. Condition (N-$\catV$) expresses
invariance at a fixed value of the parameter, and it excludes the relabelling
and the rescaling of the response, which are equivariances (see
Remark~\ref{cav:scaleV}).
\end{remark}
\section{Markov-categorical background}\label{sec:markov}

We recall the few notions on Markov categories that we need. We refer the
reader to~\cite{fritz_synthetic_2020,perrone_markov_2024} for the general
theory.

\subsection{The Giry monad}

Let $X=(X,\Sigma_X)$ be a measurable space. We denote by $\G X$ the set of the
probability measures on $X$, endowed with the smallest $\sigma$-algebra such
that the evaluation maps
\[
\operatorname{ev}_A \colon \G X\to [0,1],
\qquad
\operatorname{ev}_A(\pi)=\pi(A),
\]
are measurable for all $A\in\Sigma_X$. Let $f \colon X\to Y$ be a measurable
map. We denote by $\G f \colon \G X\to \G Y$ the pushforward along $f$, i.e.
\[
(\G f)(\pi)(B) \;=\; (f_\ast\pi)(B) \;=\; \pi\bigl(f^{-1}(B)\bigr)
\]
for all $\pi \in \G X$ and $B\in\Sigma_Y$. In this way we obtain a functor
$\G \colon \Meas\to\Meas$. Let
\[
\eta_X \colon X\to \G X,
\qquad
\eta_X(x)=\delta_x,
\]
where $\delta_x$ is the Dirac measure at $x$, and let
\[
\mu_X \colon \G\G X\to\G X,
\qquad
\mu_X(\Xi)(A) = \int_{\G X}\nu(A)\,\Xi(\mathrm{d}\nu),
\qquad
A\in\Sigma_X.
\]
Thus $\mu_X$ maps a probability measure on the space of the probability
measures on $X$ to its mixture. Equivalently,
\[
\int_X h(x)\,\mu_X(\Xi)(\mathrm{d}x)
=
\int_{\G X}
\Bigl(
\int_X h(x)\,\nu(\mathrm{d}x)
\Bigr)\,
\Xi(\mathrm{d}\nu)
\]
for every bounded measurable function $h \colon X\to\R$. The triple
$(\G,\eta,\mu)$ is a monad on $\Meas$, called the \emph{Giry monad}
(see~\cite{giry_categorical_1982}).

\subsection{The category \texorpdfstring{$\Stoch$}{Stoch}}

A Kleisli morphism of the Giry monad from $X$ to $Y$ is a measurable map
$k \colon X\to\G Y$. Equivalently, it is a \emph{Markov kernel}, i.e.\ a map
$k \colon X\times\Sigma_Y\to[0,1]$ such that $B\mapsto k(x)(B)$ is a
probability measure on $Y$ for all $x\in X$, and $x\mapsto k(x)(B)$ is
measurable for all $B\in\Sigma_Y$. The composition of $k \colon X\to\G Y$ and
$k' \colon Y\to\G Z$ is given by the Chapman--Kolmogorov formula
\begin{equation}\label{eq:CK}
(k'\circ k)(x)(C)
=
\int_Y k'(y)(C)\,k(x)(\mathrm{d}y),\qquad C\in\Sigma_Z.
\end{equation}

\begin{definition}\label{def:stoch}
The Kleisli category of the Giry monad is denoted by $\Stoch$. Its objects are
the measurable spaces and its morphisms $f \colon X \rightsquigarrow Y$ are
the Markov kernels, with the composition~\eqref{eq:CK}
(see~\cite{fritz_synthetic_2020,giry_categorical_1982,perrone_markov_2024}).
We denote by $\FinStoch$ the full subcategory of $\Stoch$ of the finite
discrete measurable spaces. Its morphisms are the stochastic matrices.
\end{definition}

We write $\cat{A}(B,C)$ for the set of the morphisms from $B$ to $C$ in a
category $\cat{A}$. By definition there is a bijection
\begin{equation}\label{eq:kleisli}
\Stoch(X, Y) \;\cong\; \Meas\bigl(X, \G Y\bigr),
\end{equation}
which is natural in $X$, via the precomposition with measurable maps, and in
$Y$, via the pushforward.

The category $\Stoch$ is a \emph{Markov category}, i.e.\ it is symmetric
monoidal under the product of measurable spaces and it has copy and discard
maps at every object, with naturality imposed on the discard maps
alone~\cite[Def.~2.1]{fritz_synthetic_2020}. We write
$\mathrm{copy}_{X} \colon X \to X \otimes X$ for the copy map at $X$.

\begin{definition}\label{def:deterministic}
A morphism $f \colon X \to Y$ of a Markov category is called
\emph{deterministic} if it is a homomorphism of the copy comonoids, i.e.\ if
\[
\mathrm{copy}_{Y} \circ f \;=\; (f \otimes f) \circ \mathrm{copy}_{X}
\]
(see~\cite[Def.~10.1]{fritz_synthetic_2020}).
\end{definition}

Thus $f$ is deterministic if copying its output is the same as applying $f$
independently to two copies of the input. In $\Stoch$ this means that
$f(x)(B) \in \{0,1\}$ for every $x \in X$ and every $B \in \Sigma_{Y}$, i.e.\
that each $f(x)$ is a $\{0,1\}$-valued
measure~\cite[Ex.~10.4]{fritz_synthetic_2020}. In $\FinStoch$ it means that
the entries of the stochastic matrix of $f$ belong to $\{0,1\}$. The
deterministic morphisms contain the identities and are closed under
composition and under $\otimes$. Hence they form a wide symmetric monoidal
subcategory $\Stoch_{\det} \subseteq \Stoch$.

\begin{definition}\label{def:dirac}
The \emph{Dirac embedding} $\del \colon \Meas \to \Stoch$ is the functor which
is the identity on the objects and maps a measurable map $f$ to the kernel
$x \mapsto \delta_{f(x)}$, called the \emph{Dirac kernel} of $f$. It is strong
monoidal and its image is contained in $\Stoch_{\det}$, since a Dirac kernel is
$\{0,1\}$-valued.
\end{definition}

\begin{lemma}\label{lem:dirac}
Let $Y$ be a measurable space.
\begin{enumerate}
\item The functor $\del$ is injective on $\Meas(X,Y)$ for all $X$ if and only
if $\Sigma_{Y}$ separates points. This holds when $Y$ is standard Borel,
discrete or finite, and it does not hold for all $Y$.
\item If $\Sigma_{Y}$ is countably generated and separates points, in
particular if $Y$ is standard Borel, then every deterministic morphism
$X \rightsquigarrow Y$ is the Dirac kernel of a measurable map. This always
holds in $\FinStoch$.
\item Suppose that the codomains of the measurable maps $f_{j}$ separate
points. Then a diagram of Dirac kernels $\del f_{j}$ commutes in $\Stoch$ if
and only if the diagram of the $f_{j}$ commutes in $\Meas$.
\end{enumerate}
\end{lemma}

\begin{proof}
We have $\del f = \del g$ if and only if $\delta_{f(x)} = \delta_{g(x)}$ for
all $x$, and $\delta_{y} = \delta_{y'}$ implies $y = y'$ for all $y, y'$ if and
only if $\Sigma_{Y}$ separates points (for the converse take $X$ a point).
This gives (1), and (3) follows from
(1) since $\del$ is a functor.

For (2), let $(B_n)_{n}$ generate $\Sigma_Y$ and let $\nu$ be a
$\{0,1\}$-valued probability measure on $Y$. We set $C_n = B_n$ if
$\nu(B_n) = 1$ and $C_n = Y \setminus B_n$ otherwise, and
$C = \bigcap_n C_n$. Then $\nu(C) = 1$, hence $C \neq \emptyset$. Let
$y, y' \in C$. The class of the sets $B \in \Sigma_Y$ which contain both or
neither of $y, y'$ is a $\sigma$-algebra which contains every $B_n$, hence it
is $\Sigma_Y$, and $y = y'$ because $\Sigma_Y$ separates points. Thus
$C = \{y\}$ and $\nu = \delta_y$. Let now $k \colon X \rightsquigarrow Y$ be
deterministic. Each $k(x)$ is
$\{0,1\}$-valued~\cite[Ex.~10.4]{fritz_synthetic_2020}, hence
$k(x) = \delta_{f(x)}$ for a unique $f(x) \in Y$. The map $f$ is measurable
because $f^{-1}(B) = \{x : k(x)(B) = 1\} \in \Sigma_X$ for every
$B \in \Sigma_Y$, and $k = \del f$.
\end{proof}

\begin{remark}\label{rem:dirac-hyp}
The property of Lemma~\ref{lem:dirac}(2) says that the $\{0,1\}$-valued
probability measures on $Y$ are the points of $Y$. Conditions~(a) and~(b) of
Appendix~\ref{sec:Bref} are the two instances of this requirement which are
used in Proposition~\ref{wprop:copower}.
\end{remark}
\section{The bridge}\label{sec:bridge}

In this section we prove Theorem~\ref{thm:main}. By~\eqref{eq:kleisli} a
component $P_{v,d} \colon U_\catK\Theta(\omega) \to \Prb(\Gamma(v,u))$ of a
statistical model gives a morphism
$\Theta(\omega) \rightsquigarrow \Gamma(v,u)$ of $\Stoch$ as soon as it is
measurable in $\theta$. This condition is empty in the finite case, which we
consider first.

\subsection{Finite case}\label{sec:finite}

Suppose that the design data of Definition~\ref{def:designdata} are
\emph{finite}, i.e.\ that the objects of $\catV$, $\catO$ and $\catU$ have
finite underlying sets, that the spaces $\Gamma(v,u)$ are finite and discrete,
and that $U_\catK\Theta(\omega)$ is finite for every $\omega$. We endow each
$U_\catK\Theta(\omega)$ with the discrete $\sigma$-algebra, so that
$U_\catK\Theta$ lifts uniquely to a functor
$\Theta_\Meas \colon \catO^{\op} \to \Meas$ such that
$U_\catK \Theta = |\!-\!| \circ \Theta_\Meas$, where
$|\!-\!| \colon \Meas\to\Set$ is the forgetful functor. We set
\[
\mathbf{L} \coloneqq \del \circ \Theta_\Meas \circ \pr_\catO^{\op} \circ \pr_2,
\qquad
\mathbf{R} \coloneqq \del \circ \Gamma \circ (\id_\catV \times \pr_\catU^{\op}),
\]
where $\del$ is the Dirac embedding of Definition~\ref{def:dirac}. These are
two functors $\catV \times \catD^{\op} \to \FinStoch$.

\begin{theorem}\label{thm:bridge-finite}
Suppose that the design data are finite. For a statistical model
$P \colon L \Rightarrow R$ let $\mathbf{P}_{v,d}$ be the kernel
$\theta \mapsto P_{v,d}(\,\cdot \mid \theta)$. Then $P \mapsto \mathbf{P}$ is a
bijection between
\begin{enumerate}
\item the statistical models $P \colon L \Rightarrow R$ of
Definition~\ref{def:mcbmodel};
\item the natural transformations
$\mathbf{P}\colon \mathbf{L} \Rightarrow \mathbf{R}$, i.e.\ the families of
morphisms
$\mathbf{P}_{v,d} \colon \mathbf{L}(v,d) \rightsquigarrow \mathbf{R}(v,d)$ of
$\FinStoch$, indexed by the objects of $\catV\times\catD^{\op}$, such that the
square
\begin{equation}\label{eq:natsquare}
\begin{tikzcd}[column sep=3em]
\mathbf{L}(v,d) \arrow[r, rightsquigarrow, "\mathbf{P}_{v,d}"]
\arrow[d, "\mathbf{L}(m)"']
& \mathbf{R}(v,d) \arrow[d, "\mathbf{R}(m)"] \\
\mathbf{L}(v',d') \arrow[r, rightsquigarrow, "\mathbf{P}_{v',d'}"']
& \mathbf{R}(v',d')
\end{tikzcd}
\end{equation}
commutes in $\FinStoch$ for every morphism $m \colon (v,d) \to (v',d')$ of
$\catV \times \catD^{\op}$.
\end{enumerate}
The morphisms $\mathbf{L}(m)$ and $\mathbf{R}(m)$ are Dirac kernels, hence
deterministic. The components $\mathbf{P}_{v,d}$ are arbitrary stochastic
matrices, and $\mathbf{P}_{v,d}$ is deterministic if and only if
$P_{v,d}(\,\cdot \mid \theta)$ is a Dirac measure for every $\theta$.
\end{theorem}

\begin{proof}
Let $d = (u,\omega,x)$. Since $\del$ is the identity on the objects we have
$\mathbf{L}(v,d) = \Theta_\Meas(\omega)$ and $\mathbf{R}(v,d) = \Gamma(v,u)$,
which are finite and discrete. Recall that $U_\catK\Theta(\omega)$ is the
underlying set of $\Theta_\Meas(\omega)$ and that $\Prb(S)$ is the underlying
set of $\G S$. Since $\Theta_\Meas(\omega)$ is discrete, every map from it is
measurable, hence
\[
\Set\bigl(U_\catK\Theta(\omega), \Prb(\Gamma(v,u))\bigr)
= \Meas\bigl(\Theta_\Meas(\omega), \G\Gamma(v,u)\bigr).
\]
Composing with~\eqref{eq:kleisli} we obtain a bijection
\begin{equation}\label{eq:transpose}
\Set\bigl(U_\catK\Theta(\omega),\, \Prb(\Gamma(v,u))\bigr)
\;\xrightarrow{\ \cong\ }\;
\FinStoch\bigl(\Theta_\Meas(\omega),\, \Gamma(v,u)\bigr),
\qquad P_{v,d} \longmapsto \mathbf{P}_{v,d}.
\end{equation}

Let $k \colon X \rightsquigarrow Y$ be a kernel and let $g \colon Y \to Y'$ be
a measurable map. By~\eqref{eq:CK} we have
\begin{equation}\label{eq:push}
\bigl(\del(g) \circ k\bigr)(x)(B)
= \int \delta_{g(y)}(B)\, k(x)(\mathrm{d}y)
= k(x)\bigl(g^{-1}B\bigr)
= \bigl(\G(g) \circ k\bigr)(x)(B),
\end{equation}
i.e.\ the composition with $\del(g)$ corresponds under~\eqref{eq:transpose} to
the composition with the pushforward $\Prb(g) = \G(g)$. Let now
$h \colon X' \to X$ be a measurable map. Then
\begin{equation}\label{eq:pull}
\bigl(k \circ \del(h)\bigr)(x')(B)
= \int k(x)(B)\, \delta_{h(x')}(\mathrm{d}x)
= k\bigl(h(x')\bigr)(B),
\end{equation}
i.e.\ the precomposition with $\del(h)$ is the precomposition with $h$.

Let $m = (\alpha, (\varphi,\psi)^{\op})$ be a morphism of
$\catV \times \catD^{\op}$ with target $(v',d')$, where $d' = (u',\omega',x')$,
so that $(\varphi,\psi) \colon d' \to d$ is a morphism of $\catD$ with
$\varphi \colon u' \to u$ and $\psi \colon \omega' \to \omega$. We have
$L(m) = U_\catK\Theta(\psi)$ and
$R(m) = \Prb\bigl(\Gamma(\alpha,\varphi)\bigr)$, while
\[
\mathbf{L}(m) = \del\bigl(\Theta_\Meas(\psi)\bigr),
\qquad
\mathbf{R}(m) = \del\bigl(\Gamma(\alpha,\varphi)\bigr).
\]
The naturality square of $P$ at $m$ is the equality
\[
\Prb\bigl(\Gamma(\alpha,\varphi)\bigr) \circ P_{v,d}
\;=\; P_{v',d'} \circ U_\catK\Theta(\psi)
\]
of maps $U_\catK\Theta(\omega) \to \Prb(\Gamma(v',u'))$. By~\eqref{eq:push}
the left hand side corresponds under~\eqref{eq:transpose} to
$\mathbf{R}(m) \circ \mathbf{P}_{v,d}$, and by~\eqref{eq:pull} the right hand
side corresponds to $\mathbf{P}_{v',d'} \circ \mathbf{L}(m)$.
Since~\eqref{eq:transpose} is a bijection, the naturality square of $P$ at $m$
commutes if and only if~\eqref{eq:natsquare} commutes. Hence
$P \mapsto \mathbf{P}$ is a bijection between the statistical models and the
natural transformations $\mathbf{L} \Rightarrow \mathbf{R}$.

The morphisms $\mathbf{L}(m)$ and $\mathbf{R}(m)$ are Dirac kernels, hence
they are deterministic by Definition~\ref{def:dirac}. Finally
$\mathbf{P}_{v,d}$ belongs to the image of $\del$ if and only if
$P_{v,d}(\,\cdot \mid \theta)$ is a Dirac measure for every $\theta$, and the
theorem is proved.
\end{proof}

\begin{remark}\label{rem:bridge-finite}
The discreteness of the parameter object is used only to
obtain~\eqref{eq:transpose}, while the finiteness of the spaces is used only to
have $\mathbf{L}$, $\mathbf{R}$ and $\mathbf{P}$ in $\FinStoch$. The equality
which gives~\eqref{eq:transpose} is the bijection of the adjunction
$D \dashv U$ of Proposition~\ref{prop:coreflection} at the pair
$\bigl(U_\catK\Theta(\omega), \G\Gamma(v,u)\bigr)$. By
Theorem~\ref{thm:bridge-finite} all the notions available in $\FinStoch$ apply
to the finite statistical models.
\end{remark}

\subsection{Measurable case}

\begin{definition}\label{def:measurable-model}
Suppose that the functor $U_\catK\Theta \colon \catO^{\op} \to \Set$ lifts to a
functor $\Theta_\Meas \colon \catO^{\op} \to \Meas$, i.e.\ that
$|\!-\!| \circ \Theta_\Meas = U_\catK \circ \Theta$. This means that each
parameter object is endowed with a measurable structure and that the
reparametrisation maps $\Theta(\psi)$ are measurable. A statistical model
$P \colon L \Rightarrow R$ is called \emph{measurable} if each component
$P_{v,d}$ is measurable as a map $\Theta_\Meas(\omega) \to \G \Gamma(v,u)$,
i.e.\ if $\theta \mapsto P_{v,d}(A \mid \theta)$ is measurable for every
measurable subset $A$ of $\Gamma(v,u)$.
\end{definition}

Let $\mathbf{L}, \mathbf{R} \colon \catV \times \catD^{\op} \to \Stoch$ be
given as in the finite case, i.e.\
$\mathbf{L} = \del \circ \Theta_\Meas \circ \pr_\catO^{\op} \circ \pr_2$ and
$\mathbf{R} = \del \circ \Gamma \circ (\id_\catV \times \pr_\catU^{\op})$.

\begin{theorem}\label{thm:bridge-meas}
Let $\Theta_\Meas$ be as in Definition~\ref{def:measurable-model}. Then
$P \mapsto \mathbf{P}$, where $\mathbf{P}_{v,d}$ is the kernel
$\theta \mapsto P_{v,d}(\,\cdot \mid \theta)$, is a bijection between the
measurable statistical models and the natural transformations
$\mathbf{P} \colon \mathbf{L} \Rightarrow \mathbf{R}$ of functors
$\catV \times \catD^{\op} \to \Stoch$, i.e.\ the families of morphisms
$\mathbf{P}_{v,d} \colon \mathbf{L}(v,d) \rightsquigarrow \mathbf{R}(v,d)$ of
$\Stoch$ such that the square~\eqref{eq:natsquare} commutes in $\Stoch$ for
every morphism $m$. The morphisms $\mathbf{L}(m)$ and $\mathbf{R}(m)$ are
Dirac kernels, hence deterministic, for every morphism $m$. The components
$\mathbf{P}_{v,d}$ need not be deterministic.
\end{theorem}

\begin{proof}
The bijection~\eqref{eq:kleisli} and its naturality hold for arbitrary
measurable spaces, and $P_{v,d}$ belongs to
$\Meas\bigl(\Theta_\Meas(\omega),\,\G\Gamma(v,u)\bigr)$ by hypothesis. The
proof then goes as the one of Theorem~\ref{thm:bridge-finite},
since~\eqref{eq:push} and~\eqref{eq:pull} do not use finiteness.
\end{proof}

\begin{definition}\label{def:bridged}
The natural transformation $\mathbf{P}$ of Theorem~\ref{thm:bridge-meas} is
called the \emph{bridged model} associated to $P$.
\end{definition}

\begin{proof}[Proof of Theorem~\ref{thm:main}]
It follows from Theorem~\ref{thm:bridge-finite} and
Theorem~\ref{thm:bridge-meas}.
\end{proof}

\begin{remark}\label{rem:meas-hyp}
For a concrete model the hypothesis says that
$\theta \mapsto P_{v,d}(A \mid \theta)$ is a measurable function for every event
$A$. This is what is needed in order to integrate the model against a prior,
and it holds when the model has a density which is jointly measurable in the
observation and in the parameter. It is a restriction with respect to the
formulation of McCullagh, which is $\Set$-valued and imposes no measurability
in $\theta$.
Hence Theorem~\ref{thm:bridge-meas} does not say that every statistical model
in the sense of Definition~\ref{def:mcbmodel} gives a natural transformation
in $\Stoch$. The hypothesis holds, with the obvious Borel structures, for
every exponential family in its natural parametrisation and for all the
examples of~\cite{mccullagh_what_2002} with finite dimensional parameter
spaces. The measure-valued and function-valued parameter spaces
of~\cite[\S8]{mccullagh_what_2002} require the choice of a $\sigma$-algebra,
and we do not consider them.

Since $\mathbf{L}(m)$ and $\mathbf{R}(m)$ are deterministic for every $m$,
the functors $\mathbf{L}$ and $\mathbf{R}$ are objects of the Markov category
$\mathrm{Fun}(\catV\times\catD^{\op},\Stoch)$ of
Fritz~\cite[Prop.~7.1]{fritz_synthetic_2020}, whose objects are the functors
into the deterministic morphisms, whose morphisms are all the natural
transformations, and whose monoidal structure and copy maps are pointwise;
the determinism of the functors is what makes the pointwise copy maps
natural. Theorem~\ref{thm:bridge-meas} thus says that a measurable
statistical model is a morphism $\mathbf{L} \rightsquigarrow \mathbf{R}$ of
this category, and the same reading applies to Definition~\ref{def:coherence}
below: a $\cat{Q}$-coherent family between functors with deterministic
images is a morphism of $\mathrm{Fun}(\cat{Q}^{\op},\Stoch)$. We use this
only as a description. Whether $\mathrm{Fun}(\catD,\cat{C})$ has
conditionals when $\cat{C}$ does is an open
problem~\cite[Problem~11.9]{fritz_synthetic_2020}, since the conditionals
taken at each object need not form a natural transformation;
Remark~\ref{rem:open} meets the same question for the Bayesian inversions.
\end{remark}

\begin{example}\label{ex:oneway-cont}
In Example~\ref{ex:oneway-setup} we have $\Theta(\omega)=\R^\omega\times\R_{>0}$
with its Borel structure, and the measurability of
$(\mu,\sigma^2) \mapsto \bigotimes_i \Norm(\mu_{x(i)},\sigma^2)$ is classical.
Hence Theorem~\ref{thm:bridge-meas} applies, and the conditions of
Remark~\ref{rem:oneway-slots} become commuting squares of kernels.
\end{example}

\begin{remark}[Reduction to the finite case]\label{rem:reduction}
The data of an experiment are finite, since responses and covariates are
bounded by the range of the instrument and are recorded with finite precision.
This does not reduce Theorem~\ref{thm:bridge-meas} to
Theorem~\ref{thm:bridge-finite}, but the measurable case can be approximated by
finite ones in the following sense. Let $F \subset \R$ be a finite set, for
example a set of floating-point numbers, and let $\rho \colon \R \to F$ be the
rounding map, which is measurable and surjective. We add the object $F$ and the
morphism $\rho$ to $\catV$ and we set
$P_{F,d}(\,\cdot \mid \theta) \coloneqq (\rho^{u})_{*}\,P_{\R,d}(\,\cdot \mid
\theta)$, so that~\eqref{eq:NV} holds by construction. Hence the quantised
model is the image of the model under a morphism of response scales, with no
hypothesis and no error term. It is not an instance of
Theorem~\ref{thm:bridge-finite}. Indeed $\rho$ does not change $\Theta$, so
that the parameter objects remain continua. A finite grid of parameter values
is preserved by the maps $\Theta(\psi)$ when these are reindexings, as in the
one-way layout. This is the setting of Section~\ref{sec:hochschild}, and
Proposition~\ref{prop:oneway-cont-recalib} gives the statement which the grids
approximate. A grid is not preserved when $\Theta(\psi) = \psi^{*}$ uses the
linear structure, as in Examples~\ref{ex:linreg} and~\ref{ex:logreg}, because a
floating-point grid is not closed under addition. Naturality is an equational
condition, and we do not develop a notion of approximate naturality here. For
this reason the computations of Section~\ref{sec:hochschild} are carried out in
exact arithmetic. Note also that the pushforward of a Gaussian family along
$\rho$ is not an exponential family, and that the group means and the
within-group sum of squares are not sufficient for it. Finally, let
$(\rho_n \colon S \to F_n)_{n}$ be a refining sequence of finite quantisations
whose cells generate the $\sigma$-algebra of $S$. The canonical map
$\Prb(S) \to \varprojlim_n \Prb(F_n)$ is injective, since the cells form a
$\pi$-system, and it is not surjective in general. For example, for the dyadic
quantisation of $[0,1)$, the consistent family concentrated on the cells
$[\tfrac12 - 2^{-n}, \tfrac12)$ has no preimage. Thus the measurable case is a
projective limit of finite cases and not a special case of them.
\end{remark}

\subsection{The arrow category}

In Appendix~\ref{sec:Bref} we compare the category $\StatMod$ of Br\o ns with
the category whose objects are the morphisms of $\Stoch$ and whose morphisms
are the commuting squares with deterministic sides. The bridge is a computation
with a representable functor and a copower; see~\eqref{eq:copower}. On the
objects whose parameter set and sample space satisfy two regularity
hypotheses, which always hold in the finite case, $\StatMod$ is equivalent to
the category of the arrows of $\Stoch$ whose domain is a copower of the
monoidal unit (Proposition~\ref{wprop:copower}). Furthermore, $\StatMod$ is a
coreflective subcategory of its measurable analogue
(Proposition~\ref{prop:coreflection}). The sections that follow do not depend
on the appendix.
\section{Meaningfulness as naturality}\label{sec:meaningful}

In this section we prove Theorem~\ref{thm:main-tjur}. Recall that
Tjur~\cite{tjur_discussion_2002} considers a parameter function, i.e.\ a rule
which assigns to each finite sample of covariate values a parameter of the
corresponding model, in his words, a ``function of the unknown
distribution''~\cite[p.~1299]{tjur_discussion_2002}. He calls it
\emph{meaningful} when it is invariant under the formation of marginal models,
i.e.\ when reducing the design by removing some of the sampled $x$'s leaves
the associated parameter unchanged. The criterion of
McCullagh~\cite[\S4.5]{mccullagh_what_2002} is that a subparameter be a natural
transformation of functors on the design category. Thus Tjur quantifies over
the removals of sample points and McCullagh over all the design morphisms.

\subsection{Coherence}

Let the design data be as in Definition~\ref{def:designdata}. Recall that an
\emph{insertion} is a proper inclusion $\omega' \subsetneq \omega$ between
objects of $\catO$ (see Example~\ref{ex:oneway-setup}). We denote by
$\catO_{\mathrm{red}} \subseteq \catO$ the subcategory generated by the
insertions. Its morphisms are the inclusions, proper or not. We denote by
$\catO_{\mathrm{inj}} \subseteq \catO$ the subcategory generated by
$\catO_{\mathrm{red}}$ and by the isomorphisms of $\catO$. In
Example~\ref{ex:oneway-setup} its morphisms are the injections. An insertion
leaves the surviving labels unchanged, while an injection which is not an
inclusion is a relabelling followed by an insertion.

\begin{definition}\label{def:coherence}
Let $F, G \colon \catO^{\op} \to \Stoch$ be functors and let
$\cat{Q} \subseteq \catO$ be a subcategory which contains all the objects,
called the \emph{quantifier class}. A family
$t = (t_{\omega} \colon F(\omega) \rightsquigarrow G(\omega))_{\omega \in \ob\catO}$
of morphisms of $\Stoch$ is called \emph{$\cat{Q}$-coherent} if
\begin{equation}\label{eq:coherence}
t_{\omega} \circ F(\psi) \;=\; G(\psi) \circ t_{\omega'}
\qquad\text{for every } \psi \colon \omega \to \omega' \text{ in } \cat{Q},
\end{equation}
i.e.\ if $t$ is a natural transformation
$F|_{\cat{Q}^{\op}} \Rightarrow G|_{\cat{Q}^{\op}}$.
\end{definition}

Coherence over a class implies coherence over every subclass, and
$\catO$-coherence is naturality. In words, a family is coherent over $\cat{Q}$
if every change of covariate space which belongs to $\cat{Q}$ carries the
family to itself.

McCullagh takes the morphisms of $\catO$ to be
injective~\cite[\S4.1]{mccullagh_what_2002}. Hence under his convention a
merging is not a morphism, and the difference between his criterion and the
one of Tjur is given by the bijections, i.e.\ by the relabellings. In
Definition~\ref{def:designdata} we follow Br\o ns and we impose no such
restriction. Since every map of covariate spaces is a surjection followed by
an injection, in our setting there are three quantifier classes
\[
\underbrace{\catO_{\mathrm{red}}}_{\text{Tjur}}
\;\subseteq\;
\underbrace{\catO_{\mathrm{inj}}}_{\text{McCullagh as stated}}
\;\subseteq\;
\underbrace{\catO}_{\text{naturality here}},
\]
which are strictly nested in Example~\ref{ex:oneway-setup}. The first
inclusion adds the relabellings and the second one adds the mergings. The
corresponding conditions are ordered in the opposite way, i.e.\ coherence over
$\catO$ is the strongest one. Neither source asserts that the criterion of
Tjur and the one of McCullagh coincide, and in general they do not. By
Theorem~\ref{thm:main} the three conditions are equalities between morphisms
of the same category, so that they can be compared. In the one-way layout, the
failure of the marginal dispersion is detected by the insertions, i.e.\ by the
weaker criterion of Tjur (see Example~\ref{ex:oneway-findings}).

\subsection{Design-indexed quantities}

\begin{definition}\label{def:meaningful}
Let $\mathbf{P} \colon \mathbf{L} \Rightarrow \mathbf{R}$ be a bridged model
(Definition~\ref{def:bridged}) and let
$\Lambda \colon \catO^{\op} \to \Meas$ be a functor, called the \emph{target}.
The measurable space $\Lambda(\omega)$ is the space of values at $\omega$.
A \emph{design-indexed quantity} for $\mathbf P$ with target $\Lambda$ is a
family of morphisms of $\Stoch$
\[
g_{\omega} \colon \Theta_\Meas(\omega) \rightsquigarrow \Lambda(\omega),
\qquad \omega \in \ob\catO .
\]
It is called \emph{deterministic} if $g_{\omega} = \del f_{\omega}$ for
measurable maps
$f_{\omega} \colon \Theta_\Meas(\omega) \to \Lambda(\omega)$, i.e.\ if it is an
ordinary functional of the parameter. The quantity $g$ is called
\emph{meaningful} if it is $\catO$-coherent in the sense of
Definition~\ref{def:coherence} with $F = \del\,\Theta_\Meas$ and
$G = \del\,\Lambda$, i.e.\ if
\begin{equation}\label{eq:tjursquare}
g_{\omega} \circ \del\Theta_\Meas(\psi) \;=\; \del\Lambda(\psi) \circ
g_{\omega'}
\quad\text{in } \Stoch
\end{equation}
for every $\psi \colon \omega \to \omega'$ in $\catO$. It is called
\emph{Tjur-natural} if it is $\catO_{\mathrm{red}}$-coherent, i.e.\
if~\eqref{eq:tjursquare} holds for every $\psi$ in $\catO_{\mathrm{red}}$.
\end{definition}

Thus a meaningful quantity is a natural transformation
$\del\,\Theta_\Meas \Rightarrow \del\,\Lambda$ of functors
$\catO^{\op} \to \Stoch$. The functors are the same for the deterministic
quantities and for the general ones, and only the components differ.

\begin{example}\label{ex:targets}
In the one-way layout (Example~\ref{ex:oneway-setup}) we have
$\Theta(\omega) = \R^{\omega} \times \R_{>0}$ and
$\Theta(\psi)(\mu',\sigma^{2}) = (\mu' \circ \psi, \sigma^{2})$. The vector of
the group means $g_{\omega}(\mu,\sigma^{2}) = \mu$ has target
$\Lambda(\omega) = \R^{\omega}$ with $\Lambda(\psi) = \psi^{*}$, i.e.\ the
restriction $\mu' \mapsto \mu' \circ \psi$. The variance
$g_{\omega}(\mu,\sigma^{2}) = \sigma^{2}$ has the constant target
$\Lambda(\omega) = \R_{>0}$ with $\Lambda(\psi) = \id$. Both are meaningful.
Indeed~\eqref{eq:tjursquare} reads $\mu' \circ \psi = \psi^{*}\mu'$ in the
first case and $\sigma^{2} = \sigma^{2}$ in the second one.
\end{example}

\begin{remark}\label{rem:indexings}
The components of McCullagh in~\cite[\S4.5]{mccullagh_what_2002} are indexed
by the designs and are natural over all the design morphisms
$(\psi_{d},\psi_{c})$, while in Definition~\ref{def:meaningful} the quantities
are indexed by the objects of $\catO$. The two indexings agree on the families
which factor through $\pr_{\catO}$. Indeed, for these families the squares with
identity covariate leg are trivial. Furthermore, for every
$\psi \colon \omega \to \omega'$ and every design $(u,\omega,x)$ the pair
$(\id_{u},\psi) \colon (u,\omega,x) \to (u,\omega',\psi\circ x)$ is a morphism
of $\catD$, so that $\pr_{\catO}$ is surjective on the morphisms.
\end{remark}

\subsection{Tjur's criterion}\label{sec:tjur}

In this subsection $\cat{C} = \Set$, the functors $J_{\catU}$ and $J_{\catO}$
are the underlying-set functors; $\catO$ is a category of finite sets such as
$\catO_{\mathrm{lab}}$, and $\catU$ is a full subcategory of
$\FinSet_{\mathrm{inj}}$ which contains the underlying sets of the objects of
$\catO$, so that every inclusion of a subsample between such sets is a
morphism of $\catU$. The regression examples of Section~\ref{sec:worked},
whose covariate spaces are vector spaces, do not satisfy these hypotheses.

\begin{definition}\label{def:tjurreduction}
Let $(u,\omega,x)$ be a design with $x$ surjective. A \emph{Tjur reduction} of
$(u,\omega,x)$ is a morphism of $\catD$
\[
(\varphi,\psi) \colon (u',\omega',x') \longrightarrow (u,\omega,x)
\]
where $\varphi \colon u' \hookrightarrow u$ is the inclusion of a proper
subsample, $\omega' = x(\varphi(u'))$ is the surviving support, $x'$ is the
corestriction of $x \circ \varphi$, and
$\psi \colon \omega' \hookrightarrow \omega$ is the inclusion.
\end{definition}

A Tjur reduction is the removal of the sample points
$u \setminus \varphi(u')$ together with the covariate values which are no
longer observed, i.e.\ Tjur's ``removal of some of the
$x$'s''~\cite[p.~1299]{tjur_discussion_2002}. His condition at such a
reduction compares the parameter assigned to the marginal model with the
parameter of the bigger sample, the former being for him a function of the
marginal distribution. Since a bridged model is natural at $(\varphi,\psi)$,
see~\eqref{eq:ND}, the marginal law at the reduced design depends on $\theta$
only through $\Theta_\Meas(\psi)\theta$. The values at $\omega$ are carried to
$\omega'$ by $\Lambda(\psi)$, which is the identity in Tjur's setting, where
the scale is fixed. Hence the comparison is the square~\eqref{eq:tjursquare}
at $\psi$, and we say that $g$ \emph{satisfies Tjur's condition} at the
reduction if that square commutes.

\begin{lemma}\label{lem:tjurshadow}
A design-indexed quantity $g$ satisfies Tjur's condition at every Tjur
reduction of every design with surjective design map if and only if $g$ is
Tjur-natural.
\end{lemma}

\begin{proof}
Suppose that $g$ is Tjur-natural. The covariate leg $\psi$ of a Tjur reduction
is the inclusion of the surviving support. It is an insertion when the removal
empties some fibre of $x$, and it is the identity otherwise. The
square~\eqref{eq:tjursquare} holds in the first case by hypothesis and in the
second one trivially.

Conversely, it is enough to check~\eqref{eq:tjursquare} at each insertion
$\psi \colon \omega' \subsetneq \omega$, since such squares can be pasted along
the compositions and hold trivially at the identities. Consider the
tautological design $(\omega,\omega,\id_\omega)$, which has one unit for each
label and surjective design map, and the subsample
$\varphi = \psi \colon \omega' \hookrightarrow \omega$. It is a morphism of
$\catU$ because $\omega'$ and $\omega$ are objects of $\catO$, hence of
$\catU$, and $\catU$ is full in $\FinSet_{\mathrm{inj}}$. The surviving
support is $x(\varphi(\omega')) = \omega'$, so that
$(\psi,\psi) \colon (\omega',\omega',\id) \to (\omega,\omega,\id)$ is a Tjur
reduction with covariate leg $\psi$, and Tjur's condition at it is the
square~\eqref{eq:tjursquare} at $\psi$.
\end{proof}

The parameter functions of Tjur are indexed by the finite samples of covariate
values, with multiplicities, and not by the covariate spaces, while
Lemma~\ref{lem:tjurshadow} concerns quantities which are already indexed by the
support. We now remove this restriction.

\begin{definition}\label{def:sampleindexed}
A \emph{sample-indexed quantity} for $\mathbf P$ with target
$\Lambda \colon \catO^{\op} \to \Meas$ is a family of morphisms of $\Stoch$
\[
G_{d} \colon \Theta_\Meas(\omega) \rightsquigarrow \Lambda(\omega),
\qquad d = (u,\omega,x) \text{ a design with } x \text{ surjective},
\]
i.e.\ one morphism for each finite sample of covariate values, where the
sample is the pair $(u,x)$ and $\omega$ is its support. It is called
\emph{Tjur-meaningful} if
\begin{equation}\label{eq:sampletjur}
G_{d'} \circ \del\Theta_\Meas(\psi) \;=\; \del\Lambda(\psi) \circ G_{d}
\quad\text{in } \Stoch
\end{equation}
for every Tjur reduction $(\varphi,\psi) \colon d' \to d$. It is called
\emph{support-indexed} if $G_{(u,\omega,x)} = g_{\omega}$ for a design-indexed
quantity $g$, i.e.\ if it depends on the sample only through its support.
\end{definition}

\begin{example}\label{ex:sampleindexed}
Let $\Lambda$ be constant at $\R$. The sample size $G_{(u,\omega,x)} = |u|$ is
a sample-indexed quantity. In the simple regression model of
Example~\ref{ex:alphabetaxbar}, the covariate average weighted by the
multiplicities gives the sample-indexed quantity
\[
G_{(u,\omega,x)}(\alpha,\beta,\sigma^{2}) = \alpha + \beta\,|u|^{-1}
\textstyle\sum_{i \in u} x(i).
\]
These are the examples of Tjur~\cite[pp.~1299--1300]{tjur_discussion_2002},
and neither of them is support-indexed.
\end{example}

\begin{proposition}\label{prop:descent}
Suppose that $\catU$ contains a disjoint union $u \sqcup u'$ of any two of its
objects, as $\FinSet_{\mathrm{inj}}$ does. For a sample-indexed quantity $G$
with target $\Lambda$, the following are equivalent:
\begin{enumerate}
\item $G$ is Tjur-meaningful;
\item $G$ is support-indexed, i.e.\ $G_{(u,\omega,x)} = g_{\omega}$, and the
design-indexed quantity $g$ is Tjur-natural.
\end{enumerate}
The quantity $g$ of part~(2) is unique, namely
$g_{\omega} = G_{(\omega,\omega,\id_{\omega})}$. In particular a
Tjur-meaningful sample-indexed quantity depends neither on the sample size nor
on the multiplicities.
\end{proposition}

\begin{proof}
(2)~$\Rightarrow$~(1). Let $(\varphi,\psi) \colon d' \to d$ be a Tjur
reduction. The covariate leg $\psi$ is the inclusion of the surviving support,
i.e.\ an insertion or an identity, and~\eqref{eq:sampletjur} is the
square~\eqref{eq:tjursquare} at $\psi$. It holds in the first case because $g$
is Tjur-natural and in the second one trivially.

(1)~$\Rightarrow$~(2). Let $d = (u,\omega,x)$ and $d' = (u',\omega,x')$ be two
designs with the same support $\omega \neq \emptyset$, so that $u$ and $u'$
are non-empty. We set $d'' = (u \sqcup u',\omega,[x,x'])$, whose design map is
surjective because $x$ is. The two injections of the coproduct are inclusions
of proper subsamples into $u \sqcup u'$, the surviving support of each of them
is $\omega$, and the corestrictions are $x$ and $x'$. Hence
$(\iota_{u},\id_{\omega}) \colon d \to d''$ and
$(\iota_{u'},\id_{\omega}) \colon d' \to d''$ are Tjur reductions with
identity covariate leg. At such a reduction~\eqref{eq:sampletjur} reads
$G_{d} = G_{d''}$, since $\Theta_\Meas$ and $\Lambda$ map identities to
identities. Therefore $G_{d} = G_{d''} = G_{d'}$, i.e.\ $G$ depends on the
sample only through its support. For $\omega = \emptyset$ there is nothing to
prove. We set $g_{\omega} = G_{(\omega,\omega,\id_{\omega})}$, where the
tautological design exists because $\omega$ is an object of $\catU$. Let
$\psi \colon \omega' \subsetneq \omega$ be an insertion. The tautological
reduction $(\psi,\psi) \colon (\omega',\omega',\id) \to (\omega,\omega,\id)$
of the proof of Lemma~\ref{lem:tjurshadow} is a Tjur reduction with covariate
leg $\psi$, and~\eqref{eq:sampletjur} at it is the
square~\eqref{eq:tjursquare} for $g$ at $\psi$. Hence $g$ is Tjur-natural. The
uniqueness follows from the formula for $g_{\omega}$, and the last statement
follows from part~(2).
\end{proof}

Thus Tjur's criterion, imposed on the quantities that he considers, first
forces the indexing by the support and then reduces to Tjur-naturality in the
sense of Definition~\ref{def:meaningful}.

\begin{remark}\label{rem:distribution}
The parameter functions of Tjur are functions of the unknown
distribution~\cite[p.~1299]{tjur_discussion_2002}, while our quantities are
functions of the parameter. Every function of the distribution is a function
of the parameter, and the converse holds when
$\theta \mapsto \mathbf P_{v,d}(\,\cdot \mid \theta)$ is injective. We do not
pursue this distinction.
\end{remark}

\subsection{Meaningful and Tjur-natural quantities}

\begin{proposition}\label{prop:tjur}
Let $g$ be a design-indexed quantity.
\begin{enumerate}
\item Suppose that $g$ is deterministic, with $g_\omega = \del f_\omega$, and
that the $\sigma$-algebras of the spaces $\Lambda(\omega)$ separate points.
Then $g$ is meaningful if and only if the maps $f_\omega$ form a natural
transformation $\Theta_\Meas \Rightarrow \Lambda$ in $\Meas$.
\item If $g$ is meaningful then it is Tjur-natural. The converse holds when
$\catO_{\mathrm{red}} = \catO$ and it does not hold in general (see
Example~\ref{ex:tjurgap}).
\item Meaningfulness is preserved by the composition with the natural
transformations of targets $\Lambda \Rightarrow \Lambda'$ (in $\Meas$, or in
$\Stoch$ for the general quantities). It is also preserved by the pullback
along the functors $F \colon \catO' \to \catO$ such that
$J_{\catO'} = J_{\catO} \circ F$, i.e.\ by a change of the covariate scheme. In
this case $(u,\omega',x) \mapsto (u,F\omega',x)$ is a functor
$\catD' \to \catD$, along which the bridged model $\mathbf{P}$ is restricted.
\end{enumerate}
\end{proposition}

\begin{proof}
(1) Let $g = \del f$. Then~\eqref{eq:tjursquare} reads
$\del(f_\omega \circ \Theta_\Meas(\psi)) = \del(\Lambda(\psi) \circ
f_{\omega'})$. Since $\del$ is injective on the maps into $\Lambda(\omega)$ by
Lemma~\ref{lem:dirac}(1), this holds if and only if
$f_\omega \circ \Theta_\Meas(\psi) = \Lambda(\psi) \circ f_{\omega'}$.

(2) Tjur-naturality is~\eqref{eq:tjursquare} for the morphisms of
$\catO_{\mathrm{red}}$ only. This gives the first two statements. The quantity
of Example~\ref{ex:tjurgap} is Tjur-natural and not meaningful. Note that the
quantity of Example~\ref{ex:alphabetaxbar} fails already on
$\catO_{\mathrm{red}}$, so that it does not distinguish the two conditions.

(3) The first statement follows by pasting the squares. For the pullback,
$g_{F\omega'}$ is natural in $\omega'$ because $F$ maps the morphisms of
$\catO'$ to morphisms of $\catO$. The hypothesis $J_{\catO'} = J_{\catO}F$
gives that the restricted model is a bridged model over $\catO'$, since a
design $(u,\omega',x)$ for $\catO'$ has
$x \colon J_{\catU}u \to J_{\catO'}\omega' = J_{\catO}F\omega'$.
\end{proof}

\begin{example}\label{ex:alphabetaxbar}
We write the non-example of
Tjur~\cite[pp.~1299--1300]{tjur_discussion_2002} in our notation. Consider the
simple regression model with $\Theta(\omega) = \R^2 \times \R_{>0}$ and
coordinates $(\alpha,\beta,\sigma^2)$, let $\catO$ be the category of the
finite selections of covariate values, and let $\Lambda$ be constant at $\R$
with $\Lambda(\psi) = \id$. This is the unreplicated case, where the design
map is injective, so that sample points and covariate values coincide. The
general case, where $\bar x$ is weighted by the multiplicities, is the
sample-indexed quantity of Example~\ref{ex:sampleindexed}, which is excluded by
Proposition~\ref{prop:descent} before any comparison between support-indexed
quantities is made. We set
\[
g_{\omega}(\alpha,\beta,\sigma^2) \;=\; \alpha + \beta\,\bar x_{\omega},
\qquad
\bar x_{\omega} = |\omega|^{-1}\textstyle\sum_{x \in \omega} x .
\]
Each $g_\omega$ is a parameter of the model at $\omega$. Let
$\psi \colon \omega \subsetneq \omega'$ be an insertion. The parameter map
$\Theta(\psi)$ is the identity, while $\bar x_\omega \neq \bar x_{\omega'}$ in
general. Hence~\eqref{eq:tjursquare} reads
$\alpha + \beta\bar x_{\omega} = \alpha + \beta\bar x_{\omega'}$, and it fails
when $\beta \neq 0$. Since the failure occurs on $\catO_{\mathrm{red}}$, the
quantity is neither Tjur-natural nor meaningful. The same argument, where
$\Theta(\psi)$ is again the identity, applies to the correlation coefficient,
which depends on the design through the variance of the
covariate~\cite[\S6.2]{mccullagh_what_2002}. In this case it is the scale of
the design, and not its mean, which is not preserved. In the bridged setting,
both failures are equalities which do not hold.
\end{example}

\begin{example}\label{ex:tjurgap}
Let $\catO = \catO_{\mathrm{lab}}$ (Example~\ref{ex:oneway-setup}), let
$\Theta$ be any parameter functor, and let
$\Lambda(\omega) = \R^{\omega \times \omega}$ with
$\Lambda(\psi) = (\psi\times\psi)^{*}$, i.e.\
$\Lambda(\psi)(v)_{(a,b)} = v_{(\psi a,\,\psi b)}$. We set
$g_\omega(\theta)_{(a,b)} = [a \neq b]$, the indicator that two treatment
labels are distinct. The square~\eqref{eq:tjursquare} at
$\psi \colon \omega \to \omega'$ reads $[a \neq b] = [\psi a \neq \psi b]$ for
all $a,b \in \omega$, which holds if and only if $\psi$ is injective. Hence
$g$ is natural along every injection. In particular it is Tjur-natural, and it
satisfies the criterion of McCullagh as stated. On the other hand, $g$ is not
natural along any map which is not injective, for example along
$p \colon \{1,2\} \to \{*\}$. Thus the difference between Tjur-naturality and
meaningfulness occurs at the mergings.
\end{example}

\begin{proof}[Proof of Theorem~\ref{thm:main-tjur}]
It follows from Propositions~\ref{prop:descent} and~\ref{prop:tjur}(2) and
from Example~\ref{ex:tjurgap}.
\end{proof}

\begin{corollary}\label{cor:decidable}
Suppose that $\catO$ has finitely many morphisms, that
$\Theta_\Meas(\omega)$ and $\Lambda(\omega)$ are finite and discrete for every
$\omega$, and that the entries of the kernels $g_\omega$ belong to a subfield
$k \subseteq \R$ with decidable equality, for example $\mathbb{Q}$ or the
field of the real algebraic numbers. Then the meaningfulness of $g$ is
decidable. Indeed, it is equivalent to the finite conjunction, over the
morphisms $\psi \colon \omega \to \omega'$ of any generating set of $\catO$,
of the matrix equalities~\eqref{eq:tjursquare}, each of which is an equality
between two stochastic matrices of size
$|\Theta_\Meas(\omega')| \times |\Lambda(\omega)|$ with entries in $k$.
\end{corollary}

\begin{proof}
If~\eqref{eq:tjursquare} holds at $\psi$ and at $\psi'$ then it holds at
$\psi'\psi$, by pasting the two squares, since $\del\Theta_\Meas$ and
$\del\Lambda$ are functors. Hence it is enough to check a generating set,
which is finite, and each check is an equality between finitely many elements
of $k$, which is decidable by hypothesis.
\end{proof}
\section{A cohomological formulation of coherence for finite
designs}\label{sec:hochschild}

In this section we prove Theorem~\ref{thm:main-obstruction}. We suppose that
the design scheme is finite, and we show that the failure of naturality of a
quantity is a $1$-cochain of a complex of Baues and Wirsching. This links the
present framework to the diagrammatic Hochschild cohomology of Gerstenhaber
and Schack and to the cohomology of small categories of Baues and
Wirsching~\cite{gerstenhaber_schack_1983,baues_cohomology_1985}, as revisited
in~\cite[\S3]{caputi_diagrammatic_2025}. The cohomology controls the
obstructions to the $\catO$-coherence of the quantities
(Definition~\ref{def:coherence}).

\subsection{The diagram of algebras of a finite design scheme}

Let $\catO$ be a finite category of covariate spaces and suppose that the
parameter sets $\Theta(\omega)$ and the target sets $\Lambda(\omega)$ are
finite, as in Corollary~\ref{cor:decidable}. Let $k \supseteq \mathbb{Q}$ be a
field.

\begin{definition}\label{def:diagalg}
We denote by $\Alg_k$ the category of associative unital $k$-algebras and
unital homomorphisms. The \emph{parameter diagram of algebras} is the functor
\[
A \colon \catO \longrightarrow \Alg_k,
\qquad
A(\omega) \coloneqq k^{\,\Theta(\omega)},
\qquad
A(\psi) \coloneqq \Theta(\psi)^{*},
\]
where $k^{\,\Theta(\omega)}$ is the algebra of the functions on the parameter
set, with the pointwise operations, and $\Theta(\psi)^{*}$ is the
precomposition with $\Theta(\psi)$. It is covariant because $\Theta$ and
$k^{(-)}$ are both contravariant. In the same way a target $\Lambda$ gives the
functor $B(\omega) \coloneqq k^{\Lambda(\omega)}$.
\end{definition}

Let $g = (g_\omega \colon \Theta(\omega) \to \Lambda(\omega))$ be a
deterministic design-indexed quantity. It gives a family of homomorphisms of
algebras $g^{*} = (g_\omega^{*} \colon B(\omega) \to A(\omega))$, and $g$ is
meaningful (Definition~\ref{def:meaningful}) if and only if $g^{*}$ is a
natural transformation $B \Rightarrow A$. Indeed, all the spaces are finite and
discrete, so that meaningfulness is the naturality of the underlying maps by
Proposition~\ref{prop:tjur}(1). Furthermore $k^{(-)}$ is faithful, since a map
$f$ is recovered from $f^{*}$ through the indicator functions of its fibres.
Hence $g$ is natural if and only if $g^{*}$ is. For
$\psi \colon \omega \to \omega'$ we set
\begin{equation}\label{eq:defect}
c(\psi) \;\coloneqq\; A(\psi)\, g_{\omega}^{*} \;-\; g_{\omega'}^{*}\, B(\psi)
\;\in\; \Hom_k\bigl(B(\omega), A(\omega')\bigr),
\end{equation}
and we call $c = c_g$ the \emph{defect} of $g$. It measures the failure of
naturality of $g$ at $\psi$.

The defect is a $1$-cochain on $\catO$ in the sense of Baues and
Wirsching~\cite[(1.4)]{baues_cohomology_1985}, with coefficients in a natural
system $D \colon F\catO \to \mathrm{Ab}$ in their
sense~\cite[(1.2)]{baues_cohomology_1985}. Here $F\catO$ is the category of
factorisations of $\catO$, i.e.\ the twisted arrow category, denoted by
$\mathrm{Tw}\,\catO$ in~\cite{caputi_diagrammatic_2025}. Its objects are the
morphisms of $\catO$, and a morphism $\psi \to \psi'$ is a pair
$(\alpha,\beta)$ such that $\psi' = \alpha\psi\beta$. The natural system that
we use is obtained by pulling back the bimodule
$M(\omega,\omega') = \Hom_k(B(\omega),A(\omega'))$ along the forgetful functor
$F\catO \to \catO^{\op} \times \catO$~\cite[(1.16)--(1.18)]{baues_cohomology_1985}.
Thus $D(\psi) = \Hom_k(B(\omega),A(\omega'))$ for
$\psi \colon \omega \to \omega'$, and
$(\alpha,\beta) \colon \psi \to \alpha\psi\beta$ acts by
$f \mapsto A(\alpha)\,f\,B(\beta)$. With bimodule coefficients, this is the
Hochschild--Mitchell cohomology~\cite[(8.5)]{baues_cohomology_1985}. We write
\[
C^{0} = \prod\nolimits_\omega \Hom_k(B(\omega),A(\omega)),
\qquad
(d^{0}h)(\psi) = A(\psi)h_\omega - h_{\omega'}B(\psi),
\]
for the group of the $0$-cochains and for the differential. The equations
$d^{0}h = 0$, one for each $\psi$, define the
end~\cite[IX.5]{maclane_categories_1998}
\begin{equation}\label{eq:H0}
\Nat(B,A) \;=\; \int_{\omega} \Hom_k\bigl(B(\omega),A(\omega)\bigr)
\;=\; \ker d^{0} \;=\; H^{0},
\end{equation}
and the defect~\eqref{eq:defect} of a quantity $g$ is the coboundary
$c_g = d^{0}g^{*}$. By~\cite[(5.3)--(5.6)]{baues_cohomology_1985} the group
$H^{1}$ is the group of the derivations of $\catO$ into $D$, i.e.\ of the
$1$-cochains $c$ such that
$c(\psi'\psi) = A(\psi')\,c(\psi) + c(\psi')\,B(\psi)$, modulo the inner
derivations $d^{0}h$. We write $Z^{1}$ and $B^{1}$ for the groups of the
$1$-cocycles and of the $1$-coboundaries.

\begin{remark}\label{rem:GS}
When $B = A$, the complex above is the row of Hochschild degree $1$ of the
Gerstenhaber--Schack bicomplex of the diagram
$A$~\cite[Prop.~3.8]{caputi_diagrammatic_2025}. This bicomplex is defined over
a poset in~\cite{gerstenhaber_schack_1983}, for diagrams over a small category
as in~\cite[\S2]{gerstenhaber_simplicial_1983}, and over a general finite
category in~\cite[eq.~(4)]{caputi_diagrammatic_2025}. Here $A$ is read as a
presheaf on $\catO^{\op}$, which is the convention
of~\cite{caputi_diagrammatic_2025}, and the Baues--Wirsching complexes of
$\catO$ and of $\catO^{\op}$ with the transported natural system are
isomorphic up to sign, since $F(\catO^{\op}) \cong F\catO$. The group $H^{1}$
computed below is the cohomology of that row for its horizontal
differential, which is the Baues--Wirsching one, i.e.\ it lives in bidegree
$(1,1)$ of the bicomplex. It is not the Gerstenhaber--Schack cohomology in
total degree~$2$, and no vertical (Hochschild) differential is used here. No
result of~\cite{caputi_diagrammatic_2025} enters a proof in this paper. That
work is cited for the identification just recalled, which concerns the case
$B = A$ and not the case computed below, and for the convention on the twisted
arrow category. The relation of the bicomplex to the Baues--Wirsching
cohomology is controlled by a spectral sequence attributed to
Robinson~\cite{robinson_cohomology_2008}
in~\cite[\S1, Prop.~3.10 and Thm.~3.11]{caputi_diagrammatic_2025}. That
sequence takes the vertical differential first, so that its $E_2$ page has
Hochschild cohomology groups as coefficients, and we do not use it here.
\end{remark}

\begin{lemma}\label{lem:cocycle}
The cochain $c$ of~\eqref{eq:defect} satisfies the $1$-cocycle identity
\[
c(\psi'\psi) = A(\psi')\,c(\psi) + c(\psi')\,B(\psi)
\]
for all composable $\psi,\psi'$, and $c = 0$ if and only if $g$ is
meaningful.
\end{lemma}

\begin{proof}
We have $c = d^{0}g^{*}$ by~\eqref{eq:defect}, and the identity is
$d^{1}d^{0} = 0$. Explicitly, since $A$ and $B$ are functors, adding and
subtracting the term $A(\psi')g^{*}_{\omega'}B(\psi)$ we obtain
\[
A(\psi'\psi)g^{*}_{\omega} - g^{*}_{\omega''}B(\psi'\psi)
= A(\psi')\bigl(A(\psi)g^{*}_{\omega} - g^{*}_{\omega'}B(\psi)\bigr)
+ \bigl(A(\psi')g^{*}_{\omega'} - g^{*}_{\omega''}B(\psi')\bigr)B(\psi).
\]
Finally $c = 0$ means that $g^{*}$ is natural, which holds if and only if $g$
is natural, as we have seen after Definition~\ref{def:diagalg}.
\end{proof}

Let $g$ be a deterministic design-indexed quantity, given on the objects of
$\catO$ only. We call it a \emph{raw family}. An \emph{admissible class} is a
subset $\mathcal{C} \subseteq C^{0}$, and a \emph{correction} of $g$ in
$\mathcal{C}$ is an element $h \in \mathcal{C}$ such that $g^{*} + h$ is
natural.

\begin{proposition}\label{wprop:obstruction}
Let $g$ be a raw family and let $\mathcal{C} \subseteq C^{0}$ be an admissible
class.
\begin{enumerate}
\item There is a correction of $g$ in $\mathcal{C}$ if and only if
$c_g = -\,d^{0}h$ for some $h \in \mathcal{C}$.
\item We have $c_g \in B^{1}$ for every raw family, hence $[c_g] = 0$ in
$H^{1}$ with no hypothesis. If $\mathcal{C} = C^{0}$ then $h = -g^{*}$ is
always a correction.
\item Suppose that $\mathcal{C}$ is a linear subspace. The obstruction to the
existence of a correction of $g$ in $\mathcal{C}$ is the class of $g^{*}$ in
$C^{0}/(\mathcal{C} + \Nat(B,A))$. The differential $d^{0}$ identifies this
quotient with $B^{1}/d^{0}(\mathcal{C})$, i.e.\ with the cokernel of $d^{0}$
restricted to $\mathcal{C}$ and corestricted to $B^{1}$. The image of the
obstruction in $H^{1} = Z^{1}/B^{1}$ is zero.
\end{enumerate}
\end{proposition}

\begin{proof}
Let $h \in C^{0}$. Then $g^{*}+h$ is natural if and only if
$d^{0}(g^{*}+h) = 0$, i.e.\ if and only if $c_g = -\,d^{0}h$. This gives
part~(1). Part~(2) follows from $c_g = d^{0}g^{*}$ and from the case
$h = -g^{*}$. For part~(3), we have $c_g = -\,d^{0}h$ with
$h \in \mathcal{C}$ if and only if
$g^{*} \in \mathcal{C} + \ker d^{0} = \mathcal{C} + \Nat(B,A)$. The map $d^{0}$
induces an isomorphism $C^{0}/\ker d^{0} \cong B^{1}$ which carries
$\mathcal{C} + \ker d^{0}$ onto $d^{0}(\mathcal{C})$, and $B^{1}$ is the
denominator of $H^{1}$.
\end{proof}

\begin{remark}\label{rem:obstruction}
A $0$-cochain is a family of linear maps. Hence $g^{*}+h$ is the dual of a
design-indexed quantity only when it is a family of homomorphisms of algebras,
and this has to be ensured by $\mathcal{C}$. The class of the recalibrations
of Example~\ref{ex:oneway-findings} does so by construction. By part~(2) an
admissible class must be a proper subset of $C^{0}$ in order to carry an
obstruction. By part~(3) the group $H^{1}(\catO;\Hom_k(B,A))$ is not an
obstruction group for the quantities. It is an invariant of the design scheme
and of the target $\Lambda$, and Example~\ref{ex:oneway-h1} gives a case where
it is not zero. All the computations are finite because the complex is finite
dimensional in each degree. For a linear subspace $\mathcal{C}$, they are
linear algebra, and for a finite class which is not linear, such as the
recalibrations, they are combinatorics (Proposition~\ref{prop:recalib}). Both
$c_g$ and the correction $h$, when it exists, are computed exactly.
\end{remark}

The obstruction of part~(3) becomes a cohomology class when the complex is
taken relative to the admissible class.

\begin{proposition}\label{prop:relative}
Let $C^{\bullet} = C^{\bullet}(\catO;D)$ be the Baues--Wirsching complex
introduced above and let $\mathcal{C} \subseteq C^{0}$ be a linear subspace.
Let $C^{\bullet}_{\mathcal{C}} \subseteq C^{\bullet}$ be the subcomplex with
$C^{0}_{\mathcal{C}} = \mathcal{C}$ and $C^{n}_{\mathcal{C}} = C^{n}$ for
$n \ge 1$. Then
\[
H^{0}(C^{\bullet}_{\mathcal{C}}) = \mathcal{C} \cap \Nat(B,A),
\qquad
H^{1}(C^{\bullet}_{\mathcal{C}}) = Z^{1}/d^{0}(\mathcal{C}),
\qquad
H^{n}(C^{\bullet}_{\mathcal{C}}) = H^{n}(\catO;\Hom_k(B,A))
\ \text{ for } n \ge 2,
\]
and the short exact sequence of complexes
$0 \to C^{\bullet}_{\mathcal{C}} \to C^{\bullet} \to C^{0}/\mathcal{C} \to 0$,
where the quotient is concentrated in degree~$0$, gives the exact sequence
\begin{equation}\label{eq:relativeLES}
0 \to \mathcal{C} \cap \Nat(B,A) \to \Nat(B,A) \to C^{0}/\mathcal{C}
\xrightarrow{\ \delta\ } H^{1}(C^{\bullet}_{\mathcal{C}}) \to
H^{1}(\catO;\Hom_k(B,A)) \to 0
\end{equation}
with connecting map $\delta[h] = [d^{0}h]$. Furthermore:
\begin{enumerate}
\item for a raw family $g$ we have $\delta[g^{*}] = [c_g]_{\mathcal{C}}$, the
class of the defect in $H^{1}(C^{\bullet}_{\mathcal{C}})$, and $g$ has a
correction in $\mathcal{C}$ if and only if $[c_g]_{\mathcal{C}} = 0$. The
correction is unique modulo $\mathcal{C} \cap \Nat(B,A)$, hence it is unique
when $\mathcal{C} \cap \Nat(B,A) = 0$;
\item the group $B^{1}/d^{0}(\mathcal{C})$ of
Proposition~\ref{wprop:obstruction}(3) is the image of $\delta$, which is the
kernel of the comparison map
$H^{1}(C^{\bullet}_{\mathcal{C}}) \to H^{1}(\catO;\Hom_k(B,A))$. The absolute
group is the cokernel of $\delta$, i.e.\ the part of
$H^{1}(C^{\bullet}_{\mathcal{C}})$ which is not reached by any quantity.
\end{enumerate}
\end{proposition}

\begin{proof}
Note that $C^{\bullet}_{\mathcal{C}}$ is a subcomplex of $C^{\bullet}$ because
$d^{0}(\mathcal{C}) \subseteq C^{1}$ and its terms of positive degree are those
of $C^{\bullet}$. The quotient is $C^{0}/\mathcal{C}$ in degree~$0$ and zero
elsewhere, with zero differential. The cohomology of the subcomplex is
$\mathcal{C} \cap \ker d^{0}$ in degree~$0$, $Z^{1}/d^{0}(\mathcal{C})$ in
degree~$1$ and $H^{n}(C^{\bullet})$ for $n \ge 2$, and
$\ker d^{0} = \Nat(B,A)$ by~\eqref{eq:H0}. The long exact sequence of the
short exact sequence of complexes is~\eqref{eq:relativeLES}, since the
cohomology of the quotient is zero in positive degrees. Its connecting map
lifts $[h] \in C^{0}/\mathcal{C}$ to $h \in C^{0}$ and applies $d^{0}$, whose
values belong to $C^{1} = C^{1}_{\mathcal{C}}$, hence
$\delta[h] = [d^{0}h]$. The exactness at $C^{0}/\mathcal{C}$ says that
$\delta[h] = 0$ if and only if $h \in \mathcal{C} + \Nat(B,A)$. The exactness
at $H^{1}(C^{\bullet}_{\mathcal{C}})$ says that the image of $\delta$, which
is $B^{1}/d^{0}(\mathcal{C})$, is the kernel of the comparison map, which is
surjective because $Z^{1}$ maps onto $Z^{1}/B^{1}$. This gives part~(2).
Part~(1) follows from the exactness at $C^{0}/\mathcal{C}$ applied to
$h = g^{*}$, since $d^{0}g^{*} = c_g$, and from
Proposition~\ref{wprop:obstruction}(3). If $h,h' \in \mathcal{C}$ are two
corrections of the same $g$ then $g^{*}+h$ and $g^{*}+h'$ are natural, hence
$h - h' \in \mathcal{C} \cap \Nat(B,A)$.
\end{proof}

\subsection{The one-way layout as test case}

We now compute the objects introduced above for the one-way layout. We first
give the setting (Example~\ref{ex:oneway}), then we describe what is computed
and how, and finally we give the results
(Propositions~\ref{prop:canonical}, \ref{prop:recalib}
and~\ref{prop:oneway-cont-recalib},
Examples~\ref{ex:oneway-findings} and~\ref{ex:oneway-h1}).

\begin{example}\label{ex:oneway}
We restrict Example~\ref{ex:oneway-setup} to two covariate spaces, namely
$\omega_2 = \{1,2\}$ (two treatment groups) and $\omega_1 = \{*\}$ (the groups
merged). We denote by $\catO_{\mathrm{arr}}$ the category with these two
objects and with the merge $p \colon \omega_2 \to \omega_1$ as the only
morphism different from the identities. We denote by $\catO_{\mathrm{full}}$
the full subcategory of $\FinSet$ on the same two objects. Its morphisms
different from the identities are $p$, the injections
$i_1, i_2 \colon \omega_1 \to \omega_2$ (selection of a group), the swap
$\sigma$ (relabelling of the treatments) and the constants $c_k = i_k p$. Thus
$\catO_{\mathrm{full}}$ has eight morphisms, with $p\,i_k = \id$ and
$\sigma i_1 = i_2$. We consider both categories because the merge alone does
not detect the failure that we want to study (see
Example~\ref{ex:oneway-findings}(i)).

We identify $\omega_1 = \{*\}$ with $\{1\}$. Then $i_1$ is the insertion
$\{1\} \subsetneq \{1,2\}$, which corresponds by
Lemma~\ref{lem:tjurshadow} to Tjur's reduction to the first group, and
$i_2 = \sigma i_1$ is that insertion followed by a relabelling. The reduction
to the second group is the insertion $\{2\} \subsetneq \{1,2\}$, which is
represented by $i_2$ up to the same identification. Neither $\{1\}$ nor
$\{2\}$ is an object of $\catO_{\mathrm{full}}$, so that the insertions of
Section~\ref{sec:meaningful} are represented in $\catO_{\mathrm{full}}$ and
not contained in it, and ``the insertion $i_k$'' below is always meant in
this sense. By the cocycle identity of Lemma~\ref{lem:cocycle} we have
$c(i_2) = A(\sigma)c(i_1) + c(\sigma)$, hence nothing depends on the
identification that we choose.

We discretise the parameters. Let $\Theta(\omega) = M^{\omega} \times S$, where
$M$ and $S$ are finite sets of admissible means and dispersions. Then
$\Theta(p)(\mu,s) = ((\mu,\mu),s)$ is the diagonal,
$\Theta(i_k)((\mu_1,\mu_2),s) = (\mu_k,s)$ is the $k$-th coordinate and
$\Theta(\sigma)$ is the swap. Hence
\[
A(\omega_1) = k^{M \times S}, \qquad
A(\omega_2) = k^{M^{2} \times S}, \qquad
A(\psi) = \Theta(\psi)^{*},
\]
so that $A(p)$ is the restriction to the diagonal and $A(i_k)f$ is $f$ as a
function of the $k$-th coordinate. We take as $g$ the \emph{marginal}, or
total, dispersion under balanced allocation,
\[
g_{\omega_1}(\mu,s) = s,
\qquad
g_{\omega_2}\bigl((\mu_1,\mu_2),s\bigr) = s + \tfrac{(\mu_1-\mu_2)^2}{4}.
\]
It is the analogue for the one-way layout of Tjur's $\alpha + \beta\bar x$ of
Example~\ref{ex:alphabetaxbar}. Like that quantity, it is a parameter to which
a term is added that depends on the groups included in the design. The target
$\Lambda$ is constant at the finite set
$T \coloneqq g_{\omega_2}\bigl(\Theta(\omega_2)\bigr) \subset \mathbb{Q}$ of
the values of $g$, which contains $S = g_{\omega_1}(\Theta(\omega_1))$, with
$\Lambda(\psi) = \id_{T}$. Hence $B(\omega) = k^{T}$ for both objects and
$B(\psi) = \id$.

The Baues--Wirsching complex is finite dimensional. We use the normalised
subcomplex, which has the same cohomology as the complex
of~\cite[(1.4)]{baues_cohomology_1985}, as for every cosimplicial abelian
group. For both categories
\[
C^{0} = \Hom_k(B(\omega_1), A(\omega_1)) \oplus \Hom_k(B(\omega_2),
A(\omega_2)),
\qquad
(d^{0}h)(\psi) = A(\psi)h_\omega - h_{\omega'}B(\psi).
\]
The group $C^{1}_{\mathrm{norm}}$ has one component over
$\catO_{\mathrm{arr}}$, namely $\Hom_k(B(\omega_2),A(\omega_1))$ at
$\psi = p$, and six components over $\catO_{\mathrm{full}}$, one for each of
$p, i_1, i_2, \sigma, c_1, c_2$. The differential $d^{1}$ is the operator
whose vanishing is the cocycle identity of Lemma~\ref{lem:cocycle}, and
$H^{1} = \ker d^{1}/\operatorname{im} d^{0}$. Over $\catO_{\mathrm{arr}}$ no
two morphisms different from the identities can be composed, hence
$d^{1} = 0$ and $H^{1}$ is the cokernel of $d^{0}$. Thus everything is a
finite computation with matrices once $M$, $S$ and $g$ are fixed.
\end{example}

\subsubsection*{What is computed}

The computations are carried out by two scripts, in exact arithmetic over
$\mathbb{Q}$.\footnote{Scripts and committed outputs are available at
\url{https://github.com/Vaxx66/calamus}, in the release tagged
\texttt{v37}, which also contains the version of this manuscript to which
the numbers quoted below refer; a fresh run reproduces the committed outputs
byte for byte. The repository describes the outputs of the scripts in
detail.} Given one of the two categories, the finite sets $M$ and $S$ and the
pair $(g_{\omega_1}, g_{\omega_2})$, the first script computes the defect $c$
of~\eqref{eq:defect}, the dimensions of $C^{1}$, $Z^{1}$, $B^{1}$ and $H^{1}$,
a particular solution $h$ of $c = -\,d^{0}h$ with the corrected quantity, and
the partition of $T$ of Proposition~\ref{prop:recalib}. It is run five times.
Run~1 has $\catO_{\mathrm{arr}}$, $M = \{0,1\}$ and $S = \{1,2\}$. Run~2 has
$\catO_{\mathrm{full}}$ with the same $M$ and $S$. Run~3 has
$\catO_{\mathrm{full}}$, $M = \{0,1\}$ and $S = \{1,\tfrac54\}$. Check~A is
run~1 with $g_{\omega_2}$ replaced by
$g'_{\omega_2}((\mu_1,\mu_2),s) = s + \mu_1/4$, which does not agree with
$g_{\omega_1}$ on the diagonal. Check~B has $\catO_{\mathrm{full}}$,
$M = \{0,1,2\}$ and $S = \{1,2\}$. The second script computes $H^{1}$ again in
an independent way, through the reduced coefficient system of
Example~\ref{ex:oneway-h1}. All the numbers quoted below are read from the
committed outputs. The statements which hold for all $M$ and $S$ are proved,
and the runs are instances of them.

\subsubsection*{Results}

We first identify the within-group dispersion as the correction of the
marginal dispersion, which is canonical with respect to a given class.

\begin{proposition}\label{prop:canonical}
In the setting of Example~\ref{ex:oneway}, over $\catO_{\mathrm{full}}$ and for
any $M$ and $S$, let
\[
\mathcal{C}_{1} \;\coloneqq\; \{\,h \in C^{0} : h_{\omega_1} = 0\,\}
\]
be the class of corrections which leave the component at the merged design
unchanged. Then $\mathcal{C}_{1} \cap \Nat(B,A) = 0$, so that a correction in
$\mathcal{C}_{1}$ is unique when it exists. A deterministic raw family
$g = (g_{\omega_1},g_{\omega_2})$ with target constant at $T$ has a correction
in $\mathcal{C}_{1}$ if and only if $g_{\omega_1}$ does not depend on the
mean, i.e.\ if and only if
$g_{\omega_1} \circ \Theta(i_1) = g_{\omega_1} \circ \Theta(i_2)$. In this case,
the corrected family is deterministic, with component
$g_{\omega_1} \circ \Theta(i_1)$ at $\omega_2$. For the marginal dispersion, the
corrected component at $\omega_2$ is the within-group dispersion
$((\mu_1,\mu_2),s) \mapsto s$.
\end{proposition}

\begin{proof}
Since $\Lambda$ is constant we have $B(\psi) = \id$, and the equations
$d^{0}h = 0$ at $i_1$ and $i_2$ read $h_{\omega_2} = A(i_k)h_{\omega_1}$.
Hence a natural $0$-cochain is determined by its component at $\omega_1$, so
that $\mathcal{C}_{1} \cap \Nat(B,A) = 0$, and the uniqueness follows from
Proposition~\ref{prop:relative}(1). Let $h \in \mathcal{C}_{1}$ be a
correction. The natural family $g^{*}+h$ has component $g^{*}_{\omega_1}$ at
$\omega_1$, hence its component at $\omega_2$ is
$A(i_1)g^{*}_{\omega_1} = A(i_2)g^{*}_{\omega_1}$. This is the stated
condition, because $k^{(-)}$ is faithful and
$A(i_k)g^{*}_{\omega_1} = (g_{\omega_1}\circ\Theta(i_k))^{*}$. Conversely,
suppose that the condition holds and set $n_{\omega_1} = g^{*}_{\omega_1}$ and
$n_{\omega_2} = A(i_1)g^{*}_{\omega_1}$. The equations $d^{0}n = 0$ hold at
the generators $i_1, i_2, p, \sigma$. Indeed they hold at $i_1$ by definition,
at $i_2$ by the condition, at $p$ because $A(p)A(i_1) = A(p\,i_1) = \id$, and
at $\sigma$ because $A(\sigma)A(i_1) = A(\sigma i_1) = A(i_2)$. Hence $n$ is
natural, $h = n - g^{*}$ belongs to $\mathcal{C}_{1}$, and $n$ is the dual of
the deterministic family $(g_{\omega_1},\, g_{\omega_1}\circ\Theta(i_1))$. For
the marginal dispersion $g_{\omega_1}(\mu,s) = s$ does not depend on $\mu$,
and $g_{\omega_1}\circ\Theta(i_1)$ is $((\mu_1,\mu_2),s) \mapsto s$.
\end{proof}

\begin{proof}[Proof of Theorem~\ref{thm:main-obstruction}]
The first two statements are Proposition~\ref{prop:relative}(1) and~(2). The
statements on the one-way layout are Proposition~\ref{prop:canonical}.
\end{proof}

We now consider the \emph{recalibrations} of $g$, i.e.\ the families
$r \circ g = (r_{\omega_1}g_{\omega_1},\, r_{\omega_2}g_{\omega_2})$, where
$r = (r_{\omega_1},r_{\omega_2})$ is a pair of maps $T \to T$. A recalibration
expresses the quantity on a different scale of values.

\begin{proposition}\label{prop:recalib}
In the setting of Example~\ref{ex:oneway}, let $\sim$ be the equivalence
relation on $T$ generated by $s \sim s + (\mu_1-\mu_2)^{2}/4$ for $s \in S$
and $\mu_1,\mu_2 \in M$. The recalibration $r \circ g$ is natural over
$\catO_{\mathrm{full}}$ if and only if $r_{\omega_2}$ is constant on each
class of $\sim$ and $r_{\omega_1} = r_{\omega_2}$ on $S$. Furthermore:
\begin{enumerate}
\item there is a natural recalibration with $r_{\omega_2}$ not constant if and
only if $\sim$ has at least two classes;
\item if $|M| \ge 2$ then no natural recalibration has $r_{\omega_2}$
injective, i.e.\ the marginal dispersion cannot be expressed on another scale
in a natural way without loss of information;
\item over $\catO_{\mathrm{arr}}$ the recalibration $r \circ g$ is natural if
and only if $r_{\omega_1} = r_{\omega_2}$ on $S$.
\end{enumerate}
\end{proposition}

\begin{proof}
The square~\eqref{eq:tjursquare} at $i_k$ reads
$r_{\omega_1}(g_{\omega_1}(\Theta(i_k)\theta)) =
r_{\omega_2}(g_{\omega_2}(\theta))$ for all $\theta = ((\mu_1,\mu_2),s)$, i.e.\
$r_{\omega_1}(s) = r_{\omega_2}\bigl(s + (\mu_1-\mu_2)^{2}/4\bigr)$, and these
are the same equations for $k = 1,2$. The case $\mu_1 = \mu_2$ gives
$r_{\omega_1} = r_{\omega_2}$ on $S$. Substituting, we obtain
$r_{\omega_2}(s) = r_{\omega_2}(s + (\mu_1-\mu_2)^{2}/4)$ for every generating
pair of $\sim$, i.e.\ $r_{\omega_2}$ is constant on the classes, since every
element of $T$ is of the form $s + (\mu_1-\mu_2)^{2}/4$. Conversely, these two
conditions give the squares at $i_1$ and $i_2$. The square at $p$ is the case
$\mu_1 = \mu_2$, the square at $\sigma$ holds for every $r$ because
$g_{\omega_2}$ is symmetric, and the squares at $c_k = i_k p$ follow by
pasting. Part~(1) holds because $r_{\omega_2}$ is not constant if and only if
it separates two classes. For part~(2), let $\mu_1 \neq \mu_2$ in $M$ and
$s \in S$. Then $s$ and $s + (\mu_1-\mu_2)^{2}/4$ are distinct elements of the
same class, on which an injective $r_{\omega_2}$ cannot be constant. Part~(3)
is the square at $p$ alone.
\end{proof}

\begin{example}\label{ex:oneway-findings}
\emph{(i) The merge alone detects nothing (run~1 and check~A).} Over
$\catO_{\mathrm{arr}}$ the marginal dispersion is natural, i.e.\ its defect
is zero, because the between-group term $(\mu_1-\mu_2)^2/4$ vanishes on the
diagonal $\Theta(p)$. This is a property of the quantity and not of the
category. Indeed, in check~A the same category is given a quantity which does
not agree with $g_{\omega_1}$ on the diagonal. The script finds a defect
different from zero at $p$ and corrects it, and the particular solution leaves
$g'_{\omega_2}$ unchanged and replaces the component at the merged design with
$(\mu,s) \mapsto s + \mu/4$. The cohomology detects nothing either, for every
$M$ and $S$. Indeed $B(p) = \id$ gives that $d^{0}$ is surjective onto
$C^{1}_{\mathrm{norm}} = \Hom_k(B(\omega_2),A(\omega_1))$, and $d^{1} = 0$. In
run~1 we have $\dim C^{1} = \dim Z^{1} = \dim B^{1} = 16$ and $H^{1} = 0$.

\emph{(ii) The quantity is not meaningful, and a deterministic correction is
found (run~2).} Over $\catO_{\mathrm{full}}$ with $M = \{0,1\}$ and
$S = \{1,2\}$, so that $T = \{1,\tfrac54,2,\tfrac94\}$, the defect is
different from zero on $\{i_1,i_2,c_1,c_2\}$ and it is zero on $p$ and
$\sigma$. The cocycle identity holds entry by entry on all the $22$ composable
pairs. Thus the failure is detected by the selections $i_1, i_2$ of a group and
not by the merge, i.e.\ it is detected by the criterion of Tjur, which
quantifies over the removals of sample points only. The defect at the
insertion $i_1$ is a violation of his condition at the reduction to the first
group (Lemma~\ref{lem:tjurshadow}). Here $\dim C^{1} = 176$,
$\dim Z^{1} = \dim B^{1} = 40$ and $H^{1} = 0$. The group vanishes although the
defect does not, in agreement with Proposition~\ref{wprop:obstruction}(2), and
its vanishing gives no information on the quantity. Solving $c = -\,d^{0}h$
over the whole group $C^{0}$, which is always possible by the same part, we
find a corrected family which is again deterministic. The solution obtained by
setting to zero the free parameters of the row-reduced system has the
\emph{within-group} dispersion $((\mu_1,\mu_2),s) \mapsto s$ as component at
$\omega_2$, i.e.\ the between-group term is removed. The set of all the
solutions is the affine space $\,-g^{*} + \Nat(B, A)$, which contains other
deterministic elements, so that this solution is not canonical among all the
corrections. It is canonical with respect to the class $\mathcal{C}_{1}$ of
Proposition~\ref{prop:canonical}. Indeed, a natural family over
$\catO_{\mathrm{full}}$ is determined by its component at $\omega_1$, so that a
corrected family whose component at $\omega_2$ is the within-group dispersion
leaves $g_{\omega_1}$ unchanged, i.e.\ the correction belongs to
$\mathcal{C}_{1}$. Thus the statistically sensible correction is the unique
one which does not change the design on which $g$ is already a within-group
dispersion.

\emph{(iii) An obstruction different from zero with respect to an admissible
class (runs~2, 3 and check~B).} We take as admissible the recalibrations
$r \circ g$ with $r_{\omega_2}$ not constant, so that the quantity remains a
non-trivial function of the observed marginal dispersion. The statistically
natural class is the smaller one of the injective recalibrations, which
preserve the information, and by Proposition~\ref{prop:recalib}(2) it contains
no natural element for any grid with $|M| \ge 2$. We keep the larger class
because it is the weakest one which can carry an obstruction, so that an
obstruction with respect to it is not a consequence of injectivity. By
Proposition~\ref{prop:recalib}, naturality forces $r_{\omega_2}$ to be constant
on the classes of the partition of $T$ generated by
$s \sim s + (\mu_1-\mu_2)^2/4$, and the existence of a natural recalibration is
decided by computing this partition. We say that the design is
\emph{scale-confounded} if some marginal value $s + (\mu_1-\mu_2)^2/4$ with
$\mu_1 \neq \mu_2$ is itself an admissible within-group dispersion, i.e.\ it
belongs to $S$. With $M = \{0,1\}$ and $S = \{1,\tfrac54\}$ (run~3, where
$T = \{1,\tfrac54,\tfrac32\}$), or with $S = \{1,2\}$ and three levels of the
mean $M = \{0,1,2\}$ (check~B), the partition has a single class. Hence every
natural recalibration is constant, and the obstruction with respect to this
class is not zero, i.e.\ no recalibration of the marginal dispersion which
carries information is natural. With $M = \{0,1\}$ and $S = \{1,2\}$ as
in~(ii) the partition has the two classes $\{1,\tfrac54\}$ and
$\{2,\tfrac94\}$ (run~2), and there is a natural recalibration which is not
constant, but no bijective one. In runs~2, 3 and in check~B the unrestricted
correction of~(ii) gives the within-group dispersion, i.e.\ the correction in
$\mathcal{C}_{1}$ of Proposition~\ref{prop:canonical}.
\end{example}

The finite runs discretise a statement which holds in the Gaussian layout
itself, where it is simpler.

\begin{proposition}[Continuous case]\label{prop:oneway-cont-recalib}
Let $\Theta(\omega) = \R^{\omega} \times \R_{>0}$ as in
Example~\ref{ex:oneway-setup}, let $g$ be the marginal dispersion, i.e.\
$g_{\omega_1}(\mu,s) = s$ and
$g_{\omega_2}((\mu_1,\mu_2),s) = s + (\mu_1-\mu_2)^{2}/4$, with target constant
at $T = \R_{>0} = g_{\omega_2}(\Theta(\omega_2))$, and let
$r_{\omega_1}, r_{\omega_2} \colon T \to T$ be any maps. If $r \circ g$
satisfies Tjur's condition at the reduction to the first group, i.e.\ if the
square~\eqref{eq:tjursquare} at $i_1$ commutes, then
$r_{\omega_1} = r_{\omega_2}$ is constant. Hence no recalibration of the
marginal dispersion, which is not constant, is Tjur-natural, and a fortiori none
is meaningful.
\end{proposition}

\begin{proof}
The square at $i_1$ reads
$r_{\omega_1}(s) = r_{\omega_2}\bigl(s + (\mu_1-\mu_2)^{2}/4\bigr)$ for all
$s > 0$ and all real $\mu_1,\mu_2$. For $\mu_1 = \mu_2$ it gives
$r_{\omega_1} = r_{\omega_2}$. For $0 < s < s'$, choosing
$\mu_1 - \mu_2 = 2\sqrt{s'-s}$ we obtain
$r_{\omega_2}(s) = r_{\omega_2}(s')$.
\end{proof}

Note that no measurability is used and that the class of the recalibrations is
not restricted in advance. In the continuous case every admissible marginal
value is confounded with every within-group dispersion. The partition of $T$
in the finite runs discretises this argument. Its two classes in
run~2 depend on the grid, where the only available step is
$(1-0)^{2}/4 = \tfrac14$ and $\tfrac54 \notin S$ interrupts the chain from $1$
to $2$. A third level of the mean (check~B) or a scale-confounded $S$ (run~3)
restores the connection. Thus the finite computation is an exact certificate
for the given grid, and the continuous statement says what the grids
approximate.

\begin{example}\label{ex:oneway-h1}
We compute $H^{1}(\catO;\Hom_k(B,A))$, see
Proposition~\ref{wprop:obstruction}(3). For $\catO_{\mathrm{arr}}$ it is zero
in every run, as it must be by
Example~\ref{ex:oneway-findings}(i). For $\catO_{\mathrm{full}}$ it is zero
with two levels of the mean (runs~2 and~3), and it is not zero with three
levels. With $M = \{0,1,2\}$, $S = \{1,2\}$ and $\Lambda$ constant at $T$ as
above, so that $|T| = 5$, check~B gives $\dim C^{1} = 480$,
$\dim Z^{1} = 120$ and $\dim B^{1} = 110$, hence
$\dim_{k} H^{1}(\catO_{\mathrm{full}};\Hom_k(B,A)) = 10 = 2\,|T|$. The group
depends on $\Lambda$ as well as on $\catO$. Indeed, when $B$ is constant with
$B(\psi) = \id$ it is the direct sum of $|T|$ copies of the cohomology with
coefficients $\psi \mapsto A(\operatorname{cod}\psi)$. This coefficient system
is the pullback along $F\catO \to \catO$ of the $\catO$-module
$A$~\cite[(1.18)(2)]{baues_cohomology_1985}, hence
by~\cite[(8.5)(B)]{baues_cohomology_1985}
\[
H^{1}\bigl(\catO;\Hom_k(B,A)\bigr) \;\cong\; \bigl(\varprojlim\nolimits^{(1)}_{\catO} A\bigr)^{\oplus |T|},
\]
i.e.\ $|T|$ copies of the first derived limit of the parameter diagram, and
$\Lambda$ appears only through $|T|$. The second script computes the reduced
group directly and confirms this factorisation. Its dimension is $2$ for
$|M| = 3$ (and $2 \cdot 5 = 10$) and $6$ for $|M| = 4$ (where $|T| = 7$, and
$6 \cdot 7 = 42$), so that the group grows with the number of the levels of
the mean. Its algebraic meaning is the usual one for a first derived limit.
When $B = A$ the cocycle identity of Lemma~\ref{lem:cocycle} says that
$A(\psi) + \varepsilon\, c(\psi)$ is functorial modulo $\varepsilon^{2}$, and
$d^{0}h$ is the deformation induced by the change of coordinates
$1 + \varepsilon h_\omega$ at each object. Hence $H^{1}$ is the space of the
first-order deformations of the parameter diagram, as a diagram of vector
spaces, modulo these changes of coordinates. We do not know its statistical
meaning.
\end{example}
\section{The ambient category for the Bayesian shape}\label{sec:ambient}

In Sections~\ref{sec:bayes} and~\ref{sec:ridge} we use states, conditionals and
almost sure equality, which do not appear in the previous sections. They are
synthetic notions of Fritz~\cite{fritz_synthetic_2020} and of Cho and
Jacobs~\cite{cho_disintegration_2019}, and we recall them here with the
references to the two sources.

\subsection{States and the category \texorpdfstring{$\BorelStoch$}{BorelStoch}}

Let $(\cat C,\otimes,I)$ be a Markov category. The unit $I$ is terminal, and a
\emph{state} on $X$ is a morphism $I \rightsquigarrow X$. In $\FinStoch$ a
state is a probability vector~\cite[Ex.~2.5]{fritz_synthetic_2020}, and in
$\Stoch$ it is a probability measure. The deterministic morphisms are those of
Definition~\ref{def:deterministic}~\cite[Def.~10.1]{fritz_synthetic_2020}.

The category $\Stoch$ is not suitable for our purposes, for two reasons. Its
deterministic morphisms need not come from
$\Meas$~\cite[Ex.~10.4]{fritz_synthetic_2020}, and conditionals need not
exist~\cite[Ex.~11.3]{fritz_synthetic_2020}. Both problems disappear in
$\BorelStoch$, the full subcategory of $\Stoch$ of the \emph{standard Borel}
spaces, i.e.\ of the measurable spaces which are isomorphic to a Polish space
with its Borel $\sigma$-algebra, equivalently to a Borel subset of
$\R$~\cite[p.~21]{fritz_synthetic_2020},
\cite[p.~12]{cho_disintegration_2019}. It contains $I$, and it is closed under
$\otimes$, hence it is a Markov category, and it is the Kleisli category of the
Giry monad restricted to the standard Borel
spaces~\cite[p.~12]{cho_disintegration_2019}. Its deterministic morphisms are
the measurable maps~\cite[Ex.~10.5]{fritz_synthetic_2020}, so that $\del$
restricted to the standard Borel spaces is fully faithful onto them. Note that
Kolmogorov products exist in $\BorelStoch$ only for countable index
sets~\cite[p.~21]{fritz_synthetic_2020},
\cite[Ex.~3.6]{fritz_rischel_infinite_2020} (see
Remark~\ref{rem:definetti}).

Throughout Sections~\ref{sec:bayes} and~\ref{sec:ridge} the ambient category is
$\BorelStoch$, and we use Theorem~\ref{thm:bridge-meas} only for the models
whose parameter objects $\Theta_\Meas(\omega)$ and sample spaces
$\Gamma(v,u)$ are standard Borel, so that $\mathbf{L}$, $\mathbf{R}$ and the
components $\mathbf P_{v,d}$ belong to $\BorelStoch$. This holds in the finite
one-way scheme of Example~\ref{ex:oneway}, where every object is finite, and in
Example~\ref{ex:linreg}, where $\Theta(\omega) = \omega^{*}\times\R_{>0}$ and
$\Gamma(\R,u) = \R^{u}$. All the statements below which do not use
conditioning, i.e.\ all except the posterior of Definition~\ref{def:prior} and
Proposition~\ref{prop:ridge}, hold in $\Stoch$ and, in the finite scheme, in
$\FinStoch$. Sections~\ref{sec:bayes} and~\ref{sec:ridge} do not depend on the
appendix.

\subsection{Conditionals and Bayesian inversion}

Let $\nu \colon I \rightsquigarrow X \otimes Y$ be a state, which we think of as
a joint distribution, and let $\nu_X$ and $\nu_Y$ be its marginals, i.e.\ its
compositions with the discard maps of $Y$ and of $X$ respectively. A
\emph{conditional} of $\nu$ given $X$ is a morphism
$\nu|_X \colon X \rightsquigarrow Y$ such that
\begin{equation}\label{eq:conditional}
\nu \;=\; (\id_X \otimes \nu|_X) \circ \mathrm{copy}_X \circ \nu_X ,
\end{equation}
and a conditional of $\nu$ given $Y$ is a morphism
$\nu|_Y \colon Y \rightsquigarrow X$ such that
$\nu = (\nu|_Y \otimes \id_Y) \circ \mathrm{copy}_Y \circ \nu_Y$. The latter is
a conditional given the first factor of the composition of $\nu$ with the
symmetry $X \otimes Y \cong Y \otimes X$. Cho and Jacobs call these morphisms
\emph{disintegrations} of $\nu$~\cite[Def.~3.5]{cho_disintegration_2019}, and a
Markov category \emph{has conditional distributions} if every state on a
product has them~\cite[Def.~11.1]{fritz_synthetic_2020}. In $\Stoch$ the two
equations read
\[
\nu(S \times T) \;=\; \int_{S} \nu|_X(T \mid x)\,\nu_X(\mathrm{d}x)
\;=\; \int_{T} \nu|_Y(S \mid y)\,\nu_Y(\mathrm{d}y)
\qquad (S \in \Sigma_X,\ T \in \Sigma_Y),
\]
i.e.\ $\nu|_X$ and $\nu|_Y$ are regular conditional probabilities of $\nu$
given the first and the second
coordinate~\cite[Ex.~11.3, eq.~(11.3)]{fritz_synthetic_2020},
\cite[Ex.~3.7]{cho_disintegration_2019}. Such kernels need not exist for
arbitrary measurable spaces, and $\nu|_X$ exists when $Y$ is standard
Borel~\cite[Thm.~3.11]{cho_disintegration_2019}. Hence $\BorelStoch$ has
conditional distributions~\cite[Ex.~11.3]{fritz_synthetic_2020},
\cite[Cor.~3.12]{cho_disintegration_2019}, while $\Stoch$ does not.

Let $\alpha \colon I \rightsquigarrow X$ be a state. Two morphisms
$f,g \colon X \rightsquigarrow Y$ are called \emph{$\alpha$-a.s.\ equal} if
$(\id_X \otimes f)\circ\mathrm{copy}_X\circ\alpha =
(\id_X \otimes g)\circ\mathrm{copy}_X\circ\alpha$~\cite[Def.~5.1]{cho_disintegration_2019},
\cite[Def.~13.1]{fritz_synthetic_2020}. In $\Stoch$ this means that
$f(T\mid -)$ and $g(T \mid -)$ agree $\alpha$-almost everywhere for every
$T \in \Sigma_Y$~\cite[Prop.~5.3]{cho_disintegration_2019},
\cite[Ex.~13.3]{fritz_synthetic_2020}. By definition, two conditionals of $\nu$
given $X$ are $\nu_X$-a.s.\ equal, and by symmetry two conditionals given $Y$
are $\nu_Y$-a.s.\ equal~\cite[Prop.~5.2]{cho_disintegration_2019},
\cite[Prop.~13.7]{fritz_synthetic_2020}.

Let now $\pi \colon I \rightsquigarrow X$ be a state and let
$k \colon X \rightsquigarrow Y$ be a morphism. We set
\[
\nu \;\coloneqq\; (\id_X \otimes k) \circ \mathrm{copy}_X \circ \pi
\colon I \rightsquigarrow X \otimes Y
\qquad\text{and}\qquad
m \;\coloneqq\; k \circ \pi \colon I \rightsquigarrow Y .
\]
Then $\nu_X = \pi$, $\nu_Y = m$ and $k$ is a conditional of $\nu$ given $X$. A
conditional of $\nu$ given $Y$ is called a \emph{Bayesian inversion} of $k$ with
respect to $\pi$, and it is denoted by
$k^{\dagger}_{\pi} \colon Y \rightsquigarrow X$. Thus $k^{\dagger}_{\pi}$ is a
morphism such that
\begin{equation}\label{eq:inversion}
(\id_X \otimes k)\circ\mathrm{copy}_X\circ\pi
\;=\;
(k^{\dagger}_{\pi} \otimes \id_Y)\circ\mathrm{copy}_Y\circ m
\end{equation}
(see~\cite[eq.~(5)]{cho_disintegration_2019}
and~\cite[eq.~(13.2)]{fritz_synthetic_2020}). In statistical terms, $X$ is the
parameter space, $Y$ is the sample space, $\pi$ is the prior, $k$ is the model,
$\nu$ is their joint distribution, $m$ is the prior predictive distribution and
$k^{\dagger}_{\pi}$ is the posterior. In $\Stoch$
equation~\eqref{eq:inversion} is Bayes' theorem in the form
\[
\int_{S} k(T\mid x)\,\pi(\mathrm{d}x) \;=\;
\int_{T} k^{\dagger}_{\pi}(S \mid y)\,m(\mathrm{d}y)
\qquad (S \in \Sigma_X,\ T \in \Sigma_Y)
\]
\cite[Ex.~3.9]{cho_disintegration_2019}, where both sides are the probability
that the parameter belongs to $S$ and the observation to $T$. Since a Bayesian
inversion is a conditional, it exists in $\BorelStoch$, and it is unique up to
$m$-a.s.\ equality. We write $k^{\dagger}$ when the prior is fixed.

\section{Priors as states, coherence as reduction-naturality}\label{sec:bayes}

In Sections~\ref{sec:bridge}--\ref{sec:hochschild} the parameter side of a
bridged model is deterministic, and meaningfulness is a condition on the
functionals of the parameter. In this section we consider the priors, which are
families of states on the parameter objects, and we prove parts~(1) and~(2) of
Theorem~\ref{thm:main-priors}. We shall see that the two criteria of
Section~\ref{sec:meaningful} play different roles for them. The definitions
and the results up to Remark~\ref{rem:coherence} hold for every bridged model
which satisfies the convention of Section~\ref{sec:ambient}. We then apply
them to the finite one-way scheme, since the priors given by the rules of
Jeffreys for the Gaussian one-way layout are improper; hence they are not
states (see Proposition~\ref{wprop:jeffreys}), and in Section~\ref{sec:ridge}
to the Gaussian linear model.

\subsection{Conventions}

Throughout this section, a response scale $v \in \ob\catV$ and a unit object
$u \in \ob\catU$ are fixed, and only the covariate space varies. Let
$\catD_{u} \subseteq \catD$ be the subcategory of the designs $(u,\omega,x)$
and of the morphisms $(\id_{u},\psi)$, for which~\eqref{eq:commacond} reads
$x' = J\psi \circ x$. Thus the target design is determined by $x$ and $\psi$.
Suppose that every $\omega$ admits a design map from $u$, as in
Example~\ref{ex:oneway-setup}. Then a family indexed by $\ob\catO$ is natural
over $\catO$ if and only if its pullback to $\catD_{u}$ is natural, and for
this reason we state the conditions over $\catO$ although the models are
defined over $\catD$. On $\catD_{u}^{\op}$ the sample side of the bridged model
is constant, i.e.\ $\mathbf{R}(u,\omega,x) = \Gamma(v,u)$ and
$\mathbf{R}\bigl((\id_{u},\psi)^{\op}\bigr) = \id$, and we write
$\mathbf{R}(\omega) = \Gamma(v,u)$ and $\mathbf{R}(\psi) = \id$. The components
of the model depend on the design map, and we write
\[
\mathbf P_{\omega} \;\coloneqq\; \mathbf P_{v,d}
\quad\text{for a design } d = (u,\omega,x) \in \ob\catD_{u} \text{ over }
\omega.
\]
Each identity below is stated for every such $d$, with target design
$d' = (u,\omega',J\psi \circ x)$. We do not choose one design for each
covariate space in a coherent way, and in general this is not possible. Indeed
in the one-way layout the two selections
$i_{1},i_{2} \colon \{*\} \to \{1,2\}$ would give $i_{1} \circ x = i_{2} \circ x$,
which fails for $u \neq \emptyset$. With these conventions, the naturality of
the bridged model $\mathbf P$ of Theorem~\ref{thm:bridge-meas} at
$\psi \colon \omega \to \omega'$ is the square~\eqref{eq:natsquare} at
$(\id_{u},\psi)$, i.e.
\begin{equation}\label{eq:natfibre}
\mathbf{R}(\psi) \circ \mathbf P_{\omega'}
\;=\; \mathbf P_{\omega} \circ \del\Theta_\Meas(\psi)
\qquad\text{in } \Stoch .
\end{equation}
By~\eqref{eq:pull} it is condition (N-$\catO$) of Section~\ref{sec:worked},
written for the kernels $\mathbf P_{v,d}$ instead of the maps $P_{v,d}$. We
keep the factor $\mathbf{R}(\psi)$, which is an identity here, because in this
form the proofs below hold over $\catD$, where $\mathbf{R}$ acts by
marginalisation. The quantifier classes
$\catO_{\mathrm{red}} \subseteq \catO_{\mathrm{inj}} \subseteq \catO$ are those
of Section~\ref{sec:meaningful}.

\subsection{Priors and coherence}

\begin{definition}[Prior; coherence]\label{def:prior}
Let $\mathbf P$ be a bridged model. A \emph{prior} for $\mathbf P$ is a family
of states
\[
\pi_{\omega} \colon I \rightsquigarrow \Theta_\Meas(\omega),
\qquad \omega \in \ob\catO ,
\]
i.e.\ of probability measures on the parameter objects. The associated
\emph{Bayesian model} at $\omega$ is the joint state
$(\id \otimes \mathbf P_{\omega}) \circ \mathrm{copy} \circ \pi_{\omega}
\colon I \rightsquigarrow \Theta_\Meas(\omega) \otimes \mathbf{R}(\omega)$,
and the \emph{prior predictive} is its marginal on $\mathbf{R}(\omega)$
\[
m_{\omega} \;\coloneqq\; \mathbf P_{\omega} \circ \pi_{\omega}
\colon I \rightsquigarrow \mathbf{R}(\omega).
\]
The \emph{posterior} is the Bayesian inversion
$\mathbf P_{\omega}^{\dagger} \colon \mathbf{R}(\omega) \rightsquigarrow
\Theta_\Meas(\omega)$ of $\mathbf P_{\omega}$ with respect to $\pi_{\omega}$,
i.e.\ a conditional of the Bayesian model given $\mathbf{R}(\omega)$
(see~\eqref{eq:inversion}). It exists in $\BorelStoch$, and it is unique up to
$m_{\omega}$-a.s.\ equality.

Let $\cat{Q} \subseteq \catO$ be a subcategory which contains all the objects.
The prior is called \emph{$\cat{Q}$-coherent} if the family
$(\pi_{\omega})_{\omega}$ is $\cat{Q}$-coherent in the sense of
Definition~\ref{def:coherence} with $F = \Delta_{I}$ and
$G = \del\Theta_\Meas$, where
$\Delta_{I} \colon \catO^{\op} \to \BorelStoch$ is the constant functor at the
monoidal unit $I$. Explicitly, this means that
\begin{equation}\label{eq:priorsquare}
\pi_{\omega} \;=\; \del\Theta_\Meas(\psi)\circ\pi_{\omega'}
\;=\; \Theta_\Meas(\psi)_{*}\,\pi_{\omega'}
\end{equation}
for every $\psi \colon \omega \to \omega'$ in $\cat{Q}$. We say that the prior
is \emph{coherent at} $\psi$ if~\eqref{eq:priorsquare} holds at $\psi$, so
that $\cat{Q}$-coherence is coherence at every morphism of $\cat{Q}$. We say
that the prior is \emph{coherent} if it is $\catO_{\mathrm{red}}$-coherent,
i.e.\ natural along the insertions, which is the class of Tjur. We say that it is
\emph{relabelling-coherent} if it is $\catO_{\mathrm{inj}}$-coherent, i.e.\
natural along all the injections, which is the class of McCullagh as stated.
Since $\catO_{\mathrm{red}} \subseteq \catO_{\mathrm{inj}}$, the second
condition implies the first one.
\end{definition}

In words, a prior is coherent if the prior of a smaller design is the marginal
of the prior of every larger design which contains it.

\begin{proposition}[Structure of coherence]\label{prop:coherence}
Let $\mathbf P$ be a bridged model, let $\pi$ be a prior and let
$\cat{Q} \subseteq \catO$ be a subcategory which contains all the objects.
\begin{enumerate}
\item If~\eqref{eq:priorsquare} holds at $\psi \colon \omega \to \omega'$ then
$\pi_{\omega}$ vanishes on every measurable subset of $\Theta_\Meas(\omega)$
which is disjoint from $\operatorname{im}\Theta_\Meas(\psi)$. If the image is
measurable, this means that
$\pi_{\omega}\bigl(\operatorname{im}\Theta_\Meas(\psi)\bigr) = 1$, and in the
finite case that
$\operatorname{supp}\pi_{\omega} \subseteq \operatorname{im}\Theta_\Meas(\psi)$.
\item Let $\psi_1,\psi_2 \colon \omega \to \omega'$ be two morphisms of
$\cat{Q}$. If $\pi$ is $\cat{Q}$-coherent then
$\Theta_\Meas(\psi_1)_{*}\pi_{\omega'} = \Theta_\Meas(\psi_2)_{*}\pi_{\omega'}$.
In particular, if $\pi$ is $\catO_{\mathrm{inj}}$-coherent then $\pi_{\omega}$
is invariant under $\Theta_\Meas(\operatorname{Aut}\omega)$ for every
$\omega$, while $\catO_{\mathrm{red}}$-coherence imposes no symmetry on any
component.
\item If $\pi$ is $\cat{Q}$-coherent then
$m_{\omega} = \mathbf{R}(\psi)\circ m_{\omega'}$ for every
$\psi \colon \omega \to \omega'$ in $\cat{Q}$ and every design $d$ over
$\omega$, where $m_{\omega}$ is computed at $d$ and $m_{\omega'}$ at the
target design $d'$. Thus the prior predictives, indexed by the designs of
$\catD_{u}$, form a natural transformation $\Delta_{I} \Rightarrow \mathbf{R}$
over the morphisms $(\id_{u},\psi)$ with $\psi$ in $\cat{Q}$, and we say
that $m$ is \emph{natural over} $\cat{Q}$.
\item Suppose that, for every $\omega$, the maps
$\alpha \mapsto \mathbf P_{\omega}\circ\alpha$ from the states on
$\Theta_\Meas(\omega)$ to the states on $\mathbf{R}(\omega)$, one for each
design $d$ over $\omega$, are jointly injective on
a class $\cat{S}_{\omega}$ of states which contains $\pi_{\omega}$ and
satisfies $\Theta_\Meas(\psi)_{*}\cat{S}_{\omega'} \subseteq \cat{S}_{\omega}$
for every $\psi \colon \omega \to \omega'$ in $\cat{Q}$, i.e.\ that two states
of $\cat{S}_{\omega}$ with the same image at every design over $\omega$ are
equal. This holds, for example, when one of the maps is injective on all
the states, i.e.\ when the model is identifiable in mixtures at one design.
If $m$ is
natural over $\cat{Q}$ then $\pi$ is $\cat{Q}$-coherent. In $\FinStoch$ the
injectivity of one map on all the states means that the rows of the
stochastic matrix
$\mathbf P_{\omega}$, indexed by $\Theta(\omega)$, are linearly independent.
\end{enumerate}
\end{proposition}

\begin{proof}
(1) Let $f$ be a measurable map and let $A$ be a measurable set which is
disjoint from $\operatorname{im} f$. Then
$(f_{*}\nu)(A) = \nu(f^{-1}A) = \nu(\emptyset) = 0$. Note that the image need
not be measurable, and for this reason the statement is given on the sets
which are disjoint from it.

(2) The first statement follows from~\eqref{eq:priorsquare} applied to
$\psi_1$ and to $\psi_2$. For the automorphisms take $\psi_1 = \sigma$ and
$\psi_2 = \id_{\omega}$. The morphisms of $\catO_{\mathrm{red}}$ are
inclusions, and an inclusion is determined by its source and its target, so
that $\catO_{\mathrm{red}}$ has no pairs of distinct parallel morphisms.

(3) Let $\psi$ be a morphism of $\cat{Q}$. By the naturality of $\mathbf P$ in
the form~\eqref{eq:natfibre} we have
\[
\mathbf{R}(\psi) \circ m_{\omega'} =
\mathbf{R}(\psi)\circ\mathbf P_{\omega'}\circ\pi_{\omega'} =
\mathbf P_{\omega}\circ\del\Theta_\Meas(\psi)\circ\pi_{\omega'} =
\mathbf P_{\omega}\circ\pi_{\omega} = m_{\omega}.
\]

(4) The same chain of equalities gives
$\mathbf P_{\omega}\circ\bigl(\Theta_\Meas(\psi)_{*}\pi_{\omega'}\bigr) =
\mathbf P_{\omega}\circ\pi_{\omega}$ at every design over $\omega$. Both
states belong to $\cat{S}_{\omega}$, and they are equal by the joint
injectivity. For
the statement on $\FinStoch$, write the states as row vectors, so that
$\mathbf P_{\omega}\circ\alpha = \alpha\mathbf P_{\omega}$. If
$v\mathbf P_{\omega} = 0$ then
$\sum_{\theta} v_{\theta} = \sum_{y}(v\mathbf P_{\omega})_{y} = 0$, because
the rows sum to one. Hence the kernel of $v \mapsto v\mathbf P_{\omega}$ is
contained in the hyperplane of the vectors with zero sum, which is spanned
by the differences of probability vectors, and the map is injective on the
states if and only if it is injective, i.e.\ if and only if the rows are
linearly independent.
\end{proof}

\begin{remark}\label{rem:coherence}
Consider the merge $p \colon \{1,2\} \to \{*\}$ of the one-way layout, for
which $\Theta(p)$ is the diagonal, both in the Gaussian model of
Example~\ref{ex:oneway-setup} and in the finite scheme of
Example~\ref{ex:oneway}. The diagonal is measurable, hence by part~(1),
if~\eqref{eq:priorsquare} holds at $p$ then $\pi_{\{1,2\}}$ is concentrated
on it, so that the prior for two groups carries no more information than the
prior for one group. For this reason~\eqref{eq:priorsquare} is imposed over
$\catO_{\mathrm{red}}$ or over $\catO_{\mathrm{inj}}$ and not over $\catO$.
Consider now part~(2) in the one-way layout, with $\omega = \{1,\dots,n\}$ and
$\Theta(\sigma)(\mu,s) = (\mu\circ\sigma, s)$ for a permutation $\sigma$. The
invariance under $\Theta_\Meas(\operatorname{Aut}\omega)$ is the invariance of
$\pi_{\omega}$ under all the permutations of the coordinates of the mean,
jointly with the dispersion, i.e.\ the finite counterpart of the
exchangeability of~\cite[Def.~4.1]{fritz_definetti_2021}, which is stated
there for a state on a countable Kolmogorov power (see
Remark~\ref{rem:definetti}). At the two injections
$\{*\} \to \{1,2\}$ part~(2) says that the two coordinate marginals of
$\pi_{\{1,2\}}$ agree, and at the relabelling $\{1\} \to \{2\}$, composed with
the inclusions into $\{1,2\}$, it says that the two priors for one group
coincide. Finally, note that the hypothesis of part~(4) is stronger than the
injectivity of $\theta \mapsto \mathbf P_{\omega}(\theta)$.
\end{remark}

\subsection{The finite one-way scheme}

\begin{proposition}[The four strata of the two-object scheme]\label{prop:priordim}
In the finite one-way scheme of Example~\ref{ex:oneway} we have
$\Theta(\omega_n) = M^{n}\times S$ with $\omega_1 = \{*\}$ and
$\omega_2 = \{1,2\}$, and the morphisms of $\catO_{\mathrm{full}}$ are the
merge $p$, the selections $i_1,i_2$, the swap $\sigma$ and the constants
$c_k = i_k p$. We identify $\omega_1$ with $\{1\}$. Then the traces on
$\catO_{\mathrm{full}}$ of the classes of Section~\ref{sec:meaningful} are
\[
\catO_{\mathrm{red}} = \langle i_1\rangle \;\subsetneq\;
\catO_{\mathrm{sel}} \coloneqq \langle i_1,i_2\rangle \;\subsetneq\;
\catO_{\mathrm{inj}} = \langle i_1,i_2,\sigma\rangle \;\subsetneq\;
\catO_{\mathrm{full}} ,
\]
where $\catO_{\mathrm{sel}}$ is not one of the classes of
Section~\ref{sec:meaningful}. It comes from the identification of the two
designs with one group $\{1\}$ and $\{2\}$ with the single object $\omega_1$,
i.e.\ from the relabelling $\{1\} \to \{2\}$ imposed without the relabelling
$\sigma$ of the design with two groups. For each of the four classes, the set of
the $\cat{Q}$-coherent priors is a polytope, of affine dimension
\[
\begin{array}{lll}
\catO_{\mathrm{red}}:  & |M|^{2}|S| - 1 & \text{(no constraint on $\pi_{\omega_2}$)},\\[2pt]
\catO_{\mathrm{sel}}:  & |M|^{2}|S| - 1 - |S|(|M|-1) & \text{(the two coordinate marginals of $\pi_{\omega_2}$ agree)},\\[2pt]
\catO_{\mathrm{inj}}:  & \tfrac12 |M|(|M|+1)\,|S| - 1 & \text{($\pi_{\omega_2}$ symmetric under the swap)},\\[2pt]
\catO_{\mathrm{full}}: & |M||S| - 1 & \text{($\pi_{\omega_2}$ supported on the diagonal)}.
\end{array}
\]
The second and the third dimensions coincide if and only if $|M| \le 2$, and
their difference is $|S|\binom{|M|-1}{2}$. Thus for $|M| = 2$ the equality of
the two marginals implies the symmetry, and for $|M| \ge 3$ it does not.
\end{proposition}

\begin{proof}
In all the cases $\pi_{\omega_1}$ is determined by $\pi_{\omega_2}$
through~\eqref{eq:priorsquare} at $i_1$, so that the free variable is
$\pi_{\omega_2}$, a probability vector on $M^{2}\times S$. For the first three
classes the uniform vector on $M^{2}\times S$ satisfies the constraints below
and has positive coordinates, so that it lies in the relative interior of
the polytope, and the affine dimension of the polytope is the one of the
affine space of solutions. For $\catO_{\mathrm{full}}$ the polytope is the
image of the simplex on $M\times S$ under the injective affine map
$\pi_{\omega_1} \mapsto (\pi_{\omega_1},\Theta(p)_{*}\pi_{\omega_1})$, and
its dimension is the one of the simplex. Over $\catO_{\mathrm{red}}$ there
is no further condition. Over
$\catO_{\mathrm{sel}}$ the only further condition is
Proposition~\ref{prop:coherence}(2) at $i_1,i_2$, i.e.\ that the two coordinate
marginals of $\pi_{\omega_2}$ agree. The difference of the two marginalisation
maps $\mathbb{Q}^{M^{2}\times S} \to \mathbb{Q}^{M\times S}$ takes its values
in the subspace of the vectors with zero sum over $M$ for each $s$, and it is
onto it, hence its rank is $|S|(|M|-1)$. Its row space does not contain the
normalisation functional, since every row functional vanishes on the diagonal
entries. This gives the second dimension. Over $\catO_{\mathrm{inj}}$ the
condition at $\sigma$ is the symmetry of $\pi_{\omega_2}$ under the swap,
which implies the equality of the marginals, and the symmetric probability
vectors are parametrised by the $\tfrac12|M|(|M|+1)|S|$ orbits of the swap on
$M^{2}\times S$. Over $\catO_{\mathrm{full}}$ the condition at $p$ confines
$\pi_{\omega_2}$ to the diagonal, which has $|M||S|$ points, and then the
conditions at $i_k$, $\sigma$ and $c_k$ hold. Finally, the difference of the
second and the third dimensions is
$|S|\bigl(|M|^{2}-|M|+1-\tfrac12|M|(|M|+1)\bigr) =
|S|\tfrac12(|M|-1)(|M|-2)$.
\end{proof}

The four formulas have been verified exactly over $\mathbb{Q}$ for $|M|\le 4$
and $|S|\le 3$ by the script
\texttt{oneway\_\allowbreak prior\_\allowbreak coherence.py} (see the end of
this section).

\begin{proposition}[Jeffreys' rule is not coherent, the independence prior
is: finite analogue]\label{wprop:jeffreys}
Consider the Gaussian one-way layout, with
$\theta_{\omega} = (\mu\in\R^{\omega},\ \sigma^{2}>0)$, at a design in which
every treatment group is non-empty, so that the Fisher information is
nonsingular. Jeffreys' general rule $\pi \propto \sqrt{\det I(\theta)}$
\cite{jeffreys_invariant_1946}, \cite[p.~1345, eq.~(1)]{kass_selection_1996},
gives
\[
\pi^{J}_{\omega}(\mathrm{d}\mu, \mathrm{d}\sigma^{2}) \;\propto\;
(\sigma^{2})^{-(|\omega|+2)/2}\,\mathrm{d}\mu\,\mathrm{d}\sigma^{2},
\]
whose exponent depends on the number $|\omega|$ of the mean parameters, hence
on the design. The prior obtained by treating the location parameters
separately, $\pi^{\mathrm{iJ}}_{\omega} \propto (\sigma^{2})^{-1}$, which is
Jeffreys' own recommendation for the location--scale families
\cite[pp.~182--183]{jeffreys_theory_1961},
\cite[p.~1345, eq.~(2)]{kass_selection_1996} and is called the independent
Jeffreys prior in the later
literature~\cite[pp.~631--632]{consonni_prior_2018}, has an exponent which does
not depend on $|\omega|$. Consider the proper finite analogue, where $S$ is a
finite set of squares of rational numbers and $\pi_{\omega_n}$ is uniform in
$\mu$ with weight $s^{-(n+2)/2}$ in $s$, normalised.
\begin{enumerate}
\item A product rule
$\pi_{\omega} = \mathrm{unif}(M^{\omega}) \otimes \rho_{\omega}$ is coherent if
and only if $\rho_{\omega}$ does not depend on $\omega$. Hence the analogue of
$\pi^{\mathrm{iJ}}$ is coherent and the analogue of $\pi^{J}$ is not.
\item The defect of $\pi^{J}$ is contained in the marginal of the dispersion.
Its total variation along each of the two selections of a group is equal to
the total variation between the normalisations of $s^{-3/2}$ and $s^{-2}$ on
$S$, and it does not depend on $M$. For $M=\{0,1\}$ and $S=\{1,4\}$ the two
marginals are $(\tfrac{8}{9},\tfrac{1}{9})$ and
$(\tfrac{16}{17},\tfrac{1}{17})$, and the defect is $\tfrac{8}{153}$. Let
$\rho_{\omega_n}$ be the normalised marginal of the dispersion of
$\pi^{J}_{\omega_n}$. The defect along the merge is
\[
\tfrac12\Bigl[\bigl(1-|M|^{-1}\bigr)
 + \sum_{s\in S}\bigl|\,|M|^{-1}\rho_{\omega_2}(s)-\rho_{\omega_1}(s)\bigr|\Bigr]
\;\ge\; 1-|M|^{-1},
\]
where $1-|M|^{-1}$ is the mass of $\pi^{J}_{\omega_2}$ outside the diagonal.
The equality holds if and only if
$\rho_{\omega_1}(s) \ge |M|^{-1}\rho_{\omega_2}(s)$ for every $s$, i.e.\ if and
only if $Z_{1}/Z_{2} \le |M|\,s_{\min}^{1/2}$, where
$Z_{n} = \sum_{s} s^{-(n+2)/2}$. The bound is attained in every small case, for
example in the one above, but it is not an identity. For $|M| = 2$ and
$S = \{1/4\} \cup \{(k/100)^{2} : 105 \le k \le 229\}$ the defect along the
merge is $0.5078\ldots > \tfrac12$.
\end{enumerate}
\end{proposition}

Part~(1) and the numerical content of part~(2), i.e.\ the formula for the
merge, the cases where the bound is attained and the case where it is strict,
are established for the finite analogue by exact computation over $\mathbb{Q}$
(script \texttt{oneway\_\allowbreak prior\_\allowbreak coherence.py}). The
statement on $\pi^{J}$ and $\pi^{\mathrm{iJ}}$ in the continuous case is not a
statement on $\Stoch$. Indeed, both priors are improper; hence they are not
states, i.e.\ they are not morphisms
$I \rightsquigarrow \Theta_\Meas(\omega)$. In the continuous case we only
assert the computation of the exponent, which is classical (see
Remark~\ref{rem:open}).

The dependence of the exponent on the design is Jeffreys' own objection to his
general rule. Applied to $k$ unknown means with a common variance, the general
rule gives marginal $t$ posteriors whose degrees of freedom depend on the total
number of observations only, whatever $k$ is, and he found this
unacceptable~\cite[p.~182]{jeffreys_theory_1961} (as reported
in~\cite[p.~1345]{kass_selection_1996}). His remedy was to treat the location
parameters separately, which gives $\pi^{\mathrm{iJ}}$. Part~(1) says that the
dependence on the design is a failure of naturality along the insertions, and
that his remedy is the unique product rule which restores it. This agrees
with, but is not the same as, the preference in the objective Bayes literature
for the reference priors over the multivariate Jeffreys rule when nuisance
parameters are present. For the normal location--scale family the reference
prior for $\mu$ is $\sigma^{-1}$ and not Jeffreys'
$\sigma^{-2}$~\cite[p.~119]{bernardo_reference_1979}, the multiparameter
reference process is sequential in order to avoid the marginalisation
paradoxes~\cite[pp.~905--906]{berger_formal_2009}, and the multidimensional
Jeffreys rule is recorded as a source of incoherence and of
paradoxes~\cite[p.~631]{consonni_prior_2018}. We do not claim that coherence
characterises the reference priors.

\begin{remark}[Open problems]\label{rem:open}
The priors $\pi^{J}$ and $\pi^{\mathrm{iJ}}$ are improper, and to consider them
one has to replace $\BorelStoch$ in Definition~\ref{def:prior} and in
Proposition~\ref{prop:coherence} with a codomain which admits unnormalised
states, for example the measurable spaces with the $s$-finite kernels, or the
Kleisli category of the measure monad. We do not know whether
Proposition~\ref{prop:coherence} holds in such a codomain, whether
Proposition~\ref{wprop:jeffreys} then holds for $\pi^{J}$ and
$\pi^{\mathrm{iJ}}$ as stated in the continuous case, and whether coherence
over $\catO_{\mathrm{red}}$ singles out a recognisable class of objective
priors beyond the product rules of Proposition~\ref{wprop:jeffreys}(1). We
expect only the first question to be routine.

In the marginalisation paradoxes of Dawid, Stone and
Zidek~\cite[\S1 and Thm.~2.1, p.~198]{dawid_stone_zidek_1973} the marginal
posterior of a parameter $\zeta$ depends on the data only through a statistic
$z$ whose sampling density depends only on $\zeta$, and it is not proportional
to $\pi(\mathrm{d}\zeta)f(z\mid\zeta)$ for any prior $\pi$. These are failures of
commutation between the Bayesian inversion and the marginalisation of the
parameter induced by a design reduction, i.e.\ of squares of the
form~\eqref{eq:tjursquare} whose vertical arrows are the inversions
$\mathbf P^{\dagger}_{\omega}$. We conjecture that in a codomain which admits
unnormalised states, they are the failures of naturality of the family
$(\mathbf P^{\dagger}_{\omega})_{\omega}$ over $\catO_{\mathrm{red}}$.
Already for proper priors the inversions $\mathbf P^{\dagger}_{\omega}$ of
Definition~\ref{def:prior} exist at every $\omega$, and whether they form a
natural transformation is the question of whether the diagram category of
Remark~\ref{rem:meas-hyp} has Bayesian inversions, a form of the open problem
of Fritz on conditionals in diagram
categories~\cite[Problem~11.9]{fritz_synthetic_2020}. The companion
paper~\cite[Prop.~7.4 and Section~9]{vaccarino_monads_2026} answers this
question for the priors of Definition~\ref{def:prior}. If $\pi$ is
$\cat{Q}$-coherent then the inversions are natural over $\cat{Q}$, up to
almost sure equality with respect to the prior predictives, and a prior whose
predictives and inversions are natural over $\cat{Q}$ is $\cat{Q}$-coherent.
The argument is synthetic and needs the priors to be states, so that the
question stays open for the improper priors. For the conjecture above, a
stronger statement is not possible, for two reasons. First, Dawid, Stone and
Zidek prove that the paradoxes cannot arise from proper
priors~\cite[p.~189 and p.~194, eqs.~(1.20)--(1.23)]{dawid_stone_zidek_1973},
while the incoherence of Proposition~\ref{wprop:jeffreys} already occurs for
proper priors, so that the two phenomena are not the same. Their
classification of the problems as reducible or irreducible and of the priors
as paradoxical or
paradox-free~\cite[p.~208]{dawid_stone_zidek_1973} concerns the priors for a
fixed problem, while coherence compares the priors across the designs, and it
is only conjectured that the reference priors avoid the
paradoxes~\cite[p.~123]{bernardo_reference_1979}. Second, the conjecture
cannot be stated in $\Stoch$, since improper priors are not states. Thus we do
not claim that the present framework explains the marginalisation paradoxes.
\end{remark}

\begin{remark}[Kolmogorov consistency and exchangeability]\label{rem:definetti}
Condition~\eqref{eq:priorsquare} over $\catO_{\mathrm{red}}$ says that
$(\pi_{\omega})_{\omega}$ is a compatible family over the diagram of the
reductions, i.e.\ a Kolmogorov consistent family, which is the input of an
extension problem where the designs play the role of the finite sets of
indices. The corresponding universal property is the one of the Kolmogorov
products of~\cite[Def.~3.1, Def.~4.1]{fritz_rischel_infinite_2020}, which
exist in $\BorelStoch$ by the classical extension
theorem~\cite[Ex.~3.6]{fritz_rischel_infinite_2020}. Coherence over
$\catO_{\mathrm{red}}$ gives no symmetry, since between two objects there is
at most one inclusion. Relabelling-coherence, i.e.\ coherence over
$\catO_{\mathrm{inj}}$, gives in addition, at the automorphisms of each
$\omega_n = \{1,\dots,n\}$, the invariance under the finite permutations which
the categorical de~Finetti
theorem~\cite[Thm.~4.4]{fritz_definetti_2021} assumes for a state on a
countable Kolmogorov power, in a Markov category which
satisfies~\cite[Assumption~4.2]{fritz_definetti_2021}, as $\BorelStoch$
does~\cite[Ex.~4.3]{fritz_definetti_2021}. Hence exchangeability is neither an
additional modelling assumption nor a consequence of the design reductions
alone. It comes from the relabellings, which are the morphisms that
distinguish the two classical criteria. This is only a structural observation.
We do not prove an extension theorem for the setting of McCullagh and Br\o ns,
and we have not checked the hypotheses
of~\cite{fritz_rischel_infinite_2020,fritz_definetti_2021} against
Definition~\ref{def:designdata}.
\end{remark}

\subsection*{What is computed}

The script \texttt{oneway\_\allowbreak prior\_\allowbreak coherence.py}
carries out the computations of Propositions~\ref{prop:priordim}
and~\ref{wprop:jeffreys} in exact arithmetic over $\mathbb{Q}$. It is
deposited with its committed output in the same release of
\url{https://github.com/Vaxx66/calamus} as the scripts of
Section~\ref{sec:hochschild}, and a fresh run reproduces the committed output
byte for byte. For each of the four classes of
Proposition~\ref{prop:priordim}, the script builds the linear system in the
unknown $(\pi_{\omega_1},\pi_{\omega_2})$ given by the
identities~\eqref{eq:priorsquare} along the generators of the class and by the
two normalisations, it computes the affine dimension of the set of the
solutions, and it compares it with the formulas for $|M| \le 4$ and
$|S| \le 3$. Then it evaluates the analogue of Jeffreys' rule and the product
rules on some cases, and it computes exactly the total variation defects along
$i_{1}$, $i_{2}$ and $p$, including the case where the bound $1-|M|^{-1}$ for
the merge is strict.
\section{Ridge regression as a bridged Bayesian object}\label{sec:ridge}

In this section we apply the constructions of Section~\ref{sec:bayes} to the
Gaussian linear model of Example~\ref{ex:linreg}, and we prove part~(3) of
Theorem~\ref{thm:main-priors}. We obtain the ridge estimator, together with
the correspondence between the penalty and the precision of the prior. We
shall see that the ridge prior is coherent in the sense of
Definition~\ref{def:prior}, and that its failure of naturality over the whole
category of the covariate spaces corresponds to the usual requirement that
ridge regression be applied to standardised covariates.

Throughout this section $\sigma^{2} = \sigma_{0}^{2}$ is fixed and known, and
the unknown parameter is the coefficient functional. All the objects below,
i.e.\ $\omega^{*} \cong \R^{p}$, $\R_{>0}$ and $\R^{u} \cong \R^{n}$, are
standard Borel. Hence the convention of Section~\ref{sec:ambient} applies, and
the Bayesian inversions exist and are unique up to a.s.\ equality.

\begin{definition}[Ridge datum]\label{def:ridge}
A \emph{ridge datum} on $\catO$ is a choice, for each covariate space
$\omega$, of an inner product $\langle\,,\rangle_{\omega}$ on $\omega$,
equivalently of a basis up to orthogonal transformations, together with a real
number $\tau^{2} > 0$ which does not depend on $\omega$. The datum is called
\emph{compatible} if the inner
products restrict along the inclusions of $\catO_{\mathrm{red}}$, i.e.\ if
$\langle\,,\rangle_{\omega'} =
\langle\,,\rangle_{\omega}\restriction_{\omega'}$ whenever
$\omega' \subseteq \omega$, so that every insertion is an isometric embedding.
Let $I_{\omega}$ be the form induced on $\omega^{*}$. The \emph{ridge prior}
associated to the datum is the product state
\[
\pi^{\tau}_{\omega}
\;\coloneqq\;
\Norm(0,\ \tau^{2} I_{\omega}) \otimes \delta_{\sigma_{0}^{2}}
\;\colon\; I \rightsquigarrow \omega^{*} \times \R_{>0} .
\]
\end{definition}

A Gaussian state on $\omega^{*}$ is not canonical, since it requires the choice
of a positive definite quadratic form, and this is the reason for the choice of
the inner products. Compatible data exist on the full subcategory of
$\Vect^{\mathrm{fd}}_{\R}$ whose objects are the linear subspaces of the
coordinate spaces $\R^{p}$, $p \ge 0$, which contains a representative of
every finite dimensional space and all the linear maps between its objects.
Indeed, take the standard form on each $\R^{p}$ and its restriction on each
subspace, which is the choice implicit in the description by matrices of
Example~\ref{ex:linreg}(5). In this section $\catO$ denotes this
subcategory. Compatibility is a condition on the datum, and it is used only
in Proposition~\ref{prop:ridgecoherence}(1).

\begin{proposition}[Ridge as Bayesian inversion]\label{prop:ridge}
Let $d = (u,\omega,x)$ be a design. Choose a basis of $\omega$ which is
orthonormal for the ridge datum and an enumeration of $u$, so that $x$ becomes
the $n \times p$ matrix $X$ of Example~\ref{ex:linreg}(5), and set
\[
\lambda \;\coloneqq\; \frac{\sigma_{0}^{2}}{\tau^{2}} .
\]
Let $\mathbf{P}_{d}$ be the component of the bridged model of the Gaussian
linear model (Theorem~\ref{thm:bridge-meas}).
\begin{enumerate}
\item The prior predictive state is
$m_{d} = \Norm_{n}(0,\ \tau^{2} X X^{\!\top} + \sigma_{0}^{2} I_{n})$.
\item The Bayesian inversion
$\mathbf{P}_{d}^{\dagger} \colon \R^{u} \rightsquigarrow
\omega^{*}\times\R_{>0}$ exists in $\BorelStoch$, and it is the Gaussian kernel
\[
\mathbf{P}_{d}^{\dagger}(y)
\;=\;
\Norm_{p}\Bigl(
(X^{\!\top}X + \lambda I)^{-1} X^{\!\top} y,
\;\;
\sigma_{0}^{2}(X^{\!\top}X + \lambda I)^{-1}
\Bigr)
\]
tensored with $\delta_{\sigma_{0}^{2}}$. We omit the factor
$\delta_{\sigma_{0}^{2}}$ here and in what follows.
\item The mean of the posterior is the ridge estimator of Hoerl and
Kennard~\cite[eq.~(2.1), p.~57]{hoerl_ridge_1970},
$\widehat{\beta}_{\lambda}(y) = (X^{\!\top}X + \lambda I)^{-1} X^{\!\top} y$,
and the penalty is the ratio of the two variances, i.e.\
$\lambda = \sigma_{0}^{2}/\tau^{2}$. This is the Bayesian reading which is
already given in~\cite[\S6, p.~64]{hoerl_ridge_1970}, where the ridge estimate
is the posterior mean under a centred normal prior with covariance
$(\delta^{2}/k)\,I$, where $\delta^{2}$ is not defined there and must be the
error variance $\sigma^{2}$ for the statement to hold, so that
$k = \sigma^{2}/\tau^{2}$.
\end{enumerate}
\end{proposition}

\begin{proof}
Let $\beta$ be the coordinate vector of the coefficient functional. Under
$\pi^{\tau}_{\omega}$ and $\mathbf{P}_{d}$ we have
$\beta \sim \Norm_{p}(0,\tau^{2}I)$ and
$y \mid \beta \sim \Norm_{n}(X\beta, \sigma_{0}^{2}I)$. Hence $(\beta,y)$ is
jointly Gaussian with
$\operatorname{Cov}(y) = \tau^{2}XX^{\!\top} + \sigma_{0}^{2}I$ and
$\operatorname{Cov}(\beta,y) = \tau^{2}X^{\!\top}$, and this gives~(1). The
precision of the posterior is
$\sigma_{0}^{-2}X^{\!\top}X + \tau^{-2}I
= \sigma_{0}^{-2}\bigl(X^{\!\top}X + \lambda I\bigr)$, which gives the stated
covariance, and the mean of the posterior is that covariance applied to
$\sigma_{0}^{-2}X^{\!\top}y$, i.e.\
$(X^{\!\top}X+\lambda I)^{-1}X^{\!\top}y$. Since $X^{\!\top}X + \lambda I$ is
positive definite for $\lambda>0$, these expressions are defined with no
condition on the rank of $X$. Finally, this Gaussian kernel is the Bayesian
inversion because
$(\mathbf{P}_{d}^{\dagger} \otimes \id) \circ \mathrm{copy} \circ m_{d}$ is the
joint state of Definition~\ref{def:prior}, which for jointly Gaussian pairs is
the classical formula for conditioning. This gives~(2), and~(3) follows.
\end{proof}

\begin{remark}\label{rem:ridge}
The content of Proposition~\ref{prop:ridge} is the classical conjugacy. The
categorical statements are that $\pi^{\tau}_{\omega}$ is a state, that
$\mathbf{P}_{d}$ is a morphism of $\Stoch$ (Theorem~\ref{thm:bridge-meas}), and
that the inversion exists in $\BorelStoch$. Thus, prior, likelihood, joint
distribution and posterior are four morphisms of the same category, and they
can be composed with the structure indexed by the designs of
Section~\ref{sec:bridge}.
\end{remark}

\begin{proposition}[Coherence of the ridge prior]\label{prop:ridgecoherence}
Let $\pi^{\tau}$ be as in Definition~\ref{def:ridge}, with $\catO$ as fixed
after that definition, and let
$\psi \colon \omega \to \omega'$ be a morphism of $\catO$, so that
$\Theta(\psi) = \psi^{*}\times\id$.
\begin{enumerate}
\item If $\psi$ is an isometric embedding for the chosen inner products, then
$\Theta(\psi)_{*}\,\pi^{\tau}_{\omega'} = \pi^{\tau}_{\omega}$. Hence, for a
compatible ridge datum, $\pi^{\tau}$ is coherent, i.e.\
$\catO_{\mathrm{red}}$-coherent in the sense of Definition~\ref{def:prior},
and by Proposition~\ref{prop:coherence}(3) the family of the prior predictives
$(m_{d})_{d}$ is natural over $\catO_{\mathrm{red}}$.
\item Let $A$ be a linear automorphism of $\omega$.
Then~\eqref{eq:priorsquare} holds at $A$ if and only if $A$ is orthogonal for
the chosen inner product. In particular it fails for every rescaling of the
coordinates $\mathrm{diag}(a_{1},\dots,a_{p})$ with some $|a_{j}| \neq 1$.
Hence $\pi^{\tau}$ is never $\catO_{\mathrm{inj}}$-coherent when some $\omega$
has $\dim\omega \ge 1$, whatever the datum is.
\item Let $\psi$ be surjective with $\dim\omega' < \dim\omega$, i.e.\ a merging
or a collapsing of covariate directions. Then $\psi^{*}$ has a proper image,
and by Proposition~\ref{prop:coherence}(1) coherence at $\psi$ would confine
$\pi^{\tau}_{\omega}$ to a proper subspace, which is not possible for a
nondegenerate Gaussian.
\end{enumerate}
\end{proposition}

\begin{proof}
(1) Choose an orthonormal basis of $\omega$ for
$\langle\,,\rangle_{\omega}$ and extend it to an orthonormal basis of
$\omega'$, which is possible because $\psi$ is an isometric embedding. In the
dual coordinates $\psi^{*}$ is the coordinate projection
$\R^{p'} \to \R^{p}$, and the pushforward of $\Norm(0,\tau^{2}I_{p'})$ along a
coordinate projection is $\Norm(0,\tau^{2}I_{p})$. The factor
$\delta_{\sigma_{0}^{2}}$ is preserved by $\id$. For a compatible datum every
morphism of $\catO_{\mathrm{red}}$ is an isometric embedding.

(2) The map $\Theta(A)$ acts on the coordinate vectors as $A^{\!\top}$, hence
$\Theta(A)_{*}\Norm(0,\tau^{2}I) = \Norm(0,\tau^{2}A^{\!\top}A)$, which is
equal to $\Norm(0,\tau^{2}I)$ if and only if $A^{\!\top}A = I$.

(3) It follows from Proposition~\ref{prop:coherence}(1) applied to
$\operatorname{im}\psi^{*} = (\ker\psi)^{0}$, the annihilator of $\ker\psi$.
\end{proof}

\begin{proof}[Proof of Theorem~\ref{thm:main-priors}]
Parts~(1) and~(2) are Proposition~\ref{prop:coherence}(1) and~(2). Part~(3)
follows from Proposition~\ref{prop:ridge} and
Proposition~\ref{prop:ridgecoherence}(1) and~(2).
\end{proof}

\begin{remark}\label{rem:standardise}
The ridge prior carries a choice of geometry on the covariates, and by
Proposition~\ref{prop:ridgecoherence}(2) this choice is not preserved by the
recodings. Thus, in terms of the classes of Section~\ref{sec:meaningful}, the
prior is coherent over the class of Tjur and not over the class of McCullagh.
This is the categorical form of two standard facts. The first one is that
ridge regression is applied to standardised covariates, as in the original
paper, where the model is formulated so that
$X^{\!\top}X$ is in correlation form, and ill-conditioning is measured by its
distance from the identity in that
form~\cite[pp.~55--56]{hoerl_ridge_1970}. The second one is that
$\widehat\beta_{\lambda}$, unlike the least squares estimator, is not
equivariant under the linear reparametrisations of the covariate space. Note
that the model $\mathbf{P}$ is natural over $\catO$
(Lemma~\ref{lem:linreg-model}), so that the failure of naturality comes from
the prior only.
\end{remark}

\appendix
\section{The category of statistical models and the arrows of
\texorpdfstring{$\Stoch$}{Stoch}}\label{sec:Bref}

In this appendix we compare the category $\StatMod$ of Br\o ns
(Definition~\ref{def:statmod}) with the category whose objects are the
morphisms of $\Stoch$ and whose morphisms are the commuting squares with
deterministic sides.

\subsection{Copowers and squares}

The monoidal unit $I$ of $\Stoch$ is the one-point measurable space, and
\[
\Stoch(I,S) \;=\; \Meas(1,\G S) \;=\; \G S \;=\; \Prb(S)
\]
naturally in $S$, i.e.\ $\Prb \cong \Stoch(I,-)$. The Dirac embedding $\del$
(Definition~\ref{def:dirac}) is the free functor of the Kleisli adjunction
$\Meas \rightleftarrows \Stoch$, hence it is a left adjoint, and it preserves
colimits. The coproducts in $\Meas$ are the disjoint unions, with the
disjoint-union $\sigma$-algebra. Hence for every set $A$ the discrete space
$(A,2^{A})$ is the copower of $I$ in $\Stoch$, i.e.
\[
\del\bigl(A,2^A \bigr) \;\cong\; \coprod\nolimits_{A} I \;=:\; A\cdot I .
\]
Therefore, for every set $A$ and every measurable space $S$,
\begin{equation}\label{eq:copower}
\Set\bigl(A,\Prb(S)\bigr)
\;\cong\;
\Set\bigl(A,\Stoch(I,S)\bigr)
\;\cong\;
\Stoch\bigl(A\cdot I,\;S\bigr),
\end{equation}
which is the instance of~\eqref{eq:kleisli} used in the proof of
Theorem~\ref{thm:bridge-finite}. For $p \colon A \to \Prb(S)$ we denote by
$\widehat{p} \colon A \cdot I \rightsquigarrow S$ the corresponding kernel
$\widehat{p}(a)(B) = p(a)(B)$, which is measurable in $a$ because $A \cdot I$ is
discrete. For a kernel $q \colon A \cdot I \rightsquigarrow S$ we denote by
$\check{q} \colon a \mapsto q(a)$ the corresponding map.

\begin{definition}\label{def:arrdet}
Let $q \colon X \rightsquigarrow Y$ and $q' \colon X' \rightsquigarrow Y'$ be
morphisms of $\Stoch$. A \emph{square} from $q$ to $q'$ is a pair $(a,b)$ of
deterministic morphisms $a \colon X \rightsquigarrow X'$ and
$b \colon Y \rightsquigarrow Y'$, called the \emph{legs} of the square, such
that $q' \circ a = b \circ q$. The squares compose componentwise. We denote by
$\Arrdet{\Stoch}$ the category whose objects are the morphisms of $\Stoch$ and
whose morphisms are the squares, and by $\Arrcop{\Stoch}$ its full subcategory
of the objects $q \colon X \rightsquigarrow Y$ such that $X$ is isomorphic in
$\Stoch$ to a copower $A \cdot I$.
\end{definition}

\begin{lemma}\label{lem:square-pointwise}
Let $f \colon X \to X'$ and $g \colon Y \to Y'$ be measurable maps. Then
$(\del f, \del g)$ is a square from $q$ to $q'$ if and only if
$q'\bigl(f(x)\bigr) = g_{*}\, q(x)$ in $\Prb(Y')$ for every $x \in X$.
\end{lemma}

\begin{proof}
We have $(q' \circ \del f)(x) = q'(f(x))$ by~\eqref{eq:pull} and
$(\del g \circ q)(x) = g_{*}\, q(x)$ by~\eqref{eq:push}.
\end{proof}

We shall use the following two conditions.
\begin{itemize}
\item[(a)] For a set $A$: every countably additive $\{0,1\}$-valued
probability measure on $(A,2^{A})$ is a Dirac measure.
\item[(b)] For a measurable space $Z$: the $\sigma$-algebra $\Sigma_{Z}$ is
countably generated and separates points. By Lemma~\ref{lem:dirac}(2) this
implies that for every $W$ the deterministic morphisms $W \rightsquigarrow Z$
are the Dirac kernels of the measurable maps.
\end{itemize}

\begin{lemma}[Ulam]\label{lem:app-ulam}
We call a countably additive $\{0,1\}$-valued probability measure on
$(A, 2^{A})$ which is not a Dirac measure a \emph{two-valued measure} on the
set $A$.
\begin{enumerate}
\item The two-valued measures on $A$ correspond bijectively to the countably
complete non-principal ultrafilters on
$A$~\cite[(10.3)--(10.4)]{jech_set_2003}.
\item If $A$ carries no two-valued measure, then neither does any set of
cardinality at most $|A|$, nor the power set
$2^{A}$~\cite[Bem.~2 and Satz~1]{ulam_masstheorie_1930}. In particular, no set
of cardinality at most that of the continuum carries one.
\item The set $A$ carries a two-valued measure if and only if there is a
measurable cardinal
$\kappa \le |A|$~\cite[Lemma~10.2 and Def.~10.3]{jech_set_2003}. The original
form of Ulam is that $A$ carries none when no weakly inaccessible cardinal is
$\le |A|$~\cite[Satz~(A)]{ulam_masstheorie_1930},
cf.~\cite[Thm.~10.1]{jech_set_2003}.
\end{enumerate}
\end{lemma}

Thus condition~(a) says that no measurable cardinal is $\le |A|$. By
Lemma~\ref{lem:app-ulam}(2), it holds, with no set-theoretic hypothesis, when
$A$ has at most the cardinality of the continuum, which is the case for all the
parameter sets of this paper.

\subsection{The comparison functor}

\begin{proposition}\label{wprop:copower}
The bijection~\eqref{eq:copower} extends to a functor
\[
\Jc \colon \StatMod \longrightarrow \Arrdet{\Stoch},
\qquad
(\Theta_0, S, p) \longmapsto
\bigl(\widehat{p} \colon \Theta_0 \cdot I \rightsquigarrow S\bigr),
\qquad
(r,f) \longmapsto (\del r,\del f),
\]
with the following properties.
\begin{enumerate}
\item $\Jc$ is injective on the objects, and it gives a bijection between the
objects of $\StatMod$ and the arrows of $\Stoch$ whose domain is a copower
$A \cdot I$. Hence $\Jc$ is essentially surjective onto $\Arrcop{\Stoch}$.
\item Let $x = (\Theta_0,S,p)$ and $x' = (\Theta_0',S',p')$ be objects of
$\StatMod$. If $\Theta_0'$ satisfies~(a) and $S'$ satisfies~(b), then the map
$\StatMod(x,x') \to \Arrdet{\Stoch}(\Jc x,\Jc x')$ induced by $\Jc$ is a
bijection.
\item Let $\StatMod_{(a),(b)} \subseteq \StatMod$ be the full subcategory of
the objects $(\Theta_0,S,p)$ such that $\Theta_0$ satisfies~(a) and $S$
satisfies~(b). Then $\Jc$ restricts to an equivalence between
$\StatMod_{(a),(b)}$ and the full subcategory of $\Arrcop{\Stoch}$ of the
arrows $X \rightsquigarrow S$ such that $S$ satisfies~(b) and
$X \cong A\cdot I$ with $A$ satisfying~(a). Let $\StatMod_{\mathrm{fin}}$ be the
full subcategory of the objects with $\Theta_0$ finite and $S$ finite and
discrete. Then $\Jc$ gives an isomorphism of categories
$\StatMod_{\mathrm{fin}} \cong \Arrdet{\FinStoch}$.
\end{enumerate}
\end{proposition}

\begin{proof}
Let $(r,f) \colon (\Theta_0,S,p) \to (\Theta_0',S',p')$ be a morphism of
$\StatMod$, i.e.\ $\Prb(f) \circ p = p' \circ r$
(Definition~\ref{def:statmod}). The map $r$ is measurable because its domain is
discrete, hence $\del r$ and $\del f$ are Dirac kernels, and they are
deterministic. By Lemma~\ref{lem:square-pointwise} the pair
$(\del r,\del f)$ is a square from $\widehat p$ to $\widehat{p'}$ if and only
if $\widehat{p'}(r(\theta)) = f_{*}\,\widehat p(\theta)$ for every $\theta$,
i.e.\ if and only if $p' \circ r = \Prb(f) \circ p$. Since $\del$ is a functor
and the squares compose componentwise, $\Jc$ is a functor.

(1) The assignment $(\Theta_0,S,p) \mapsto \widehat p$ is a bijection from the
objects of $\StatMod$ onto the arrows of $\Stoch$ whose domain is of the form
$A \cdot I = (A,2^{A})$, with inverse $q \mapsto (A,S,\check q)$. Indeed, the
domain of $\widehat p$ gives $\Theta_0$ as its underlying set, the codomain
gives $S$, and $p = \check{\widehat p}$. Hence $\Jc$ is injective on the
objects and its image is contained in $\Arrcop{\Stoch}$. Let now
$q \colon X \rightsquigarrow S$ be an object of $\Arrcop{\Stoch}$, and let
$v \colon A \cdot I \rightsquigarrow X$ and
$u \colon X \rightsquigarrow A \cdot I$ be mutually inverse. We show that they
are deterministic. For $a \in A$ let
$X_a \coloneqq \{x \in X : u(x)(\{a\}) = 1\}$, which is measurable. Since
$u \circ v = \id_{A \cdot I}$ we have
\[
1 \;=\; \delta_a(\{a\}) \;=\; (u \circ v)(a)(\{a\})
\;=\; \int_X u(x)(\{a\})\, v(a)(\mathrm{d}x),
\]
and the integrand is bounded by $1$, hence $v(a)(X_a) = 1$. In particular
$X_a \neq \emptyset$. For $x \in X_a$ the probability measure $u(x)$ on $2^{A}$
gives mass $1$ to $\{a\}$, i.e.\ $u(x) = \delta_a$. Since $v \circ u = \id_X$
we obtain
\[
\delta_x(F) \;=\; (v \circ u)(x)(F)
\;=\; \int_A v(b)(F)\, u(x)(\mathrm{d}b) \;=\; v(a)(F)
\qquad (F \in \Sigma_X).
\]
Thus $v(a) = \delta_x$ for every $x \in X_a$. Choosing $x_a \in X_a$ for each
$a$ we obtain a map $f \colon A \to X$, which is measurable because its domain
is discrete, such that $v = \del f$. Hence $v$ is deterministic, and $u$ is
deterministic by~\cite[Lem.~10.9(c)]{fritz_synthetic_2020}. Now
$q \circ v \colon A \cdot I \rightsquigarrow S$ has domain $A \cdot I$, so that
$q \circ v = \Jc\bigl(A, S, (q \circ v)^{\vee}\bigr)$. Furthermore
$(v, \id_S) \colon q \circ v \to q$ and $(u, \id_S) \colon q \to q \circ v$ are
squares, and they are mutually inverse in $\Arrdet{\Stoch}$. Hence
$q \cong \Jc\bigl(A,S,(q \circ v)^{\vee}\bigr)$.

(2) Let $(a,b)$ be a square from $\widehat p$ to $\widehat{p'}$. The leg
$a \colon \Theta_0 \cdot I \rightsquigarrow \Theta_0' \cdot I$ is
deterministic, hence each $a(\theta)$ is a countably additive
$\{0,1\}$-valued probability measure on
$2^{\Theta_0'}$~\cite[Ex.~10.4]{fritz_synthetic_2020}. By~(a) it is
$\delta_{r(\theta)}$ for a unique $r(\theta) \in \Theta_0'$, and $r$ is
measurable because $\Theta_0$ is discrete, so that $a = \del r$. By~(b) for
$S'$ we have $b = \del f$ for a unique measurable map $f \colon S \to S'$. By
Lemma~\ref{lem:square-pointwise} we have $p' \circ r = \Prb(f) \circ p$, i.e.\
$(r,f)$ is a morphism of $\StatMod$ with $\Jc(r,f) = (a,b)$. Hence the map
$\StatMod(x,x') \to \Arrdet{\Stoch}(\Jc x, \Jc x')$ is surjective. It is
injective by Lemma~\ref{lem:dirac}(1), since $2^{\Theta_0'}$ separates points
and $\Sigma_{S'}$ does by~(b).

(3) Let $\mathcal A$ be the full subcategory of $\Arrcop{\Stoch}$ of the
statement. The functor $\Jc$ restricts to a functor
$\StatMod_{(a),(b)} \to \mathcal A$, which is fully faithful by part~(2). It is
essentially surjective by the argument of part~(1). Indeed, for
$q \colon X \rightsquigarrow S$ in $\mathcal A$ with
$v \colon A \cdot I \rightsquigarrow X$ an isomorphism and $A$
satisfying~(a), the object $\bigl(A,S,(q \circ v)^{\vee}\bigr)$ belongs to
$\StatMod_{(a),(b)}$ and
$q \cong \Jc\bigl(A,S,(q \circ v)^{\vee}\bigr)$. Hence it is an
equivalence~\cite[IV.4, Thm.~1(iii)]{maclane_categories_1998}. Finally every
object of $\FinStoch$ is $A \cdot I$ for a finite set $A$, so that by
part~(1) $\Jc$ gives a bijection from the objects of $\StatMod_{\mathrm{fin}}$
onto the morphisms of $\FinStoch$, i.e.\ onto the objects of
$\Arrdet{\FinStoch}$. Conditions~(a) and~(b) hold for every finite set and
every finite discrete space, hence $\Jc$ is bijective on the hom-sets by
part~(2), and a functor which is bijective on the objects and on the hom-sets
is an isomorphism of categories.
\end{proof}

\subsection{The measurable analogue}

The restriction to the arrows with copower domain comes from the
$\Set$-valued formulation. Consider the comma category
\[
\Statk \;\coloneqq\; (\id_{\Meas}\downarrow\G).
\]
Its objects are the measurable maps $\Theta \to \G S$, i.e.\ by~\eqref{eq:kleisli} the morphisms
$\Theta \rightsquigarrow S$ of $\Stoch$, with no condition on the domain. We
denote by $D \colon \Set \to \Meas$ the functor which endows a set with the
$\sigma$-algebra of all its subsets and by $U \colon \Meas \to \Set$ the
underlying-set functor.

\begin{proposition}\label{prop:coreflection}
We have $D \dashv U$ and $U\G = \Prb$. Let
$\varepsilon \colon DU \Rightarrow \id_{\Meas}$ be the counit, whose component
$\varepsilon_X$ is the identity map of $UX$, and let
$\lambda \coloneqq \varepsilon\G \colon D\Prb \Rightarrow \G$.
\begin{enumerate}
\item The assignments
$(\Theta_0, S, p) \mapsto (D\Theta_0,\, S,\, \lambda_S \circ Dp)$ and
$(r,f) \mapsto (Dr, f)$ give a functor $\Phi \colon \StatMod \to \Statk$.
\item The assignments $(\Theta, S, q) \mapsto (U\Theta,\, S,\, Uq)$ and
$(a,f) \mapsto (Ua, f)$ give a functor $\Psi \colon \Statk \to \StatMod$, and
$\Phi \dashv \Psi$.
\item We have $\Psi\Phi = \id_{\StatMod}$. Hence $\Phi$ is fully faithful and
$\StatMod$ is a coreflective subcategory of $\Statk$ with coreflector $\Psi$.
\item The counit of $\Phi \dashv \Psi$ at $(\Theta, S, q)$ is
$(\varepsilon_{\Theta}, \id_S)$, which is invertible if and only if
$\Sigma_{\Theta} = 2^{\Theta}$. The essential image of $\Phi$ is the full
subcategory of $\Statk$ of the objects whose parameter space is discrete.
\end{enumerate}
No hypothesis on the spaces is needed.
\end{proposition}

\begin{proof}
We have $\Meas(DA, X) = \Set(A, UX)$ naturally, because every map from a
discrete space is measurable, hence $D \dashv U$. The underlying set of $\G S$
is $\Prb(S)$ and the underlying map of $\G(f)$ is the pushforward, hence
$U\G = \Prb$. The transformation $\lambda$ is natural, and each $\lambda_S$ is
the identity on the underlying sets.

(1) The map $\lambda_S \circ Dp$ is measurable, hence it is an object of
$\Statk$. Let $(r,f)$ be a morphism of $\StatMod$, i.e.\
$\Prb(f)\circ p = p' \circ r$. By the naturality of $\lambda$ we have
\[
\G(f) \circ \lambda_S \circ Dp
= \lambda_{S'} \circ D\bigl(\Prb(f)\bigr) \circ Dp
= \lambda_{S'} \circ D\bigl(p' \circ r\bigr)
= \bigl(\lambda_{S'} \circ Dp'\bigr) \circ Dr ,
\]
i.e.\ $(Dr,f)$ is a morphism of $\Statk$. Identities and compositions are
preserved componentwise.

(2) The functor $\Psi$ is well defined since
$Uq \colon U\Theta \to U\G S = \Prb(S)$. A morphism
$\Phi(\Theta_0,S,p) \to (\Theta',T,q)$ is a pair $(a,f)$, where
$a \colon D\Theta_0 \to \Theta'$ is measurable and
$q \circ a = \G(f) \circ \lambda_S \circ Dp$. Let
$r = Ua \colon \Theta_0 \to U\Theta'$ be the map which corresponds to $a$ under
$D \dashv U$. Both sides of the equality are measurable maps from the discrete
space $D\Theta_0$ and $U$ is faithful, hence the equality is equivalent to
$Uq \circ r = \Prb(f) \circ p$, i.e.\ to the fact that $(r,f)$ is a morphism
$(\Theta_0,S,p) \to \Psi(\Theta',T,q)$ of $\StatMod$. This correspondence is
bijective and natural in both variables.

(3) We have $UD = \id_{\Set}$ and $U\lambda_S = \id$, hence $\Psi\Phi$ is the
identity on the objects and on the morphisms. A left adjoint whose unit is
invertible is fully faithful, and its domain is then coreflective in its
codomain.

(4) We have
$\Phi\Psi(\Theta,S,q) = (DU\Theta,\, S,\, \lambda_S \circ D(Uq))$, and
$(\varepsilon_\Theta, \id_S)$ is a morphism from it to $(\Theta,S,q)$, since
both sides of the required equality are measurable maps $DU\Theta \to \G S$
with underlying map $Uq$. The triangle identities show it is the counit. The
bijection $\varepsilon_\Theta$ is an isomorphism of $\Meas$ if and only if
$\Sigma_\Theta = 2^{\Theta}$, and $X \cong DA$ in $\Meas$ implies
$\Sigma_X = 2^X$. This gives the essential image.
\end{proof}

\begin{remark}\label{rem:Jcm}
A morphism $(a,f)$ of $\Statk$ is a pair of measurable maps, and
$(a,f) \mapsto (\del a, \del f)$ gives a functor
$\Jcm \colon \Statk \to \Arrdet{\Stoch}$ by
Lemma~\ref{lem:square-pointwise}. It is bijective on the objects
by~\eqref{eq:kleisli}. By Lemma~\ref{lem:dirac}(1) and~(2) it is bijective on
the morphisms into an object $(\Theta',S',q')$ when $\Theta'$ and $S'$
satisfy~(b). Hence $\Jcm$ restricts to an equivalence between the full
subcategory of $\Statk$ of the objects whose parameter space and sample space
satisfy~(b) and the full subcategory of $\Arrdet{\Stoch}$ of the arrows between
such spaces. We have $\Jc \cong \Jcm\Phi$, where the isomorphism is the
identification $D\Theta_0 \cong \Theta_0 \cdot I$ in $\Stoch$. The functor
$\Phi$ needs no hypothesis, while conditions~(a) and~(b) are used only for
$\Jcm$ and $\Jc$. We do not know whether they are necessary. Note that there is
no equivalence with the full arrow category $\Stoch^{\rightarrow}$, whose
squares have arbitrary kernels as legs, because the squares which come from
pairs of measurable maps are deterministic.
\end{remark}

\begin{corollary}\label{cor:coreflection-diagrams}
Let $\cat{C}$ be a small category. The composition with $\Phi$ and $\Psi$ gives
functors $\Phi_{*}$ and $\Psi_{*}$ between the categories of diagrams
$[\cat{C}, \StatMod]$ and $[\cat{C}, \Statk]$ such that
$\Phi_{*} \dashv \Psi_{*}$ and $\Psi_{*}\Phi_{*} = \id$. For
$\cat{C} = \catV \times \catD^{\op}$, the bijection of
Proposition~\ref{prop:brons} gives a bijection between the McCullagh--Br\o ns
models and the functors $\catV \times \catD^{\op} \to \Statk$ whose two
projections are $D \circ L$ and
$\Gamma \circ (\id_{\catV} \times \pr_\catU^{\op})$. A diagram in $\Statk$
belongs to the essential image of $\Phi_{*}$ if and only if its parameter
spaces are discrete.
\end{corollary}

\begin{proof}
The composition with an adjunction is an adjunction between the categories of
diagrams, with unit and counit computed pointwise, and
$\Psi_{*}\Phi_{*} = (\Psi\Phi)_{*} = \id$. A functor
$F \colon \catV \times \catD^{\op} \to \Statk$ with the stated projections has
parameter spaces of the form $D\Theta_0$, hence $F = \Phi_{*}\Psi_{*}F$, since
both sides have the same underlying maps and every map from a discrete space
is measurable. Therefore $F \mapsto \Psi_{*}F$ and $G \mapsto \Phi_{*}G$ are
mutually inverse bijections between these functors and the functors
$\catV \times \catD^{\op} \to \StatMod$ of Proposition~\ref{prop:brons}(2). The
last statement follows from Proposition~\ref{prop:coreflection}(4), since the
counit is computed pointwise and discreteness is invariant under the
isomorphisms of $\Meas$.
\end{proof}

\section*{Declarations}

\noindent\textbf{Funding.} The author received no funding for this work.

\noindent\textbf{Competing interests.} The author declares no competing
interests, financial or non-financial.

\noindent\textbf{Data availability.} No datasets were generated or analysed.
The scripts which carry out the computations of Sections~\ref{sec:hochschild}
and~\ref{sec:bayes} are available, with their output, at
\url{https://github.com/Vaxx66/calamus}, release \texttt{v37}.

\noindent\textbf{Use of AI tools.} In preparing this article the author used
the large language model Claude Fable (Anthropic) to draft and revise parts
of the text, to check the proofs and the citations against their sources, and
to write the Python scripts which carry out, in exact arithmetic, the
computations of Sections~\ref{sec:hochschild} and~\ref{sec:bayes}. The author
verified every statement, proof and reference, and is solely responsible for
the content.

\bibliographystyle{plain}
\bibliography{calamus_arkiv_draft_light_v05}

\end{document}